\documentclass[12pt]{amsart}

\usepackage[T1]{fontenc}
\usepackage[utf8]{inputenc}

\usepackage{amsmath,amssymb,amsthm,mathtools}
\usepackage{mathrsfs}

\usepackage{libertine}
\usepackage[libertine]{newtxmath}

\usepackage{microtype}

\usepackage{tikz}
\usetikzlibrary{positioning,fit,backgrounds,arrows.meta,calc}
\newcommand{\figthm}[1]{\hyperref[#1]{Thm.~\ref*{#1}}}
\newcommand{\figprop}[1]{\hyperref[#1]{Prop.~\ref*{#1}}}
\newcommand{\figlem}[1]{\hyperref[#1]{Lem.~\ref*{#1}}}
\newcommand{\figcor}[1]{\hyperref[#1]{Cor.~\ref*{#1}}}

\usepackage[
  a4paper,
  left=30mm,
  right=30mm,
  top=30mm,
  bottom=26mm
]{geometry}

\usepackage{booktabs}
\usepackage{enumitem}

\setlist[itemize]{
  leftmargin=*,
  topsep=4pt,
  itemsep=2pt,
  parsep=0pt
}

\setlist[enumerate]{
  leftmargin=*,
  topsep=4pt,
  itemsep=2pt,
  parsep=0pt
}

\usepackage{xcolor}

\definecolor{linkblue}{RGB}{38,68,105}
\definecolor{citegreen}{RGB}{45,88,68}

\usepackage[
  colorlinks=true,
  linkcolor=linkblue,
  citecolor=citegreen,
  urlcolor=linkblue,
  pdfborder={0 0 0}
]{hyperref}

\usepackage[nameinlink,noabbrev]{cleveref}

\numberwithin{equation}{section}

\theoremstyle{plain}

\newtheorem{theorem}{Theorem}[section]
\newtheorem{proposition}[theorem]{Proposition}
\newtheorem{lemma}[theorem]{Lemma}
\newtheorem{corollary}[theorem]{Corollary}

\theoremstyle{definition}

\newtheorem{definition}[theorem]{Definition}

\newtheorem{example}[theorem]{Example}

\theoremstyle{remark}
\newtheorem{remark}[theorem]{Remark}

\theoremstyle{plain}
\newtheorem{maintheorem}{Theorem}

\crefname{maintheorem}{theorem}{theorems}
\Crefname{maintheorem}{Theorem}{Theorems}

\AddToHook{env/proposition/begin}{\crefalias{theorem}{proposition}}
\AddToHook{env/lemma/begin}{\crefalias{theorem}{lemma}}
\AddToHook{env/corollary/begin}{\crefalias{theorem}{corollary}}
\AddToHook{env/claim/begin}{\crefalias{theorem}{claim}}
\AddToHook{env/definition/begin}{\crefalias{theorem}{definition}}
\AddToHook{env/convention/begin}{\crefalias{theorem}{convention}}
\AddToHook{env/example/begin}{\crefalias{theorem}{example}}
\AddToHook{env/remark/begin}{\crefalias{theorem}{remark}}

\allowdisplaybreaks

\DeclareMathOperator{\PL}{PL}

\DeclareMathOperator{\Res}{Res}

\DeclareMathOperator{\ord}{ord}

\DeclareMathOperator{\DRes}{DRes}
\DeclareMathOperator{\sgn}{sgn}

\newcommand{\C}{\mathbb C}
\newcommand{\Q}{\mathbb Q}
\newcommand{\R}{\mathbb R}
\newcommand{\Z}{\mathbb Z}

\newcommand{\D}{\mathscr D}

\newcommand{\Lcal}{\mathscr L}

\newcommand{\one}{\mathbf 1}

\newcommand{\Lef}{\mathbb L}
\newcommand{\Mmot}{\mathcal M_{\mathbb C}}
\newcommand{\motzeta}{Z^{\mathrm{mot}}_{F,0}}

\title[The Multivariable Strong Monodromy Conjecture for Plane Curves]{The Multivariable Strong Monodromy Conjecture for \\ Plane Curves}

\author{Sheng Tan}

\address{School of Mathematical Sciences, Capital Normal University, Beijing 100048, China}
\email{tansheng2018@outlook.com}

\date{\today. }

\begin{document}

\begin{abstract}
We prove the multivariable Strong Monodromy Conjecture for the local topological zeta function of arbitrary tuples of nonzero nonunit plane curve germs. The entries may be nonreduced and may have common irreducible factors. We also prove that the non-equivariant local motivic and topological zeta functions have the same polar locus and generic pole orders, using a multivariable extension of the Nicaise--Xu definition of motivic pole order. Our arguments combine twisted Poincar\'e residues with resolution formulas and specialization to one variable. In arbitrary dimension, we prove that a nonzero iterated residue class on a stratum of a log resolution implies the vanishing of the Bernstein--Sato ideal at the corresponding parameter. In particular, every topological polar hyperplane whose order equals the ambient dimension is contained in this zero locus. For algebraic germs, the topological pole order is at most the motivic pole order, and the two attain maximal order simultaneously. 
\end{abstract}

\maketitle

\enlargethispage{2pt}
\tableofcontents


\section{Introduction}\label{sec:introduction}

\makeatletter
\let\introduction@tocwrite\@tocwrite
\def\@tocwrite#1#2{}
\makeatother

\subsection{Bernstein--Sato polynomial}

For a nonzero holomorphic germ $f$, the Bernstein--Sato functional equation has the form
\[
b(s)f^s=P(s)f^{s+1},
\]
where $f^s$ is a formal symbol, $b(s)\in\C[s]\setminus\{0\}$, and $P(s)$ is a differential operator with holomorphic coefficients depending polynomially on $s$. The Bernstein--Sato polynomial $b_f(s)$ is the monic polynomial of least degree for which such an equation holds. Bernstein proved the existence of such an equation for polynomials, and Kashiwara proved that the roots of $b_f(s)$ are negative rational numbers for holomorphic germs \cite{Bernstein1972,Kashiwara1976}. Malgrange and Kashiwara related the exponentials of the roots of $b_f$ to eigenvalues of Milnor monodromy \cite{Malgrange1975,Kashiwara1976}.

The poles of local zeta functions can be studied through resolution formulas. Igusa's $p$-adic zeta functions and their resolution formulas initiated the arithmetic theory \cite{Igusa2000,Denef1991}, while Denef and Loeser introduced the topological zeta function and proved its independence of the chosen resolution \cite{DenefLoeser1992}. For a resolution divisor $E$ with numerical data $(N_E,\nu_E)$, the ratio $-\nu_E/N_E$ is a possible pole. The Strong Monodromy Conjecture predicts that every actual pole of the topological zeta function is a root of $b_f$:
\[
\PL\bigl(Z_f^{\mathrm{top}}\bigr)\subseteq Z(b_f).
\]
Here $\PL$ denotes the polar locus of the rational function after cancellation of common factors, and $Z(b_f)$ is the zero set of $b_f$. The converse inclusion is not expected. Even for plane curves, the roots of the Bernstein--Sato polynomial can vary within a fixed topological type, whereas the topological zeta function is determined by the resolution data \cite[Example~1]{Oaku2022}; see also the construction of topological roots from resolution data in \cite{Blanco2026}.

\subsection{The multivariable problem}

Let $(X,0)$ be a smooth complex surface germ, let $F=(f_1,\ldots,f_r)$ be a tuple of nonzero nonunit germs vanishing at the origin, and put $h=f_1\cdots f_r$. The multivariable Bernstein--Sato ideal is
\[
B_{F,0}=\left\{b(s)\in\C[s_1,\ldots,s_r]\ \middle|\ b(s)F^s=P(s)hF^s\text{ for some }P(s)\in\D_{X,0}[s]\right\}.
\]
Here $\D_{X,0}$ is the ring of holomorphic differential operators at the origin, $s=(s_1,\ldots,s_r)$, and $F^s=f_1^{s_1}\cdots f_r^{s_r}$ is a formal symbol. Functional equations for several analytic functions were developed by Sabbah and Gyoja, with later refinements using relative $\D$-modules by Maisonobe \cite{Sabbah1987,Gyoja1993,Maisonobe2023}. Local Bernstein--Sato ideals can also be computed effectively in examples \cite{UchaCastro2004,BahloulOaku2010}.

Budur introduced the multivariable topological zeta function and formulated the Topological Multivariable Strong Monodromy Conjecture \cite[Conjecture~1.17]{Budur2015}:
\[
\PL\bigl(Z^{\mathrm{top}}_{F,0}\bigr)\subseteq Z(B_{F,0}).
\]
We prove this conjecture for plane curves. The one-variable statement for $h$ gives information along the diagonal but does not determine the individual coefficients of the affine hyperplanes in $Z(B_{F,0})$.

Sabbah's specialization complex and Alexander modules describe multivariable monodromy characters \cite{Sabbah1990}. Budur--van der Veer--Wu--Zhou identify the exponential image of the Bernstein--Sato zero locus with the corresponding cohomology-support locus \cite{BvdVWZ2021}. Exponentiation identifies parameters that differ by integers in each coordinate, so monodromy characters do not distinguish all affine translates. The strong conjecture requires the inclusion of the actual affine hyperplane in the Bernstein--Sato zero locus, not only an inclusion of their exponential images. Budur--van der Veer--Van Werde obtain affine restrictions that imply the strict-transform inclusion used below \cite{BvdVVW2024}. Related approaches through rational powers, multivariate $V$-filtrations, and logarithmic nearby cycles appear in \cite{Saito2021RationalPowers,DavisYang2026,Wu2026}.

\subsection{Previous results}

The one-variable case for plane curves is well understood. Loeser proved that $-\nu_E/N_E$ is a root of the local Bernstein--Sato polynomial when $E$ is a rupture component, and obtained the corresponding result for strict transforms \cite[Th\'eor\`eme~III.3.1 and Remarque~III.3.3]{Loeser1988}. Veys characterized the actual poles of the topological zeta function on the minimal embedded resolution \cite{Veys1995}, and later related the cancellation of candidate poles to the log canonical model \cite{Veys1997}. Together, these results establish the one-variable Strong Monodromy Conjecture for plane curves. N\'emethi--Veys developed a broader framework in dimension two using zeta functions with differential forms \cite{NemethiVeys2012}.

The main one-variable results used in our proof are due to Blanco. His study of complex zeta functions relates the contributions of rupture components and their residues to roots of the Bernstein--Sato polynomial \cite{Blanco2019}. He also proved Yano's conjecture, which determines the generic $b$-exponents of irreducible plane curves \cite{Blanco2021}. Our proof uses two results from \cite{Blanco2026}: sharp bounds for the residue numbers at components adjacent to a rupture component, and a nonvanishing result for the corresponding twisted Poincar\'e residue class.

Several multivariable results are also known. Nicaise proved the weak monodromy conjecture in dimension two using multivariable zeta functions and Alexander modules \cite{Nicaise2004}. Budur obtained the weak plane curve case by specialization from one variable \cite[Theorem~1.15]{Budur2015}. The codimension-one faces of the polytopes of Cassou-Nogu\`es--Libgober determine affine hyperplanes contained in $Z(B_{F,0})$ \cite[Theorem~4.1]{CNL2011}; see also \cite{CNL2014}.

Hyperplane arrangements form another important class. Budur--Mustaţă--Teitler proved the weak monodromy conjecture and reduced the one-variable Strong Monodromy Conjecture for reduced arrangements to the $n/d$-conjecture \cite{BMT2011}. Walther subsequently proved the Strong Monodromy Conjecture for tame arrangements, including nonreduced arrangements and, in particular, all multi-arrangements in dimension three \cite[Corollary~5.15 and Remark~5.16]{Walther2017jacobian}. Wu proved Budur's multivariable conjecture for complete factorizations of central arrangements \cite{WuDiagonal2022}. Further results for broad classes of arrangements were obtained by Bath \cite{Bath2023} and, more recently, Davis--Yang \cite{DavisYang2026}. Chen--Zuo obtained related results on monodromy, strong monodromy, and Bernstein--Sato roots for determinantal varieties and homogeneous ideals \cite{ChenZuo2025Determinantal,ChenZuo2026Analytic,ChenZuo2026PDE}.

Melle-Hern\'andez--Torrelli--Veys studied the relation between poles of maximal order, multiplicities of roots of the Bernstein--Sato polynomial, and monodromy Jordan blocks \cite{MelleTorrelliVeys2009}. Nicaise--Xu proved a theorem on poles of maximal order of motivic zeta functions \cite{NicaiseXu2016}. Multivariable motivic versions of the Strong Monodromy Conjecture have also been formulated using other definitions of the motivic polar locus \cite[Conjecture~2.3(ii)]{BvdVWZSurvey2024}. Our motivic result concerns the non-equivariant zeta function, with pole order defined in \Cref{def:motivic-NX-order}.

\subsection{Main results}

We now state the main results of this paper. The first two results concern the multivariable Strong Monodromy Conjecture and the comparison of motivic and topological pole orders for plane curves. The remaining results use residue classes to obtain points and hyperplanes in the Bernstein--Sato zero locus in arbitrary dimension.

\begin{maintheorem}[{=\,\Cref{thm:main-thm-top}}]\label{thm:main-intro}
Let $F=(f_1,\ldots,f_r)$ be a tuple of nonzero nonunit plane curve germs on a smooth complex surface germ. Then
\[
\PL\bigl(Z^{\mathrm{top}}_{F,0}\bigr)\subseteq Z(B_{F,0}).
\]
\end{maintheorem}

This proves the Topological Multivariable Strong Monodromy Conjecture for arbitrary plane curve tuples.

The second result concerns the non-equivariant local motivic zeta function considered in \Cref{sec:motivic-comparison}. Its resolution formula uses the Grothendieck classes of the strata, without a $\widehat\mu$-action or monodromic covers. We extend the definition of motivic pole order of Nicaise--Xu \cite[Remark~3.7]{NicaiseXu2016} to rational affine hyperplanes in several variables and compare this order with the generic pole order of the topological zeta function. Throughout the paper, the term motivic zeta function refers to this non-equivariant version.

\begin{maintheorem}[{=\,\Cref{thm:plane-motivic-comparison}\,+\,\Cref{cor:plane-motivic-smc}}]\label{thm:motivic-intro}
Let $F=(f_1,\ldots,f_r)$ be a tuple of nonzero nonunit plane curve germs on a smooth complex surface germ. For every rational affine hyperplane $H\subseteq\C^r$,
\[
\ord_H^{\mathrm{mot}}Z^{\mathrm{mot}}_{F,0}=\ord_HZ^{\mathrm{top}}_{F,0}.
\]
Consequently,
\[
\PL_{\mathrm{mot}}(Z^{\mathrm{mot}}_{F,0})=\PL(Z^{\mathrm{top}}_{F,0})\subseteq Z(B_{F,0}).
\]
\end{maintheorem}

The equality concerns generic pole orders along affine hyperplanes. It does not compare these orders with scheme-theoretic multiplicities of components of the Bernstein--Sato locus.

The residue criterion also applies in arbitrary dimension. Let $(X,0)$ be a smooth complex germ of dimension $n$, let $F=(f_1,\ldots,f_r)$ be a tuple of nonzero nonunit holomorphic germs, and put $h=\prod_i f_i$. Let $\pi:Y\to X$ be a log resolution of the reduced support of $(h=0)$, and let $I=(a_1,\ldots,a_k)$ be a nonempty ordered tuple of distinct indices of components of the reduced total transform. For $\alpha\in\Q_{>0}^r$ satisfying $N_{a_j}\cdot\alpha=\nu_{a_j}$ for every $j$, the normal monodromies along these components are trivial. The iterated residue then has coefficients in a rank-one local system $\Lcal_{I,\alpha}$ on $D_I^\circ$. Let $\omega$ be a nowhere vanishing holomorphic top form near $0\in X$.

\begin{maintheorem}[{=\,\Cref{thm:iterated-residue-obstruction}}]\label{thm:iterated-intro}
If
\[
[\Res_I\pi^*(F^{-\alpha}\omega)]\ne0
\quad\text{in}\quad
H^{n-k}_{\mathrm{dR}}(D_I^\circ,\Lcal_{I,\alpha}),
\]
then
\[
F^{-\alpha}\notin\D_{X,0}(hF^{-\alpha})
\quad\text{and}\quad
-\alpha\in Z(B_{F,0}).
\]
\end{maintheorem}

For a pole of order $n$, the nonzero local residue at an intersection of $n$ components gives the following inclusion.

\begin{maintheorem}[{=\,\Cref{thm:maximal-order-hyperplanes}}]\label{thm:maximal-intro}
Let $\dim X=n$, and let $H$ be an irreducible component of the polar locus of $Z^{\mathrm{top}}_{F,0}$. If the generic pole order along $H$ is $n$, then
\[
H\subseteq Z(B_{F,0}).
\]
\end{maintheorem}

In the algebraic setting, let $c_H(\pi)$ denote the largest number of components defining $H$ that meet at a point of $\pi^{-1}(0)$, with value zero when no such component meets the fibre. Then
\begin{equation}\label{eq:higher-motivic-intro}
0\leq\ord_HZ^{\mathrm{top}}_{F,0}\leq\ord_H^{\mathrm{mot}}Z^{\mathrm{mot}}_{F,0}\leq c_H(\pi)\leq n.
\end{equation}
Moreover, $\ord_H^{\mathrm{mot}}Z^{\mathrm{mot}}_{F,0}=n$ if and only if $\ord_HZ^{\mathrm{top}}_{F,0}=n$. Below maximal order, the two orders need not agree. The examples in \Cref{sec:higher-obstructions} show that the two orders can differ in dimension three and that the two polar loci can differ in dimension five.

\subsection{Outline of the proofs}

\begin{figure}[t]
\centering

\definecolor{flowviolet}{RGB}{112,48,116}
\definecolor{flowvioletfill}{RGB}{250,246,250}
\definecolor{flowblue}{RGB}{63,137,165}
\definecolor{flowbluefill}{RGB}{245,250,252}
\definecolor{flowgold}{RGB}{205,121,22}
\definecolor{flowgoldfill}{RGB}{255,252,225}
\definecolor{flowpanel}{RGB}{247,247,247}

\begin{tikzpicture}[
    scale=0.92,
    transform shape,
    >=Latex,
    line/.style={->, draw=gray!70, line width=0.55pt},
    box/.style={
        rounded corners=2pt,
        draw=flowviolet,
        fill=flowvioletfill,
        line width=0.7pt,
        align=center,
        inner xsep=4pt,
        inner ysep=4pt,
        font=\scriptsize,
        text width=3.15cm
    },
    inputbox/.style={
        rounded corners=2pt,
        draw=flowblue,
        fill=flowbluefill,
        line width=0.7pt,
        align=center,
        inner xsep=4pt,
        inner ysep=4pt,
        font=\scriptsize,
        text width=3.0cm
    },
    mainbox/.style={
        rounded corners=3pt,
        draw=flowgold,
        fill=flowgoldfill,
        line width=1pt,
        align=center,
        inner xsep=6pt,
        inner ysep=6pt,
        font=\small,
        text width=6.2cm
    },
    paneltitle/.style={
        font=\small\bfseries,
        text=black!65,
        align=center
    }
]

\node[box] (laurent) at (0,0)
{\mbox{Laurent expansion}\\
(\figthm{thm:grouped-laurent})};

\node[box, below=5mm of laurent] (dichotomy)
{Pole dichotomy\\
(\figthm{thm:grouped-wall-dichotomy})};

\node[box, below left=8mm and 10mm of dichotomy] (crossing)
{Crossing case:\\
\mbox{two components} of $S_H$ meet\\
(\figthm{thm:same-wall-positivity})};

\node[box, below=8mm of dichotomy] (selected)
{No crossing:\\
choose $A\in S_H$ with $R_{A,H}\neq 0$\\
and classify $A$\\
(\figthm{thm:selected-component-classification})};

\node[box, below right=5mm and 10mm of dichotomy] (ruptureprep)
{Rupture case:\\
rational density, scalarization\\
(\figlem{lem:selected-rupture-density},
\figlem{lem:blanco-balance},
\figprop{prop:standard-residue-bounds})};

\node[box, below=6mm of crossing] (double)
{Double residue obstruction\\
(\figthm{thm:direct-double-residue-obstruction})};

\node[box, below = 6mm of selected] (strict)
{Strict-transform case\\
(\figthm{thm:strict-transform-wall})};

\node[box, below=6mm of ruptureprep] (blanco)
{Nonvanishing of residue class\\
(\mbox{\figprop{prop:selected-numerical-exhaustion},
\figthm{thm:blanco-residue-class},}
\figprop{prop:residue-one-injection})};

\node[box, below=6mm of blanco] (residue)
{Residue obstruction\\
(\figthm{thm:direct-residue-obstruction})};

\node[box, below=6mm of residue] (direct)
{Rational point detection\\
(\figprop{prop:affine-point-lifting})};

\node[box, below=6mm of double] (lift1) 
{Rational density\\
(\figlem{lem:bounded-polytope-density},
\figthm{thm:dense-wall-lifting})};

\node[box, below=6mm of direct] (lift3)
{Rational density\\
(\figthm{thm:dense-wall-lifting})};

\node[mainbox, below  = 26mm of strict, xshift=-21mm] (main)
{\textbf{\hyperref[thm:main-intro]{THEOREM A}}
$(=\hyperref[thm:main-thm-top]{\text{Thm.~\ref*{thm:main-thm-top}}})$\\[2pt]
$\displaystyle \PL\bigl(Z^{\mathrm{top}}_{F,0}\bigr)\subseteq Z(B_{F,0})$\\[2pt]
for arbitrary plane curve tuples};

\draw[line] (laurent) -- (dichotomy);
\draw[line] (dichotomy) -- (selected);
\draw[line] (crossing) -- (double);
\draw[line] (double) -- (lift1);
\draw[line] (selected) -- (strict);
\draw[line] (selected.east) -- (ruptureprep.west |- selected.east);
\draw[line] (ruptureprep) -- (blanco);
\draw[line] (blanco) -- (residue);
\draw[line] (residue) -- (direct);
\draw[line] (direct) -- (lift3);
\draw[line] (strict.south) -- (strict |- main.north);
\draw[line] (dichotomy.west) -| (crossing.north);
\draw[line] (lift1.south) -- (main.north -| lift1.south);
\draw[line] (lift3.west) -| (main.south);

\begin{scope}[on background layer]
\node[
    fill=flowpanel,
    rounded corners=6pt,
    fit=(laurent)(dichotomy)(crossing)(selected)(ruptureprep),
    inner sep=5mm
] (panelA) {};

\node[
    fill=flowpanel,
    rounded corners=6pt,
    fit=(double)(strict)(blanco)(residue)(direct),
    inner sep=5mm
] (panelB) {};

\node[
    fill=flowpanel,
    rounded corners=6pt,
    fit=(lift1)(lift3),
    inner sep=5mm
] (panelC) {};
\end{scope}

\end{tikzpicture}

\caption{Logical structure of the proof of \Cref{thm:main-intro}.}
\label{fig:plane-smc-proof}
\end{figure}
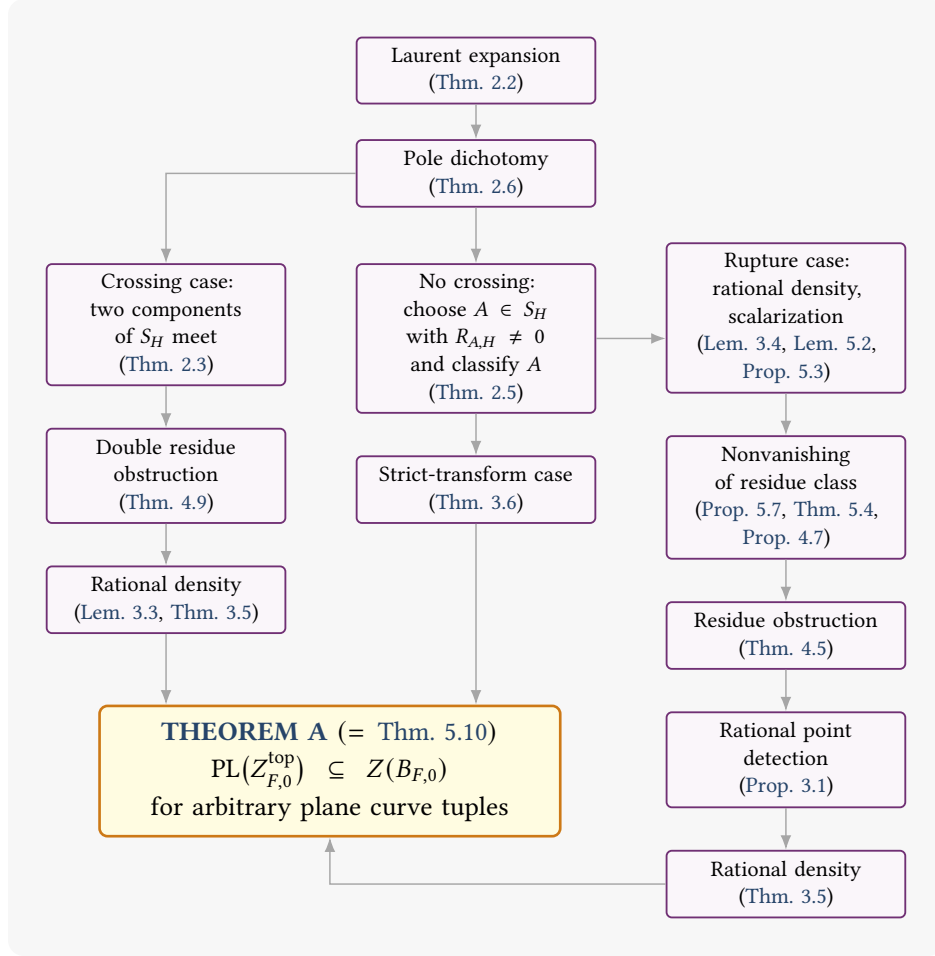

The logical structure of the proof of \Cref{thm:main-intro} is illustrated in \Cref{fig:plane-smc-proof}. We first group the irreducible components of the total transform according to the affine hyperplane they define. If two such components meet over the origin, their crossing gives a pole of order two and a nonzero double residue. Otherwise, an actual pole is simple and has a nonzero contribution from a strict transform or a rupture component. Strict transforms are treated by the known Bernstein--Sato inclusion. For rupture components, scalarization allows us to apply Blanco's residue bounds and nonvanishing theorem. In the crossing and rupture cases, the Green formula and tube integration prove $F^{-\alpha}\notin\D_{X,0}(hF^{-\alpha})$ at suitable rational parameters. Evaluation of the Bernstein--Sato equation at $s=-\alpha$, followed by rational density, gives the required hyperplane inclusion.

The motivic comparison uses cancellation identities for exceptional components of valency one or two, together with the analysis of chains of valency-two components. For each remaining hyperplane, we first fix a finite expression with a bound on the number of denominators defining it in each term. We then choose a positive integral specialization avoiding coincidences with denominators defining other hyperplanes. A motivic order smaller than the topological order would then contradict the one-variable plane curve comparison, obtained from Veys's pole classification and Hodge--Deligne realization.

In arbitrary dimension, the Green formula and iterated tube formula give the same implication from a nonzero residue class. At pole order $n$, the highest Laurent coefficient is a positive sum over zero-dimensional intersections of $n$ components. A nonzero local torus residue at one of these points gives the corresponding hyperplane inclusion in $Z(B_{F,0})$. Hodge--Deligne realization gives the general inequality between topological and motivic orders.

\subsection*{Organization of the paper}

\Cref{sec:setup} introduces the notation and studies the Laurent expansion along an affine hyperplane. \Cref{sec:cyclic-quotient} relates rational powers to the Bernstein--Sato ideal. \Cref{sec:residue-obstructions} establishes the residue obstructions on surfaces. \Cref{sec:main-theorem} is devoted to the proof of \Cref{thm:main-intro}. \Cref{sec:motivic-comparison} defines the non-equivariant motivic order and proves \Cref{thm:motivic-intro}. \Cref{sec:higher-obstructions} develops the iterated-residue obstruction in arbitrary dimension, and proves \Cref{thm:iterated-intro,thm:maximal-intro}. Finally, \Cref{sec:further-questions} summarizes the remaining higher-dimensional problems.

\subsubsection*{Acknowledgment}

The author is supported by NSFC (Grant No.~12601116), BJNSF (Grant No.~1264052) and a Beijing municipal research program for returned overseas scholars. He thanks Uli Walther for many helpful conversations on $D$-modules and related topics during his PhD studies at Purdue University. LLM tools were used for literature organization, language polishing, and auxiliary proof checking.

\makeatletter
\let\@tocwrite\introduction@tocwrite
\makeatother


\section{Preliminaries and polar hyperplanes}\label{sec:setup}

Let $(X,0)$ be a smooth complex surface germ, let $\mathcal O_{X,0}$ denote the local ring of holomorphic function germs at $0$, and let $\D_{X,0}$ denote the ring of holomorphic differential operators on $(X,0)$. Throughout,
\[
 F=(f_1,\ldots,f_r)
\]
is a tuple of arbitrary nonzero nonunit germs with $f_i(0)=0$. Put
\[
 h=f_1\cdots f_r.
\]
Let $s=(s_1,\ldots,s_r)$, and write $F^s=f_1^{s_1}\cdots f_r^{s_r}$ for the corresponding formal symbol.

\subsection{The Bernstein--Sato ideal} 

 On the free rank-one $\mathcal O_{X,0}[h^{-1},s_1,\ldots,s_r]$-module generated by the symbol $F^s$, vector fields act by
\[
\xi(gF^s)=\left(\xi(g)+g\sum_{i=1}^r s_i\frac{\xi(f_i)}{f_i}\right)F^s.
\]
The multivariable Bernstein--Sato ideal is the ideal
\begin{equation}\label{eq:bernstein-ideal}
B_{F,0}=\left\{b(s)\in\C[s_1,\ldots,s_r]\;\middle|\; b(s)F^s=P(s)hF^s \text{ for some }P(s)\in\D_{X,0}[s_1,\ldots,s_r]\right\}.
\end{equation}
When $r=1$, $B_{F,0}\subseteq\C[s]$ is principal and its monic generator is the usual Bernstein--Sato polynomial (aka $b$-function). Functional equations for several analytic functions were developed by Sabbah and Gyoja \cite{Sabbah1987,Gyoja1993}; for later structural results on multivariable Bernstein--Sato ideals and their zero loci, see also \cite{Maisonobe2023,Budur2015,BvdVWZ2021}.

\subsection{Resolution data and topological zeta function}

Let
\[
 \pi:Y\to  X
\]
be the minimal embedded log resolution of the reduced support of $(h=0)$. Minimality will be used later only when applying the sharp results for plane curves; all formulas in this section hold on any embedded log resolution. Write
\[
    D = \bigl(\pi^{-1}(h^{-1}(0))\bigr)_{\mathrm{red}} = \bigcup_{A\in J}A,
\]
where $J$ contains both exceptional components and strict transforms. For $A\in J$, set
\begin{gather}\label{eq:resolution-data}
    \begin{aligned}
        N_{A,i}=\ord_A(f_i\circ\pi), \quad N_A=(N_{A,1},\ldots,N_{A,r}),\\
        n_A=N_A\cdot\one = \sum_{i=1}^r N_{A,i}, \quad \nu_A=1+\ord_A(K_{Y/X}).
    \end{aligned}
\end{gather}
and write
\begin{equation}\label{eq:candidate-form}
 L_A(s)=N_A\cdot s+\nu_A.
\end{equation}
Here, $K_{Y/X}=K_Y-\pi^*K_X$ is the relative canonical divisor. For a strict transform $C$, one has $\nu_C=1$. If $A$ is exceptional over $0$, then $N_{A,i}>0$ for every $i$, because its divisorial valuation is centred at the maximal ideal. If $C$ is the strict transform of a reduced branch $g$, then
\[
 N_C=\bigl(\ord_g(f_1),\ldots,\ord_g(f_r)\bigr).
\]
Thus there is one strict-transform component for each branch of the reduced support, while arbitrary multiplicities and occurrences in several entries are recorded in its full order vector.

For $I\subseteq J$, define the normal crossing stratum
\[
    D_{I}^\circ = \left(\bigcap_{A\in I}A\right) \setminus \left(\bigcup_{B\notin I}B\right).
\]
The local multivariable topological zeta function is the fibre-localized version of the resolution expression in \cite[\S1.3]{Budur2015}, namely
\begin{equation}\label{eq:top-zeta}
    Z^{\mathrm{top}}_{F,0}(s) = \sum_{I\subseteq J} \chi\bigl( D_{I}^\circ\cap\pi^{-1}(0)\bigr) \prod_{A\in I}\frac{1}{L_A(s)},
\end{equation}
where $\chi$ denotes the topological Euler characteristic. Since $Y$ is a surface, a nonempty stratum contains at most two components. For $A,B\in {J}$, where $A$ and $B$ denote distinct irreducible components of the reduced total transform $D$, define
\begin{equation}\label{eq:local-fibre-coefficients}
    m_A = \chi\bigl(A^\circ\cap\pi^{-1}(0)\bigr), \quad m_{AB} = \chi\bigl(A\cap B\cap\pi^{-1}(0)\bigr) = \#\bigl(A\cap B\cap\pi^{-1}(0)\bigr).
\end{equation}
With unordered pairs in the second sum, \eqref{eq:top-zeta} becomes exactly
\begin{equation}\label{eq:surface-zeta}
 Z^{\mathrm{top}}_{F,0}(s)
 =
 \sum_{A\in J}\frac{m_A}{L_A(s)}
 +
 \sum_{\{A,B\}\subseteq J}
 \frac{m_{AB}}{L_A(s)L_B(s)}.
\end{equation}

The intersection with $\pi^{-1}(0)$ in \eqref{eq:local-fibre-coefficients} is essential. A crossing away from the local fibre does not contribute to the local zeta function. For the usual local dual graph, each drawn edge represents one crossing in the fibre and has $m_{AB}=1$. If $E$ is exceptional and $d_E$ denotes its valency in the total transform graph, then $E\simeq\mathbb P^1$ and 
\begin{equation}\label{eq:exceptional-euler}
 m_E=\chi(E^\circ)=2-d_E.
\end{equation}

If a strict transform $C$ meets an exceptional component $E$ over $0$, their intersection point is removed from $C^\circ$. Hence $m_C=0$, while $m_{CE}=1$. If the identity map is already a resolution, a single smooth branch through $0$ has singleton coefficient one. If two smooth branches meet transversely at $0$, both singleton coefficients vanish and the crossing coefficient is one. Thus \eqref{eq:surface-zeta} also applies when no blow-up is required.

We write the polar locus
\[
  \PL(Z^{\mathrm{top}}_{F,0})
  \subseteq \mathbb C^r
\]
for the reduced support of the polar divisor of the rational function \eqref{eq:surface-zeta}. Its irreducible components are the affine hyperplanes on which the reduced rational function $Z^{\mathrm{top}}_{F,0}$ actually has a pole. These need not coincide with all candidate hyperplanes $V(L_A)$ appearing in the resolution formula, since cancellations may occur. For any rational affine hyperplane $H$, let $v_H$ denote the corresponding divisorial valuation on $\C(s_1,\ldots,s_r)$ and write
\[
 \ord_H(R)=\max\{0,-v_H(R)\}
\]
for the nonnegative generic pole order of a rational function $R$ along $H$. Thus $\ord_H(R)=0$ when $R$ is regular at the generic point of $H$.

\subsection{Rational powers and residues}

For $\alpha\in\Q_{>0}^r$, put
\[
 u_\alpha=F^{-\alpha}.
\]
Viewed as a multivalued scalar function on $X\setminus(h=0)$, a local branch of $u_\alpha$ is horizontal for the branch connection
\[
 \nabla_\alpha^{\mathrm{act}}
 =d+\sum_{i=1}^r\alpha_i\frac{df_i}{f_i}.
\]
Let $\Lcal_\alpha^{\mathrm{act}}$ denote the local system of these actual analytic branches. Its monodromy has finite order, and its normal monodromy around a resolution component $A$ is
\[
 \exp\bigl(-2\pi iN_A\cdot\alpha\bigr).
\]

Fix a component $E$ and suppose
\[
 N_E\cdot\alpha=\nu_E.
\]
For every component $D$ adjacent to $E$, define
\begin{equation}\label{eq:epsilon}
 \epsilon_D(\alpha)=\nu_D-N_D\cdot\alpha.
\end{equation}
At $p=E\cap D$, take SNC coordinates $E=(x=0)$ and $D=(y=0)$ and a nowhere vanishing holomorphic two-form $\omega$ on $X$. The valuation of $\pi^*\omega$ along a component $A$ is $\nu_A-1$; hence, on a chosen analytic branch,
\begin{equation}\label{eq:local-coefficient-form}
 \pi^*(u_\alpha\omega) = U(x,y)x^{-1}y^{\epsilon_D(\alpha)-1}\,dx\wedge dy,
 \quad U(0,0)\ne0.
\end{equation}

Our Poincar\'e residue convention is
\begin{equation}\label{eq:residue-convention}
 \Res_{x=0}\left(\frac{dx}{x}\wedge\beta\right)=\beta.
\end{equation}
Here $\beta$ is a holomorphic one-form tangential to the divisor $(x=0)$. This is the usual logarithmic residue convention; see, for example, \cite[Chapitre~II]{Deligne1970}.

The scalar residue in \eqref{eq:local-coefficient-form} has actual analytic branches whose continuation around $E\cap D$ is multiplication by $\exp(2\pi i\epsilon_D)$. Let $\Lcal_{E,\alpha}^{\mathrm{act}}$ denote this actual-branch local system. The coefficient local system used below for differential forms and twisted de Rham cohomology is
\[
 \Lcal_{E,\alpha}
 :=\bigl(\Lcal_{E,\alpha}^{\mathrm{act}}\bigr)^\vee.
\]
Thus its parallel-transport monodromy around $E\cap D$ is $\exp(-2\pi i\epsilon_D)$. In a flat frame of $\Lcal_{E,\alpha}$, the displayed scalar residue branches are the local coordinates of an $\Lcal_{E,\alpha}$-valued form. Integration cycles carry coefficients in $\Lcal_{E,\alpha}^\vee=\Lcal_{E,\alpha}^{\mathrm{act}}$.

Finally, choose $q>0$ such that $b=q\alpha\in\Z_{>0}^r$ and define
\[
g_b=\prod_{i=1}^r f_i^{b_i},\quad M_A=\ord_A(g_b)=\sum_{i=1}^r b_iN_{A,i}=qN_A\cdot\alpha.
\]
Then $u_\alpha=F^{-\alpha}=g_b^{-1/q}$. Since $N_E\cdot\alpha=\nu_E$, we have $M_E=q\nu_E$, and hence the candidate exponent associated with $E$ and the chosen nowhere vanishing holomorphic two-form $\omega$ is
\[
-\frac{\nu_E}{M_E}=-\frac1q.
\]
For a component $D$ adjacent to $E$, the corresponding residue number is therefore
\[
M_D\left(-\frac1q\right)+\nu_D=-N_D\cdot\alpha+\nu_D=\epsilon_D(\alpha).
\]
Thus the residue numbers obtained from the scalarization $u_\alpha=g_b^{-1/q}$ agree exactly with the quantities $\epsilon_D(\alpha)$ defined above. The associated multivalued residue form agrees with the residue of \eqref{eq:local-coefficient-form}, up to the fixed nonzero normalization coming from the cyclic cover.


\subsection{Normalization}

Let $H\subseteq\C^r$ be an affine hyperplane of the form $H=V(L_{A_0})$ for some $A_0\in J$, and set
\begin{equation}\label{eq:same-wall-set}
    S_H=\{A\in J:V(L_A)=H\}.
\end{equation}

\begin{lemma}\label{lem:positive-wall-normalization}
There is a unique equation
\begin{equation}\label{eq:normalized-wall}
 L_H(s)=1+\lambda_H\cdot s
\end{equation}
for $H$. For every $A\in S_H$ one has the exact equality of affine polynomials
\begin{equation}\label{eq:positive-wall-multiple}
 L_A=\nu_A L_H.
\end{equation}
In particular,
\[
 \lambda_H=\frac{N_A}{\nu_A}
\]
is independent of $A\in S_H$, and every proportionality coefficient $c_A:=\nu_A$ is positive.
\end{lemma}

\begin{proof}
Two nonconstant affine linear polynomials defining the same affine hyperplane differ by a nonzero scalar. If $A,B\in S_H$ and $L_A=qL_B$, comparison of constant terms gives $q=\nu_A/\nu_B>0$. Consequently all the polynomials $L_A/\nu_A$ agree and have constant term one. Their common value is \eqref{eq:normalized-wall}, which proves existence and \eqref{eq:positive-wall-multiple}. An affine equation for $H$ with constant term one is unique.
\end{proof}

The normalization in \Cref{lem:positive-wall-normalization} is essential: positivity would be meaningless for an equation of $H$ that were allowed to be multiplied by $-1$.

\subsection{Laurent coefficients}

The denominator list in \eqref{eq:surface-zeta} may contain several components with the same affine zero set. Those terms must be grouped before one decides whether the hyperplane is an actual pole. We carry out the grouping at the generic point of the hyperplane. Since every denominator in \eqref{eq:surface-zeta} is a product of the $L_A$, every actual polar hyperplane is of the form $V(L_A)$ for some $A\in J$.

Let $K(H)$ be the function field of $H$. The localization
\[
 \C[s_1,\ldots,s_r]_{(L_H)}
\]
is a discrete valuation ring with uniformizer $L_H$ and residue field $K(H)$. If $D\notin S_H$, then $L_D$ is a unit in this ring and its residue
\[
 L_D|_H\in K(H)^*
\]
is nonzero.

\begin{theorem}
\label{thm:grouped-laurent}
At the generic point of $H$, the polar part of the local topological zeta function is
\begin{equation}\label{eq:grouped-laurent}
 Z^{\mathrm{top}}_{F,0}
 =
 \frac{C_{-2}(H)}{L_H^2}
 +
 \frac{C_{-1}(H)}{L_H}
 +C_0,
\end{equation}
where $C_0$ is regular at the generic point, the first coefficient is
\begin{equation}\label{eq:grouped-double-coefficient}
 C_{-2}(H)
 =
 \sum_{\{A,B\}\subseteq S_H}
 \frac{m_{AB}}{\nu_A\nu_B},
\end{equation}
and the residue-field coefficient is
\begin{equation}\label{eq:grouped-simple-coefficient}
 C_{-1}(H)
 =
 \sum_{A\in S_H}\frac{1}{\nu_A}R_{A,H},
\end{equation}
with
\begin{equation}\label{eq:component-wall-residue}
 R_{A,H}
 =
 m_A
 +
 \sum_{D\notin S_H}
 \frac{m_{AD}}{L_D|_H}
 \in K(H).
\end{equation}
Here $C_{-1}(H)\in K(H)$ is the coefficient of $L_H^{-1}$ in the principal part at the generic point of $H$, so \eqref{eq:grouped-laurent} is understood modulo functions regular along $H$.
\end{theorem}

\begin{proof}
Split the terms of \eqref{eq:surface-zeta} by the number of their denominators belonging to $S_H$. For $A\in S_H$, the singleton term is
\[
 \frac{m_A}{L_A}
 =
 \frac{m_A/\nu_A}{L_H}.
\]
If $A,B\in S_H$, the crossing term is exactly
\[
 \frac{m_{AB}}{L_AL_B}
 =
 \frac{m_{AB}/(\nu_A\nu_B)}{L_H^2}.
\]
Because \eqref{eq:positive-wall-multiple} is an exact equality of affine forms, this term has no hidden contribution of order $L_H^{-1}$. If $A\in S_H$ and $D\notin S_H$, then
\[
 \frac{m_{AD}}{L_AL_D}
 =
 \frac{m_{AD}}{\nu_A L_HL_D};
\]
its coefficient in the residue field at order $L_H^{-1}$ is
\[
 \frac{m_{AD}}{\nu_A(L_D|_H)}.
\]
All terms whose indices are disjoint from $S_H$ are regular. Adding the three kinds of contributions gives \eqref{eq:grouped-double-coefficient}--\eqref{eq:component-wall-residue}. Since an SNC stratum on a surface contains at most two components, there is no pole of order greater than two.
\end{proof}

In particular, $H$ has order two precisely when $C_{-2}(H)\ne0$, and, if $C_{-2}(H)=0$, it is an actual pole precisely when $C_{-1}(H)\ne0$ in $K(H)$. Poles of one of the coefficient functions along a proper subvariety of $H$ do not make $H$ a component of the polar locus.

\subsection{The dichotomy}

We now apply the Laurent coefficients of \Cref{thm:grouped-laurent} to an actual polar hyperplane $H$. If two components in $S_H$ meet in the local fiber, their crossing gives a positive second-order coefficient. Otherwise $H$ is a simple pole and its first-order coefficient selects a single component, which must be either a strict transform or a rupture component. These two cases are summarized in the following.

\begin{theorem}
\label{thm:same-wall-positivity}
Suppose that $A,B\in S_H$ meet at a point of $\pi^{-1}(0)$. Then
\begin{equation}\label{eq:positive-double-coefficient}
    C_{-2}(H) = \sum_{\substack{\{A',B'\}\subseteq S_H}}\frac{m_{A'B'}}{\nu_{A'}\nu_{B'}}>0.
\end{equation}
Consequently $H$ is an actual pole of order two. Moreover, at every positive dual point
\begin{equation}\label{eq:positive-dual-wall}
    \alpha\in W_H^+ := \{\alpha\in\R_{>0}^r:\lambda_H\cdot\alpha=1\},
\end{equation}
every $A\in S_H$ satisfies 
\begin{equation} \label{eq:same-wall-zero-index}
N_A\cdot\alpha=\nu_A.
\end{equation}

\end{theorem}

\begin{proof}
Every $m_{A'B'}$ is a nonnegative integer, and every $\nu_{A'}$ is positive. If $A,B\in S_H$ meet in the local fibre, then $m_{AB}>0$, so the corresponding term in \eqref{eq:positive-double-coefficient} is strictly positive. It follows that $C_{-2}(H)>0$. The assertion about the order follows from \Cref{thm:grouped-laurent}. Finally, \eqref{eq:positive-wall-multiple} gives
\[
 N_A\cdot\alpha
 =
 \nu_A\lambda_H\cdot\alpha
 =
 \nu_A,
\]
which is \eqref{eq:same-wall-zero-index}.
\end{proof}

\Cref{thm:same-wall-positivity} applies to exceptional--exceptional, exceptional--strict, and strict--strict crossings. The local-fibre coefficient in \eqref{eq:positive-double-coefficient} is what excludes an irrelevant crossing away from the origin.

Assume now that
\begin{equation}\label{eq:no-same-wall-edge}
 m_{AB}=0
 \quad\text{for all }A,B\in S_H.
\end{equation}
Then $C_{-2}(H)=0$. If $H$ is an actual pole,
\eqref{eq:grouped-simple-coefficient} is nonzero, so at least one component
$A\in S_H$ satisfies
\begin{equation}\label{eq:selected-component}
 R_{A,H}\ne0.
\end{equation}

\begin{lemma}\label{lem:exceptional-neighbour-identities}
Let $E$ be an exceptional component, and let $d_E$ denote its valency in the total transform graph, counting arrowheads of strict transforms. Then
\begin{align}
 \sum_{D\neq E,\, D\cap E\neq \emptyset}N_D&=-E^2N_E,
 \label{eq:principal-neighbour-identity}\\
 \sum_{D\neq E,\, D\cap E\neq \emptyset}\nu_D&=-E^2\nu_E+d_E-2,
 \label{eq:adjunction-neighbour-identity}
\end{align}
and consequently
\begin{equation}\label{eq:affine-neighbour-identity}
 \sum_{D\neq E,\, D\cap E\neq \emptyset}L_D=-E^2L_E+d_E-2.
\end{equation}
\end{lemma}

\begin{proof}
For each $i$, the principal divisor $\operatorname{div}(f_i\circ\pi)$ has zero intersection with $E$. Since the total transform has simple normal crossings, every component $D$ meeting $E$ intersects it transversely at one point. Hence
\begin{equation*}
0=\operatorname{div}(f_i\circ\pi)\cdot E=N_{E,i}E^2+\sum_{D\neq E,\, D\cap E\neq \emptyset}N_{D,i}.
\end{equation*}
This proves \eqref{eq:principal-neighbour-identity} for each coordinate and hence as an identity of vectors. For the second identity, adjunction on $E\simeq\mathbb P^1$ gives
\begin{equation*}
(K_Y+E)\cdot E=-2.
\end{equation*}
Since $K_Y=K_{Y/X}+\pi^*K_X$ and $\pi^*K_X\cdot E=0$, we obtain $K_{Y/X}\cdot E=-2-E^2$. Expanding $K_{Y/X}$ along the exceptional components gives
\begin{equation*}
-2-E^2=(\nu_E-1)E^2+\sum_{D\ne E,\,D\cap E\ne\emptyset}(\nu_D-1),
\end{equation*}
where a strict transform may be included in the sum since $\nu_D=1$. As the number of components meeting $E$ is $d_E$, rearranging gives \eqref{eq:adjunction-neighbour-identity}. Finally, using $L_D=N_D\cdot s+\nu_D$ and combining the first two identities gives \eqref{eq:affine-neighbour-identity}.
\end{proof}

These are the standard numerical-data relations on a plane resolution graph; compare the adjunction formulation in \cite[\S2.3.1]{NemethiVeys2012} and the numerical-data arguments used in the pole classification of \cite{Veys1995}.

\begin{theorem}\label{thm:selected-component-classification}
Assume $m_{AB}=0$ for all $A, B\in S_H$. If $A\in S_H$ satisfies $R_{A,H}\neq 0$, then $A$ is either a strict transform or an exceptional component of valency at least three. 
\end{theorem}

\begin{proof}
It remains to exclude exceptional components of valency one and two. Let $A=E$ be exceptional. Every neighbour of $E$ meets it in the local fibre, and distinct neighbours meet it once, so
\begin{equation}\label{eq:exceptional-component-residue}
    R_{E,H} = 2-d_E+\sum_{D\neq E,\, D\cap E\neq \emptyset}\frac{1}{L_D|_H}.
\end{equation}
Since $m_{AB}=0$ for all $A, B\in S_H$, no neighbour $D$ belongs to $S_H$; therefore every denominator in \eqref{eq:exceptional-component-residue} is nonzero in $K(H)$. Restricting \eqref{eq:affine-neighbour-identity} to $H=V(L_E)$ gives
\begin{equation}\label{eq:restricted-neighbour-sum}
    \sum_{D\neq E,\, D\cap E\neq \emptyset}L_D|_H=d_E-2.
\end{equation}
If $d_E=1$, its unique neighbour satisfies $L_D|_H=-1$, and
\[
    R_{E,H}=1+\frac1{-1}=0.
\]
If $d_E=2$, writing $a=L_{D_1}|_H\ne0$, one has
$L_{D_2}|_H=-a$ and
\[
    R_{E,H}=\frac1a+\frac1{-a}=0.
\]
Both conclusions contradict $R_{E,H}\neq 0$. Hence $d_E\ge3$.
\end{proof}

An exceptional component of valency at least three in the total transform graph is called a rupture component. Thus, under the assumptions of \Cref{thm:selected-component-classification}, any exceptional component $A\in S_H$ with $R_{A,H}\neq0$ is a rupture component.

For later use, suppose $E\in S_H$ satisfies $R_{E,H}\neq0$. If $\alpha\in W_H^+$ satisfies $\epsilon_D(\alpha)\neq0$ for every neighbour $D$ of $E$, then $-\alpha\in H$ because $L_H(-\alpha)=1-\lambda_H\cdot\alpha=0$. Moreover, for every neighbour $D$ of $E$, $L_D(-\alpha)=\nu_D-N_D\cdot\alpha=\epsilon_D(\alpha)$. Hence evaluating \eqref{eq:exceptional-component-residue} at $s=-\alpha$ gives
\begin{equation}\label{eq:selected-residue-at-alpha}
    R_{E,H}(-\alpha) = 2-d_E + \sum_{D\neq E,\, D\cap E\neq \emptyset}\frac{1}{\epsilon_D(\alpha)}.
\end{equation}

\begin{theorem}\label{thm:grouped-wall-dichotomy} 
Let $H$ be an irreducible component of the polar locus $\PL(Z^{\mathrm{top}}_{F,0})$. Exactly one of the following holds.
\begin{enumerate}[label=\textup{(\roman*)}]
\item Two components $A,B\in S_H$ meet at a point of $\pi^{-1}(0)$. Then $H$ is a pole of order two and the leading coefficient $C_{-2}(H)>0$. Moreover, for every $\alpha\in W_H^+$ and every $A\in S_H$, one has $N_A\cdot\alpha=\nu_A$.
\item Any two components in $S_H$ are disjoint over $\pi^{-1}(0)$. Then $H$ is a simple pole and there exists $A\in S_H$ such that $R_{A,H}\neq0$. Every such $A$ is either a strict transform or an exceptional component of valency at least three.
\end{enumerate}
\end{theorem}

\begin{proof}
The two cases are mutually exclusive and exhaustive. The first follows from \Cref{thm:same-wall-positivity}. In the second case, \Cref{thm:grouped-laurent} gives $C_{-2}(H)=0$. Since $H$ is an irreducible component of the polar locus, one has $C_{-1}(H)\neq0$. It then follows from \eqref{eq:grouped-simple-coefficient} that $R_{A,H}\neq0$ for some $A\in S_H$, and the conclusion follows from \Cref{thm:selected-component-classification}.
\end{proof}

In case~\textup{(ii)}, any component $A\in S_H$ satisfying $R_{A,H}\neq0$ is either a rupture component or a strict transform. Rupture components will be treated using the residue obstructions developed below, while strict transforms are handled directly by the Bernstein factor in \Cref{thm:strict-transform-wall}.


\section{Rational powers and the Bernstein--Sato ideal}\label{sec:cyclic-quotient}

In this section, we connect rational powers of $F$ with the Bernstein--Sato locus. We first show that the nonvanishing of the class of $F^{-\alpha}$ in
\begin{equation*}
    \frac{\D_{X,0}F^{-\alpha}}{\D_{X,0}(hF^{-\alpha})}
\end{equation*}
forces $-\alpha\in Z(B_{F,0})$, and then use Zariski density to extend this pointwise criterion to an affine hyperplane. Strict transforms are handled separately. We also establish the Green formula and the pullback estimates used in \Cref{sec:residue-obstructions}.

\subsection{Evaluation at rational parameters}

We first relate rational powers to the multivariable Bernstein--Sato ideal by evaluating the functional equation at $s=-\alpha$. 

Recall that $h=f_1\cdots f_r$. Put $\Delta=(h=0)$ and for $\alpha\in\Q_{>0}^r$, put $u_\alpha=F^{-\alpha}$. Let  
\begin{equation*}
\mathcal M_\alpha=\mathcal O_X(*\Delta)u_\alpha.
\end{equation*}
Its stalk at $0$ is
\begin{equation*}
\mathcal M_{\alpha,0}=\mathcal O_{X,0}[h^{-1}]u_\alpha.
\end{equation*}
The action of a holomorphic vector field $\xi$ on $\mathcal M_{\alpha,0}$ is given by
\begin{equation*}
\xi(gu_\alpha)=\left(\xi(g)-g\sum_{i=1}^r\alpha_i\frac{\xi(f_i)}{f_i}\right)u_\alpha.
\end{equation*}
This is precisely the specialization at $s=-\alpha$ of the $\D_{X,0}[s]$-action on $\mathcal O_{X,0}[h^{-1},s]F^s$. Modules generated by rational powers and the comparison between successive powers are studied systematically by Saito \cite{Saito2021RationalPowers}; the implication below uses only direct specialization of the functional equation.

\begin{proposition}\label{prop:affine-point-lifting}
Let $\alpha\in\Q_{>0}^r$. If
\begin{equation}\label{eq:cyclic-noncontainment}
 u_\alpha\notin\D_{X,0}(hu_\alpha),
\end{equation}
then
\begin{equation}\label{eq:point-in-bernstein-zero-set}
 -\alpha\in Z(B_{F,0}).
\end{equation}
\end{proposition}

\begin{proof}
Let $b(s)\in B_{F,0}$. By definition, there exists $P(s)\in\D_{X,0}[s]$ such that
\begin{equation*}
b(s)F^s=P(s)hF^s.
\end{equation*}
Evaluating at $s=-\alpha$ gives an identity in $\mathcal M_{\alpha,0}$
\begin{equation*}
b(-\alpha)u_\alpha=P(-\alpha)hu_\alpha.
\end{equation*}
If $b(-\alpha)\neq0$, then
\begin{equation*}
u_\alpha=\frac{1}{b(-\alpha)}P(-\alpha)(hu_\alpha) \in \D_{X,0}(hu_\alpha),
\end{equation*}
contrary to \eqref{eq:cyclic-noncontainment}. Hence $b(-\alpha)=0$ for every $b(s)\in B_{F,0}$, and therefore $-\alpha\in Z(B_{F,0})$.
\end{proof}

\begin{remark}
For $r=1$, the implication in \Cref{prop:affine-point-lifting} is the elementary consequence of evaluating the Bernstein functional equation at $s=-\alpha$. The converse problem is substantially more delicate: Saito studied conditions under which a root $-\alpha$ of the Bernstein--Sato polynomial yields a strict inclusion between the $\D$-modules generated by $f^{-\alpha}$ and $f^{1-\alpha}$, and showed that the converse does not hold in general \cite{Saito2021RationalPowers}. \Cref{prop:affine-point-lifting} uses only the direct implication, in the multivariable setting.
\end{remark}

\subsection{Rational density on affine hyperplanes}

The following density argument applies equally to hyperplanes defined by exceptional components and by strict transforms. Let
\begin{equation}\label{eq:general-positive-dual-wall}
 H=V(1+\lambda\cdot s), \quad W_H=\{\alpha\in\C^r:\lambda\cdot\alpha=1\}, \quad W_H^+=W_H\cap\R_{>0}^r,
\end{equation}
where $0\neq\lambda\in\Q_{\ge0}^r$. The set $W_H^+$ contains rational points. Indeed, choose $j$ with $\lambda_j>0$ and choose positive rational numbers $\alpha_i$ for $i\neq j$ sufficiently small that $\sum_{i\neq j}\lambda_i\alpha_i<1$. Then $\alpha_j=(1-\sum_{i\neq j}\lambda_i\alpha_i)/\lambda_j
$ is also positive and rational. For $r\ge2$, $W_H^+$ is a relatively open subset of the real affine hyperplane $W_H(\R)$, which has dimension $r-1$; for $r=1$, it consists of the single point $1/\lambda_1$. If all coordinates of $\lambda$ are positive, $W_H^+$ is bounded, whereas it may be unbounded when some $\lambda_i=0$. We will therefore work inside a bounded relatively open rational polytope contained in $W_H^+$ when needed.

\begin{lemma}\label{lem:bounded-polytope-density}
Let $H$ and $W_H^+$ be as in \eqref{eq:general-positive-dual-wall}. Let $q_1,\ldots,q_a$ be nonzero rational functions on $W_H$ defined over $\Q$, and let $\eta_1,\ldots,\eta_b$ be affine functions on $W_H$ with rational coefficients. Suppose $r\ge2$. Then there exists a bounded rational relatively open polytope $\mathcal B\subset W_H^+$ on which every $q_i$ is defined and nonzero. For each nonconstant $\eta_j$, remove from $\mathcal B$ all sets of the form
\begin{equation*}
\{\alpha\in\mathcal B:\eta_j(\alpha)=m\},\quad m\in\Z,
\end{equation*}
that meet $\mathcal B$. The rational points in the remaining set are Zariski dense in $W_H$.
If $r=1$, then $W_H$ consists of a single rational point, and the same conclusion is immediate.
\end{lemma}

\begin{proof}
Assume $r\ge2$. Let $Z\subset W_H$ be the union of the zero and pole loci of the functions $q_1,\ldots,q_a$. Since each $q_i$ is nonzero, $Z$ is a proper algebraic subset of $W_H$. The set $W_H^+$ is a nonempty relatively open subset of $W_H(\R)$, so it contains a point outside $Z$. Since all the data are defined over $\Q$ and rational points are dense in $W_H(\R)$, we may choose such a point with rational coordinates. We may then choose a bounded rational relatively open polytope $\mathcal B\subset W_H^+$ whose closure is disjoint from $Z$. Thus every $q_i$ is defined and nonzero on $\mathcal B$.

For each nonconstant $\eta_j$, the image $\eta_j(\overline{\mathcal B})$ is bounded, so only finitely many integers $m$ satisfy
\begin{equation*}
\{\alpha\in\mathcal B:\eta_j(\alpha)=m\}\neq\varnothing.
\end{equation*}
Hence only finitely many proper rational affine subspaces are removed from $\mathcal B$. Their complement contains a nonempty relatively open subset of $W_H(\R)$, and its rational points are Euclidean dense. If a polynomial on $W_H$ vanishes at all of these rational points, then it vanishes on a nonempty real open subset of $W_H(\R)$ and therefore vanishes identically on $W_H$. Thus the remaining rational points are Zariski dense in $W_H$.

If $r=1$, then $W_H$ consists of a single rational point, so the assertion is immediate.
\end{proof}



\begin{lemma}\label{lem:selected-rupture-density}
Assume $m_{AB}=0$ for all $A,B\in S_H$, and let $E\in S_H$ satisfy $R_{E,H}\neq0$. There exists a Zariski-dense set
\begin{equation*}
\mathcal U_E(\Q)\subset W_H^+\cap\Q^r
\end{equation*}
such that, for every $\alpha\in\mathcal U_E(\Q)$,
\begin{enumerate}[label=\textup{(\roman*)}]
\item $\epsilon_D(\alpha)\neq0$ for every neighbour $D$ of $E$;
\item $R_{E,H}(-\alpha)\neq0$;
\item if the restriction of $\epsilon_D$ to $W_H$ is nonconstant, then $\epsilon_D(\alpha)\notin\Z$.
\end{enumerate}
\end{lemma}

\begin{proof}
Let $D$ be a neighbour of $E$. Since $E\in S_H$, the assumption $m_{AB}=0$ for all $A,B\in S_H$ implies that $D\notin S_H$; otherwise $E$ and $D$ would be two components in $S_H$ meeting in the local fibre. Hence $L_D|_H$ is a nonzero element of $K(H)$. By assumption, $R_{E,H}$ is also nonzero in $K(H)$. The map $W_H\to H$, $\alpha\mapsto-\alpha$, is an affine isomorphism, and for every neighbour $D$ of $E$ one has
\begin{equation*}
L_D(-\alpha)=\nu_D-N_D\cdot\alpha=\epsilon_D(\alpha).
\end{equation*}
Apply \Cref{lem:bounded-polytope-density} to the rational functions $L_D|_H$ and $R_{E,H}$, pulled back to $W_H$ by $\alpha\mapsto-\alpha$, and to the affine functions $\epsilon_D$. We obtain a Zariski-dense set of rational points in $W_H^+$ at which every $L_D(-\alpha)$ and $R_{E,H}(-\alpha)$ is defined and nonzero, and at which every nonconstant $\epsilon_D$ is nonintegral. These are precisely \textup{(i)}--\textup{(iii)}. For $r=1$, the same conclusion follows from the last statement of \Cref{lem:bounded-polytope-density}.
\end{proof}

\begin{theorem}\label{thm:dense-wall-lifting}
Let $H=V(1+\lambda\cdot s)$ be a rational affine hyperplane with $0\ne\lambda\in\Q_{\ge0}^r$. Suppose that
\[
 u_\alpha\notin\D_{X,0}(hu_\alpha)
\]
for a Zariski-dense set of rational points $\alpha\in W_H^+$. Then
\begin{equation}\label{eq:whole-bernstein-wall}
 H\subseteq Z(B_{F,0}).
\end{equation}
\end{theorem}

\begin{proof}
By \Cref{prop:affine-point-lifting}, every $b\in B_{F,0}$ vanishes at $-\alpha$ for every point $\alpha$ in the given Zariski-dense set. The map $\alpha\mapsto-\alpha$ is an affine isomorphism from $W_H=\{\alpha\in\C^r:\lambda\cdot\alpha=1\}$ onto $H$, so the corresponding points $-\alpha$ are Zariski dense in $H$. Hence the restriction of $b$ to $H$ vanishes identically. Since this holds for every $b\in B_{F,0}$, we obtain \eqref{eq:whole-bernstein-wall}. The case $r=1$ is included, since both $W_H$ and $H$ consist of a single point.
\end{proof}

Exponentiated Bernstein supports and their relation with local-system support loci are studied in \cite{BvdVWZ2021}; logarithmic nearby-cycle methods give a related perspective in \cite{Wu2026}. The argument above is deliberately affine: it uses the value $s=-\alpha$ itself rather than only its exponential.

We will use \Cref{thm:dense-wall-lifting} in two situations. If two components $A,B\in S_H$ meet over $\pi^{-1}(0)$, then $N_A\cdot\alpha=\nu_A$ and $N_B\cdot\alpha=\nu_B$ for every $\alpha\in W_H^+$, and the double-residue obstruction developed below will give the required noncontainment at rational points. If $E\in S_H$ satisfies $R_{E,H}\neq0$ and no two components in $S_H$ meet over $\pi^{-1}(0)$, then \Cref{lem:selected-rupture-density} provides the Zariski-dense set of rational points used in the residue argument.

\subsection{Strict transforms}

For a strict transform $C$, the corresponding hyperplane is contained in the Bernstein--Sato locus by the following result.

\begin{theorem}\label{thm:strict-transform-wall}
Let $C$ be an irreducible branch of the reduced divisor $\Delta=(h=0)$, and let
\[
 N_C=(\ord_C(f_1),\ldots,\ord_C(f_r)).
\]
Then
\begin{equation}\label{eq:strict-transform-wall}
 V(N_C\cdot s+1)
 \subseteq
 Z(B_{F,0}).
\end{equation}
This remains true with arbitrary positive multiplicities of $C$ in the entries of $F$.
\end{theorem}

\begin{proof}
Choose generators $b_1,\ldots,b_t$ of $B_{F,0}\subseteq\C[s]$, together with functional equations representing them in a neighbourhood of the origin. After shrinking this neighbourhood if necessary, all these equations are defined on a common neighbourhood $U$ of $0$. Choose a smooth point $p\in C\cap U$ at which no other irreducible component of the reduced divisor $(h=0)$ is present. Restricting the functional equations to $p$ gives $b_j\in B_{F,p}$ for every $j$.

We now use the local calculation in the proof of \cite[Proposition~1.3]{BvdVVW2024}, which also applies verbatim in the analytic setting. More generally, for $a\in\Z_{\ge0}^r$ they consider
\begin{equation*}
B_{F,p}^a=\{b(s):b(s)F^s\in\D_{X,p}[s]F^{s+a}\}.
\end{equation*}
At the chosen point $p$, one may write $f_i=u_i z^{N_{C,i}}$, where $z$ is a local equation of $C$ and each $u_i$ is a unit. If $m=N_C\cdot a\neq0$, their local computation gives
\begin{equation*}
B_{F,p}^a=\left(\prod_{c=1}^m(N_C\cdot s+c)\right).
\end{equation*}
Taking $a=\one$, we have $B_{F,p}^{\one}=B_{F,p}$ and $m=N_C\cdot\one\ge1$. Hence
\begin{equation*}
V(N_C\cdot s+1)\subseteq Z(B_{F,p}).
\end{equation*}
Since each $b_j$ belongs to $B_{F,p}$, every generator $b_j$ vanishes on $V(N_C\cdot s+1)$. It follows that every element of $B_{F,0}$ vanishes on this hyperplane, and therefore
\begin{equation*}
V(N_C\cdot s+1)\subseteq Z(B_{F,0}).
\end{equation*}
The local calculation uses the full vector $N_C=(\ord_C(f_1),\ldots,\ord_C(f_r))$, so the argument allows arbitrary multiplicities of $C$ in the entries of $F$.
\end{proof}


\subsection{Green formula}


We first establish the Green formula in arbitrary dimension. Let $X$ be a smooth complex manifold of dimension $n$, and let $\omega_X$ denote its canonical bundle. The standard right $\D_X$-module structure on $\omega_X$ is defined on local holomorphic sections $\omega$ by
\begin{equation}\label{eq:right-D-action}
\omega\mathbin{\cdot}a=a\omega,\quad \omega\mathbin{\cdot}\xi=-L_\xi\omega,
\end{equation}
for $a\in\mathcal O_X$ and a holomorphic vector field $\xi$, where $L_\xi$ denotes the Lie derivative along $\xi$. The sign in \eqref{eq:right-D-action} is compatible with the commutation relation
\begin{equation*}
\omega\mathbin{\cdot}(\xi a-a\xi)=\xi(a)\omega.
\end{equation*}
For $P\in\D_X$, define
\begin{equation}\label{eq:dagger-definition}
P^\dagger\omega:=\omega\mathbin{\cdot}P.
\end{equation}
We use the convention
\begin{equation*}
(PQ)(g)=P(Q(g)).
\end{equation*}
By associativity of the right $\D_X$-action,
\begin{equation}\label{eq:dagger-composition}
(PQ)^\dagger\omega=Q^\dagger(P^\dagger\omega).
\end{equation}
If $P$ has holomorphic coefficients and $\omega$ is holomorphic, then $P^\dagger\omega$ is holomorphic. For example, in local coordinates $(x_1,\ldots,x_n)$ with $\omega=dx_1\wedge\cdots\wedge dx_n$,
\begin{equation*}
(a\partial_{x_j})^\dagger\omega=-\partial_{x_j}(a)\omega,\quad (\partial_{x_j}a)^\dagger\omega=0.
\end{equation*}


The right $\D_X$-module structure on the canonical bundle is standard; see, for example, \cite{HottaTakeuchiTanisaki2008}. The following identity is the usual Green formula for a differential operator and its formal adjoint; see, for example, \cite[Chapter~2]{LionsMagenes1972}. We record the coefficient-valued form needed below.

\begin{proposition}\label{prop:green-formula}
Let $X$ be a smooth complex manifold, let $\Lcal$ be a local system on $X$ with flat connection $\nabla$, let $g$ be a meromorphic section with coefficients in $\Lcal$, and let $\omega$ be a holomorphic top form. For every $P\in\D_X$, there exists a meromorphic form $\eta_P(g,\omega)$ of degree $\dim X-1$ with coefficients in $\Lcal$ such that
\begin{equation}\label{eq:green-formula}
P(g)\omega=gP^\dagger\omega+\nabla\eta_P(g,\omega).
\end{equation}
\end{proposition}

\begin{proof}
It is enough to prove the formula for multiplication by holomorphic functions and for holomorphic vector fields, and then to check that it is preserved under composition and addition. For multiplication by $a\in\mathcal O_X$, set $\eta_a(g,\omega)=0$. For a holomorphic vector field $\xi$, define
\begin{equation}\label{eq:first-order-concomitant}
\eta_\xi(g,\omega)=g\iota_\xi\omega,
\end{equation}
where $\iota_\xi$ denotes contraction with $\xi$. In a local flat frame of the local system, Cartan's formula gives
\begin{equation*}
d(g\iota_\xi\omega)=\xi(g)\omega+gL_\xi\omega.
\end{equation*}
Together with \eqref{eq:right-D-action}, this proves \eqref{eq:green-formula} for a vector field.

Suppose that \eqref{eq:green-formula} holds for $P$ and $Q$, and define
\begin{equation}\label{eq:concomitant-recursion}
\eta_{PQ}(g,\omega)=\eta_P(Qg,\omega)+\eta_Q(g,P^\dagger\omega).
\end{equation}
Applying the formula for $P$ to $Qg$ gives
\begin{equation*}
P(Qg)\omega=(Qg)P^\dagger\omega+\nabla\eta_P(Qg,\omega).
\end{equation*}
Applying the formula for $Q$ with the holomorphic top form $P^\dagger\omega$ gives
\begin{equation*}
(Qg)P^\dagger\omega=gQ^\dagger(P^\dagger\omega)+\nabla\eta_Q(g,P^\dagger\omega).
\end{equation*}
Combining these identities and using \eqref{eq:dagger-composition}, we obtain
\begin{equation*}
P(Qg)\omega=g(PQ)^\dagger\omega+\nabla\eta_{PQ}(g,\omega).
\end{equation*}
Linearity proves the formula for every $P\in\D_X$.

Finally, analytic continuation acts on $g$ and its derivatives by the same monodromy transformation. Hence the local identities are compatible on overlaps and define forms with coefficients in the given local system.
\end{proof}

We now return to the case $\dim X=2$. Fix $\alpha\in\Q_{>0}^r$ and put $\Delta=(h=0)$. Recall that
\begin{equation*}
u_\alpha=F^{-\alpha}=\prod_{i=1}^r f_i^{-\alpha_i},\quad \mathcal M_{\alpha,0}=\mathcal O_{X,0}[h^{-1}]u_\alpha.
\end{equation*}
For every holomorphic vector field $\xi$,
\begin{equation}\label{eq:coefficient-connection-action}
\xi(u_\alpha)=-\sum_{i=1}^r\alpha_i\frac{\xi(f_i)}{f_i}u_\alpha.
\end{equation}
Let $\Lcal_\alpha^{\mathrm{act}}$ be the actual-branch local system introduced above. In the formal frame $u_\alpha$, the left $\D_X$-module action in \eqref{eq:coefficient-connection-action} corresponds to the connection
\[
 d-\sum_{i=1}^r\alpha_i\frac{df_i}{f_i}.
\]
Its horizontal local system is
\[
 \Lcal_\alpha:=\bigl(\Lcal_\alpha^{\mathrm{act}}\bigr)^\vee.
\]
Indeed, on a simply connected open set choose a branch $\varphi=F^{-\alpha}$ and put $e=\varphi^{-1}u_\alpha$. Then $e$ is a horizontal frame and $u_\alpha=\varphi e$. A positively oriented meridian $\gamma_A$ around a component $A$ acts on the actual scalar branch by
\begin{equation}\label{eq:actual-branch-character}
\chi_\alpha^{\mathrm{act}}(\gamma_A)=\exp\bigl(-2\pi iN_A\cdot\alpha\bigr),
\end{equation}
whereas the parallel-transport monodromy of $\Lcal_\alpha$ is the inverse character. Thus scalar representatives of $\Lcal_\alpha$-valued forms on the universal cover transform by $\chi_\alpha^{\mathrm{act}}$. We write $\nabla_{\Lcal_\alpha}$ for the flat connection on these forms; in a horizontal frame it is the ordinary exterior derivative. Integration cycles carry coefficients in $\Lcal_\alpha^\vee=\Lcal_\alpha^{\mathrm{act}}$. We next apply the Green formula after pulling back to a log resolution.

\subsection{Pullbacks}

Let $\pi:Y\to X$ be an embedded log resolution, and let
\begin{equation*}
D=\bigl(\pi^{-1}(\Delta)\bigr)_{\mathrm{red}}
\end{equation*}
be the reduced total transform, where $\Delta=(h=0)$. We keep the notation $N_A$, $n_A$, and $\nu_A$ introduced earlier for the irreducible components $A$ of $D$. Recall that
\begin{equation*}
n_A=N_A\cdot\one=\ord_A(h)\ge1. 
\end{equation*}

Fix a nowhere vanishing holomorphic two-form $\omega$ on the surface germ. At a crossing $p\in A\cap B$, choose SNC coordinates with $A=(x=0)$ and $B=(y=0)$. The Jacobian divisor and the total transform equations give
\begin{align}
 \pi^*\omega
 &=J(x,y)x^{\nu_A-1}y^{\nu_B-1}\,dx\wedge dy,
 &J(0,0)&\neq0,
 \label{eq:simultaneous-jacobian}\\
 \pi^*f_i
 &=u_i(x,y)x^{N_{A,i}}y^{N_{B,i}},
 &u_i(0,0)&\neq0.
 \label{eq:simultaneous-total-transform}
\end{align}
These formulas include exceptional--exceptional, exceptional--strict, and strict--strict crossings. In particular, for every holomorphic top form $\theta=a\omega$, the coefficient of $\pi^*\theta$ is simultaneously divisible by
\begin{equation}\label{eq:simultaneous-divisibility}
 x^{\nu_A-1}y^{\nu_B-1}.
\end{equation}

Choose compatible rational-power branches of the units in \eqref{eq:simultaneous-total-transform}. If
\[
 \epsilon_C(\alpha)=\nu_C-N_C\cdot\alpha,
\]
then
\begin{equation}\label{eq:coefficient-form-crossing}
 \pi^*(u_\alpha\omega)
 =U(x,y)x^{\epsilon_A(\alpha)-1} y^{\epsilon_B(\alpha)-1}\,dx\wedge dy,
 \quad U(0,0)\neq0.
\end{equation}

We first record how the Green formula behaves under pullback. After pulling $\Lcal_\alpha$ back to $Y\setminus D$, the ordinary pullback of differential forms commutes with the flat connections:
\begin{equation*}
\pi^*(\nabla_{\Lcal_\alpha}\eta)=\nabla_{\pi^{-1}\Lcal_\alpha}(\pi^*\eta).
\end{equation*}
Here we pull back the differential forms and the coefficient system; there is no pullback of the differential operator $P$.

Assume that $u_\alpha\in\D_{X,0}(hu_\alpha)$, and choose $P\in\D_{X,0}$ such that
\begin{equation}\label{eq:cyclic-containment-operator}
u_\alpha=P(hu_\alpha).
\end{equation}
Let $\omega$ be a nowhere vanishing holomorphic two-form near $0$. By \Cref{prop:green-formula},
\begin{equation}\label{eq:green-cyclic-identity}
u_\alpha\omega=(hu_\alpha)P^\dagger\omega+\nabla_{\Lcal_\alpha}\eta_P(hu_\alpha,\omega).
\end{equation}
We pull back this identity to $Y$. If $N_E\cdot\alpha=\nu_E$, then the first term on the right has order along $E$ at least
\begin{equation}\label{eq:one-component-green-order}
n_E-N_E\cdot\alpha+\nu_E-1=n_E-1\ge0.
\end{equation}
If two components $A$ and $B$ meet and satisfy $N_A\cdot\alpha=\nu_A$ and $N_B\cdot\alpha=\nu_B$, then \eqref{eq:simultaneous-divisibility} shows that the pullback of the first term on the right is a holomorphic multiple of
\begin{equation}\label{eq:two-component-green-order}
x^{n_A-1}y^{n_B-1}dx\wedge dy.
\end{equation}
These estimates follow from \eqref{eq:simultaneous-jacobian}--\eqref{eq:simultaneous-divisibility}.


\section{Residue obstructions on surfaces}\label{sec:residue-obstructions}

We now use the Green formula to show that $u_\alpha\notin\D_{X,0}(hu_\alpha)$ under suitable residue conditions. For a component $E$ satisfying $N_E\cdot\alpha=\nu_E$, integration over tubes around cycles in $E^\circ$ shows that the Poincar\'e residue of an exact form is trivial in twisted cohomology; the residue extends across those intersection points $E\cap D$ for which $\epsilon_D(\alpha)=1$. If two components $A$ and $B$ meet and satisfy $N_A\cdot\alpha=\nu_A$ and $N_B\cdot\alpha=\nu_B$, integration over a small torus around the crossing gives an obstruction from the double residue. The argument applies equally to all types of crossings.

\subsection{Local systems on a punctured tubular neighborhood}

Let $Y$ be a smooth complex surface, let $E\subset Y$ be a smooth connected curve, and let $E^\circ=E\setminus \Sigma$, where $\Sigma$ is finite and $E^\circ$ has finite topological type. Choose a Hermitian metric on the normal bundle $N_{E/Y}|_{E^\circ}$. A sufficiently small positive radius function determines an oriented circle bundle
\begin{equation*}
q:S(N_{E/Y}|_{E^\circ})\to  E^\circ
\end{equation*}
embedded in a punctured tubular neighborhood $T^\times$ of $E^\circ$. The circle bundle is a deformation retract of $T^\times$.

Let $\Lcal$ be a local system of rank one on $T^\times$ whose monodromy has finite order, and assume that the monodromy around a positively oriented normal circle is trivial. Since $T^\times$ deformation retracts onto the circle bundle, the homotopy exact sequence gives a surjection $\pi_1(T^\times)\to\pi_1(E^\circ)$ whose kernel is normally generated by the class of a normal circle. It follows that the monodromy character of $\Lcal$ factors through $\pi_1(E^\circ)$. Hence there is a local system $\Lcal_E$ of rank one on $E^\circ$, also with monodromy of finite order, such that
\begin{equation}\label{eq:tube-system-pullback}
\Lcal|_{S(N_{E/Y}|_{E^\circ})}\simeq q^*\Lcal_E.
\end{equation}
Differential forms take values in $\Lcal$, while integration cycles carry coefficients in the dual local system $\Lcal^\vee$.

Let $\eta$ be a meromorphic one-form on $T^\times$ with coefficients in $\Lcal$ and with finite pole order along the missing zero section. Suppose that
\begin{equation*}
\Theta=\nabla_{\Lcal}\eta
\end{equation*}
extends meromorphically across $E^\circ$ and has at most a simple pole in the normal direction. In local holomorphic coordinates with $E=(x=0)$, it can be written as
\begin{equation}\label{eq:simple-normal-pole}
\Theta=\frac{dx}{x}\wedge\beta(x,y)+\gamma(x,y),
\end{equation}
where $\beta$ is a holomorphic one-form tangent to $E$ and depends holomorphically on $x$, while $\gamma$ is regular in the normal variable $x$. Since the monodromy around a normal circle is trivial, the residue
\begin{equation}\label{eq:residue-of-theta}
\Res_E(\Theta)=\beta(0,y)
\end{equation}
is a globally defined one-form on $E^\circ$ with coefficients in $\Lcal_E$. Since the connection is flat, $\nabla_{\Lcal}\Theta=\nabla_{\Lcal}^2\eta=0$. The residue is compatible with the induced flat connection on $E^\circ$, and therefore $\Res_E(\Theta)$ is closed.

\subsection{The Gysin map and tube classes}

Let $D(N_{E/Y}|_{E^\circ})$ and $S(N_{E/Y}|_{E^\circ})$ denote the normal disk and circle bundles over $E^\circ$, respectively, and write
\begin{equation*}
p:D(N_{E/Y}|_{E^\circ})\to  E^\circ,\quad q:S(N_{E/Y}|_{E^\circ})\to  E^\circ
\end{equation*}
for their projections. The complex structure on the normal line bundle determines a canonical orientation of its underlying real two-dimensional fibres, and hence an orientation of the circle fibres.

\begin{proposition}\label{prop:degree-one-thom-gysin}
    There is a natural homology map
    \begin{equation}\label{eq:gysin-homology-transfer}
    q_E^!:H_k\bigl(E^\circ,\Lcal_E^\vee\bigr)\to H_{k+1}\bigl(S(N_{E/Y}|_{E^\circ}),q^*\Lcal_E^\vee\bigr)
    \end{equation}
    defined by
    \begin{equation*}
    q_E^!:=\partial\circ\operatorname{Th},
    \end{equation*}
    where
    \begin{equation*}
    \operatorname{Th}:H_k\bigl(E^\circ,\Lcal_E^\vee\bigr)\xrightarrow{\sim}H_{k+2}\bigl(D(N_{E/Y}|_{E^\circ}),S(N_{E/Y}|_{E^\circ});p^*\Lcal_E^\vee\bigr)
    \end{equation*}
    is the homology Thom isomorphism and $\partial$ is the boundary map of the disk and circle bundle pair. Over an open set on which the normal bundle is trivial, the map is given by
    \begin{equation*}
    [\gamma]\mapsto[S^1]\times[\gamma],
    \end{equation*}
    where $S^1$ is the positively oriented normal circle. The map is independent of the choice of Hermitian metric and of the positive radius used to realize the circle bundle inside the punctured tubular neighborhood.
\end{proposition}

\begin{proof}
    The complex structure gives the normal line bundle a canonical orientation as a real vector bundle, so the Thom isomorphism is defined with coefficients in $p^*\Lcal_E^\vee$. When the normal bundle is trivial, the Thom class is represented by the oriented normal disk, and the boundary map sends it to the positively oriented normal circle. This gives the stated local description of $q_E^!$. Naturality of the Thom isomorphism and of the boundary map shows that these local descriptions agree on overlaps. Finally, homotopies of the Hermitian metric or of the positive radius identify the corresponding disk and circle bundle pairs and preserve the Thom class, so the resulting homology map is independent of these choices. Standard background on the Thom isomorphism and Gysin maps can be found in \cite[Chapter~VI, Section~11]{Bredon1993}.
\end{proof}

For $[\gamma]\in H_1(E^\circ,\Lcal_E^\vee)$, we call $q_E^![\gamma]$ the tube class associated with $[\gamma]$. A singular homology class has a compact representative, so the normal radius may be chosen sufficiently small over its support that the corresponding tube avoids the deleted points $\Sigma$ and, in the applications below, the other components of the total transform.

\begin{lemma}\label{lem:twisted-tube-formula}
For every $[\gamma]\in H_1(E^\circ,\Lcal_E^\vee)$,
\begin{equation}\label{eq:twisted-tube-formula}
\left\langle[\Theta],q_E^![\gamma]\right\rangle=2\pi i\left\langle[\Res_E(\Theta)],[\gamma]\right\rangle.
\end{equation}
\end{lemma}

\begin{proof}
Locally, the Gysin map $q_E^!$ sends a cycle $\gamma$ to the product of $\gamma$ with the positively oriented normal circle. If $x$ is a local holomorphic coordinate with $E=(x=0)$, then
\begin{equation*}
\int_{|x|=\varepsilon}\frac{dx}{x}=2\pi i.
\end{equation*}
Applying integration along the circle fibres to \eqref{eq:simple-normal-pole} therefore gives $2\pi i\beta(0,y)$. Indeed, terms that are regular in the normal variable and contain $dx$ have zero circle integral by Cauchy's theorem, while terms without a differential in the normal direction vanish upon restriction to the circle fibres. 

If $x'=a(x,y)x$ is another local equation of $E$, where $a$ is nowhere vanishing, then
\begin{equation*}
\frac{dx'}{x'}=\frac{dx}{x}+\frac{da}{a}.
\end{equation*}
The second term is regular in the normal direction and does not change the residue. On overlaps, the transition functions of the local system cancel those of the dual coefficients on the cycles, so the local fibre integrals agree. The orientation of the normal circle gives the positive sign in \eqref{eq:twisted-tube-formula}. Background on integration along fibres can be found in \cite[Chapter~I, Section~6]{BottTu1982}.
\end{proof}

\subsection{Twisted Stokes and the integration pairing}

Let $C$ be a singular two-chain with coefficients in $\Lcal^\vee$, and let $\zeta$ be an $\Lcal$-valued one-form defined in a neighbourhood of the support of $C$. Stokes' theorem in local flat frames, together with the pairing between $\Lcal$ and $\Lcal^\vee$, gives
\begin{equation}\label{eq:twisted-stokes}
\int_C\nabla_{\Lcal}\zeta=\int_{\partial C}\zeta.
\end{equation}
In particular, $\zeta$ need not extend across the missing divisor.

The integration pairing
\begin{equation}\label{eq:twisted-perfect-pairing}
H^1_{\mathrm{dR}}(E^\circ,\Lcal_E)\times H_1(E^\circ,\Lcal_E^\vee)\to\C
\end{equation}
is perfect. Indeed, $E^\circ$ has finite CW type, and the twisted de Rham theorem identifies $H^1_{\mathrm{dR}}(E^\circ,\Lcal_E)$ with the cohomology of the cellular cochain complex with coefficients in $\Lcal_E$. This cochain complex is naturally dual to the cellular chain complex with coefficients in $\Lcal_E^\vee$. Since these complexes are finite dimensional over $\C$, dualization is exact, and the pairing in \eqref{eq:twisted-perfect-pairing} is nondegenerate. For the de Rham description of flat bundles and local systems, see \cite{Deligne1970}; standard background on homology with local coefficients can also be found in \cite[Chapter~VI]{Bredon1993}.

\begin{theorem}\label{thm:tube-residue-exactness}
Under the hypotheses above,
\begin{equation}\label{eq:tube-residue-zero}
\left[\Res_E(\nabla_{\Lcal}\eta)\right]=0\quad\text{in}\quad H^1_{\mathrm{dR}}(E^\circ,\Lcal_E).
\end{equation}
The form $\eta$ may have an arbitrary finite meromorphic pole in the normal direction; only $\nabla_{\Lcal}\eta$ is required to have at most a simple normal pole.
\end{theorem}

\begin{proof}
Let $[\gamma]\in H_1(E^\circ,\Lcal_E^\vee)$. The tube class $q_E^![\gamma]$ is represented at positive normal radius, where $\eta$ is defined. Since it is a cycle, twisted Stokes gives
\begin{equation*}
\left\langle[\nabla_{\Lcal}\eta],q_E^![\gamma]\right\rangle=\int_{q_E^![\gamma]}\nabla_{\Lcal}\eta=0.
\end{equation*}
Applying \Cref{lem:twisted-tube-formula} with $\Theta=\nabla_{\Lcal}\eta$, we obtain
\begin{equation*}
0=2\pi i\left\langle\left[\Res_E(\nabla_{\Lcal}\eta)\right],[\gamma]\right\rangle.
\end{equation*}
This holds for every $[\gamma]\in H_1(E^\circ,\Lcal_E^\vee)$. The perfect pairing \eqref{eq:twisted-perfect-pairing} therefore gives \eqref{eq:tube-residue-zero}.
\end{proof}

\begin{remark}
The assumption that $\eta$ is meromorphic with finite pole order is essential. For example, the multivalued form $(\log x)dy/y$ satisfies
\begin{equation*}
d\left((\log x)\frac{dy}{y}\right)=\frac{dx}{x}\wedge\frac{dy}{y},
\end{equation*}
but $\log x$ has additive monodromy and does not define a meromorphic section of a local system of the type considered above.
\end{remark}

\subsection{The residue obstruction}

Let us return to the embedded resolution of the reduced support of $(h=0)$. Let $E$ be a component of the total transform and suppose
\begin{equation}\label{eq:E-wall-equation}
N_E\cdot\alpha=\nu_E.
\end{equation}
By \eqref{eq:actual-branch-character}, the normal monodromy of $\Lcal_\alpha^{\mathrm{act}}$ around $E$ is
\begin{equation*}
\exp(-2\pi iN_E\cdot\alpha)=\exp(-2\pi i\nu_E)=1.
\end{equation*}
Consequently the dual coefficient system $\Lcal_\alpha$ also has trivial normal monodromy. Both systems descend in the tangential direction. Let $E^\circ$ denote $E$ with its intersections with the other components of the total transform removed. The descended coefficient system is $\Lcal_{E,\alpha}$ as defined above, and the local expression \eqref{eq:coefficient-form-crossing} defines the Poincar\'e residue
\begin{equation}\label{eq:coefficient-poincare-residue}
R_{E,\alpha}:=\Res_E\pi^*(u_\alpha\omega)
\in\Omega^1(E^\circ;\Lcal_{E,\alpha}).
\end{equation}
The corresponding integration pairing uses homology with coefficients in $\Lcal_{E,\alpha}^\vee$.

\begin{theorem}\label{thm:direct-residue-obstruction}
Let $E$ be a component of the total transform and let $\alpha\in\Q_{>0}^r$ satisfy $N_E\cdot\alpha=\nu_E$. If
\begin{equation}\label{eq:nonzero-residue-class}
[R_{E,\alpha}]\neq0\quad\text{in}\quad H^1_{\mathrm{dR}}(E^\circ,\Lcal_{E,\alpha}),
\end{equation}
then
\begin{equation*}
u_\alpha\notin\D_{X,0}(hu_\alpha).
\end{equation*}
\end{theorem}

\begin{proof}
Suppose, to the contrary, that $u_\alpha=P(hu_\alpha)$ for some $P\in\D_{X,0}$. By the Green formula,
\begin{equation*}
u_\alpha\omega=(hu_\alpha)P^\dagger\omega+\nabla_{\Lcal_\alpha}\eta_P(hu_\alpha,\omega).
\end{equation*}
Pulling this identity back to $Y$ gives
\begin{equation}\label{eq:pulled-back-green-residue}
\pi^*(u_\alpha\omega)=\pi^*\bigl((hu_\alpha)P^\dagger\omega\bigr)+\nabla_{\pi^{-1}\Lcal_\alpha}\pi^*\eta_P(hu_\alpha,\omega).
\end{equation}
By \eqref{eq:one-component-green-order}, the first term on the right has order at least $n_E-1\geq0$ along $E$, and therefore has zero Poincar\'e residue. On the other hand, the equality $N_E\cdot\alpha=\nu_E$ implies that $\pi^*(u_\alpha\omega)$ has a simple pole in the normal direction to $E$. It follows from \eqref{eq:pulled-back-green-residue} that the second term $\nabla_{\pi^{-1}\Lcal_\alpha}\pi^*\eta_P(hu_\alpha,\omega)$ also has at most a simple normal pole along $E$. The form $\pi^*\eta_P(hu_\alpha,\omega)$ is meromorphic with finite pole order, so \Cref{thm:tube-residue-exactness} applies on a punctured tubular neighborhood of $E^\circ$. Hence
\begin{equation*}
\left[\Res_E\left(\nabla_{\pi^{-1}\Lcal_\alpha}\pi^*\eta_P(hu_\alpha,\omega)\right)\right]=0\quad\text{in}\quad H^1_{\mathrm{dR}}(E^\circ,\Lcal_{E,\alpha}).
\end{equation*}
Taking the Poincar\'e residue of \eqref{eq:pulled-back-green-residue}, the first term on the right contributes zero, while the left-hand side gives $R_{E,\alpha}$. Therefore
\begin{equation*}
[R_{E,\alpha}]=\left[\Res_E\left(\nabla_{\pi^{-1}\Lcal_\alpha}\pi^*\eta_P(hu_\alpha,\omega)\right)\right]=0,
\end{equation*}
contradicting \eqref{eq:nonzero-residue-class}.
\end{proof}

Thus a nonzero residue class shows that the inclusion
\begin{equation*}
\D_{X,0}(hu_\alpha)\subset\D_{X,0}u_\alpha
\end{equation*}
is strict. Equivalently, \Cref{thm:direct-residue-obstruction} shows that the class of $u_\alpha$ in
\begin{equation*}
\frac{\D_{X,0}u_\alpha}{\D_{X,0}(hu_\alpha)}
\end{equation*}
is nonzero whenever $[R_{E,\alpha}]$ is nonzero. Combining this with \Cref{prop:affine-point-lifting} gives the following consequence.

\begin{corollary}\label{cor:residue-gives-Bernstein-point}
Let $E$ be a component of the total transform and let $\alpha\in\Q_{>0}^r$ satisfy $N_E\cdot\alpha=\nu_E$. If
\begin{equation*}
[R_{E,\alpha}]\neq0\quad\text{in}\quad H^1_{\mathrm{dR}}(E^\circ,\Lcal_{E,\alpha}),
\end{equation*}
then
\begin{equation*}
-\alpha\in Z(B_{F,0}).
\end{equation*}
\end{corollary}

\begin{proof}
By \Cref{thm:direct-residue-obstruction}, $u_\alpha\notin\D_{X,0}(hu_\alpha)$. The conclusion follows from \Cref{prop:affine-point-lifting}.
\end{proof}

The local system determined by $F^{-\alpha}$ depends only on the monodromy character and therefore does not distinguish $\alpha$ from its integral translates. The argument above, however, evaluates the Bernstein--Sato equation directly at $s=-\alpha$. Thus a nonzero residue class detects the specified point $-\alpha\in Z(B_{F,0})$. If such nonvanishing holds for a Zariski-dense set of rational points in $W_H^+$, then \Cref{thm:dense-wall-lifting} gives $H\subseteq Z(B_{F,0})$. 

\subsection{Extension across intersections}\label{subsec:residue-one}

Let $E$ and $A$ be distinct irreducible components of the reduced total transform $D$, and let $p\in E\cap A$. In local coordinates $E=(x=0)$ and $A=(y=0)$, \eqref{eq:coefficient-form-crossing} becomes
\begin{equation}\label{eq:residue-one-local-form}
\pi^*(u_\alpha\omega)=U(x,y)x^{-1}y^{\epsilon_A(\alpha)-1}dx\wedge dy,
\end{equation}
where $\epsilon_A(\alpha)=\nu_A-N_A\cdot\alpha$. Hence
\begin{equation}\label{eq:residue-branch-local-form}
R_{E,\alpha}=U(0,y)y^{\epsilon_A(\alpha)-1}dy.
\end{equation}
The displayed scalar residue branch has analytic-continuation factor
\begin{equation}\label{eq:residue-branch-monodromy}
\exp(-2\pi iN_A\cdot\alpha)=\exp(2\pi i\epsilon_A(\alpha)),
\end{equation}
since $\nu_A\in\Z$. The parallel-transport monodromy of its coefficient local system $\Lcal_{E,\alpha}$ is the inverse factor $\exp(-2\pi i\epsilon_A(\alpha))$. If $\epsilon_A(\alpha)=1$, both characters are trivial and the residue form in \eqref{eq:residue-branch-local-form} is regular at $p$; hence the coefficient local system and the residue extend across $p$. We will use this to fill such intersection points without losing a nonzero residue class. If $\epsilon_A(\alpha)=0$, then
\begin{equation*}
R_{E,\alpha}=U(0,y)\frac{dy}{y},
\end{equation*}
which has a logarithmic pole at $p$, although its local monodromy is again trivial.

\begin{proposition}\label{prop:residue-one-injection}
Let $S_1$ be a finite set of intersection points $p\in E\cap A$, with $A\neq E$ an irreducible component of the reduced total transform satisfying $\epsilon_A(\alpha)=1$. Set
\begin{equation*}
E^{\mathrm{eff}}=E^\circ\cup S_1,\quad j:E^\circ\hookrightarrow E^{\mathrm{eff}}.
\end{equation*}
Then $\Lcal_{E,\alpha}$ extends across the points of $S_1$, and restriction induces an injection
\begin{equation}\label{eq:residue-one-injection}
j^*:H^1_{\mathrm{dR}}(E^{\mathrm{eff}},\Lcal_{E,\alpha})\hookrightarrow H^1_{\mathrm{dR}}(E^\circ,\Lcal_{E,\alpha}).
\end{equation}
In particular, a nonzero residue class on $E^{\mathrm{eff}}$ remains nonzero after restriction to $E^\circ$.
\end{proposition}

\begin{proof}
By the twisted de Rham theorem, it is enough to prove the corresponding injectivity for cohomology with local coefficients. Since $\Lcal_{E,\alpha}$ extends across every point of $S_1$ as $\epsilon_A(\alpha)=1$, excision gives
\begin{equation*}
H^k(E^{\mathrm{eff}},E^\circ;\Lcal_{E,\alpha})\simeq\bigoplus_{p\in S_1}H^k(B_p,B_p^*;\C)\otimes(\Lcal_{E,\alpha})_p,
\end{equation*}
where $B_p$ is a small disk centred at $p$ and $B_p^*=B_p\setminus\{p\}$. For a disk and a punctured disk, the relative long exact sequence contains
\begin{equation*}
H^0(B;\C)\xrightarrow{\sim}H^0(B^*;\C)\to H^1(B,B^*;\C)\to H^1(B;\C)=0,
\end{equation*}
so $H^1(B,B^*;\C)=0$. Hence
\begin{equation*}
H^1(E^{\mathrm{eff}},E^\circ;\Lcal_{E,\alpha})=0.
\end{equation*}
The long exact sequence of the pair $(E^{\mathrm{eff}},E^\circ)$ contains
\begin{equation*}
H^1(E^{\mathrm{eff}},E^\circ;\Lcal_{E,\alpha})\to H^1(E^{\mathrm{eff}},\Lcal_{E,\alpha})\xrightarrow{j^*}H^1(E^\circ,\Lcal_{E,\alpha}).
\end{equation*}
Since the first group vanishes, exactness implies that $\ker j^*=0$. Thus $j^*$ is injective.
\end{proof}

If two components $A$ and $B$ meet and satisfy
\begin{equation*}
N_A\cdot\alpha=\nu_A,\quad N_B\cdot\alpha=\nu_B,
\end{equation*}
then the residue along either component has a logarithmic pole at the crossing, so the preceding extension argument does not apply. In this case the Green formula can instead be tested on a small torus around the crossing, leading to the double-residue obstruction below.

\subsection{The double-residue obstruction}

Let $A$ and $B$ be smooth components meeting transversely at $p$. Orient a small real torus $T_{A,B}$ by taking first the positively oriented normal meridian of $A$ and then that of $B$. Let $\Lcal$ be the coefficient local system dual to the actual branch system under consideration, and suppose that its monodromy around both meridians is trivial. Then $\Lcal$ is trivial on the punctured bidisk. After choosing a nonzero flat section of $\Lcal^\vee$, define
\begin{equation}\label{eq:double-residue-torus}
\DRes_{A,B}(\Theta)=\frac{1}{(2\pi i)^2}\int_{T_{A,B}}\Theta,
\end{equation}
for an $\Lcal$-valued meromorphic two-form $\Theta$.
Changing the chosen dual flat section rescales this value by a nonzero constant, so its vanishing or nonvanishing is intrinsic. Equivalently, before choosing such a section, the integral may be viewed as taking values in the one-dimensional coefficient fibre.

In SNC coordinates $A=(x=0)$ and $B=(y=0)$, choose the corresponding flat frame of $\Lcal$. Then
\begin{equation}\label{eq:double-residue-coordinate}
\DRes_{A,B}\left(U(x,y)\frac{dx}{x}\wedge\frac{dy}{y}\right)=U(0,0).
\end{equation}
The torus definition is independent of the choice of local coordinates. Indeed, replacing a local defining equation by a unit multiple changes its logarithmic differential by a regular form and deforms the torus within the complement. Interchanging $A$ and $B$ reverses the chosen ordering and changes the sign, but does not affect whether the double residue vanishes.

\begin{lemma}\label{lem:exact-double-residue}
Let $\eta$ be an $\Lcal$-valued meromorphic one-form with finite pole order along $A+B$. If the monodromy of $\Lcal$ around both normal meridians is trivial, then
\begin{equation}\label{eq:exact-double-residue-zero}
\DRes_{A,B}(\nabla_{\Lcal}\eta)=0.
\end{equation}
\end{lemma}

\begin{proof}
The chosen dual flat section determines a twisted fundamental class of the closed torus $T_{A,B}$. Since $\eta$ is defined in a neighbourhood of this torus, \eqref{eq:twisted-stokes} gives
\begin{equation*}
\int_{T_{A,B}}\nabla_{\Lcal}\eta=\int_{\partial T_{A,B}}\eta=0.
\end{equation*}
The conclusion follows from \eqref{eq:double-residue-torus}.
\end{proof}

\begin{theorem}
\label{thm:direct-double-residue-obstruction}
Let $A$ and $B$ be any two components of the SNC total transform meeting at $p$. If
\begin{equation}\label{eq:two-wall-equalities}
 N_A\cdot\alpha=\nu_A,
 \quad
 N_B\cdot\alpha=\nu_B,
\end{equation}
then
\[
 u_\alpha\notin\D_{X,0}(hu_\alpha).
\]
\end{theorem}

\begin{proof}
By \eqref{eq:actual-branch-character} and \eqref{eq:two-wall-equalities}, the actual-branch monodromies around the normal meridians of $A$ and $B$ are
\begin{equation*}
\exp(-2\pi iN_A\cdot\alpha)=\exp(-2\pi i\nu_A)=1,\quad \exp(-2\pi iN_B\cdot\alpha)=\exp(-2\pi i\nu_B)=1,
\end{equation*}
since $\nu_A,\nu_B\in\Z$. Their inverse characters, which are the monodromies of the coefficient system, are therefore also equal to $1$. In the crossing coordinates of \eqref{eq:coefficient-form-crossing}, the equalities \eqref{eq:two-wall-equalities} give
\begin{equation*}
\pi^*(u_\alpha\omega)=U(x,y)\frac{dx}{x}\wedge\frac{dy}{y},\quad U(0,0)\neq0.
\end{equation*}
Hence, by \eqref{eq:double-residue-coordinate},
\begin{equation}\label{eq:double-log-coefficient-form}
\DRes_{A,B}\bigl(\pi^*(u_\alpha\omega)\bigr)=U(0,0)\neq0.
\end{equation}

Suppose, to the contrary, that $u_\alpha=P(hu_\alpha)$ for some $P\in\D_{X,0}$. The Green formula gives
\begin{equation*}
u_\alpha\omega=(hu_\alpha)P^\dagger\omega+\nabla_{\Lcal_\alpha}\eta_P(hu_\alpha,\omega).
\end{equation*}
Pulling this identity back to $Y$ gives
\begin{equation*}
\pi^*(u_\alpha\omega)=\pi^*\bigl((hu_\alpha)P^\dagger\omega\bigr)+\nabla_{\pi^{-1}\Lcal_\alpha}\pi^*\eta_P(hu_\alpha,\omega).
\end{equation*}
By \eqref{eq:two-component-green-order}, the first term on the right is a holomorphic multiple of
\begin{equation*}
x^{n_A-1}y^{n_B-1}dx\wedge dy.
\end{equation*}
Since $n_A,n_B\geq1$, this term has zero double residue. The second term also has zero double residue by \Cref{lem:exact-double-residue}, since the normal monodromy is trivial around both $A$ and $B$. Applying $\DRes_{A,B}$ to the pulled-back Green formula therefore gives
\begin{equation*}
\DRes_{A,B}\bigl(\pi^*(u_\alpha\omega)\bigr)=0,
\end{equation*}
contradicting \eqref{eq:double-log-coefficient-form}. Thus $u_\alpha\notin\D_{X,0}(hu_\alpha)$.
\end{proof}

\subsection{Crossing types}

The preceding theorem depends only on the SNC structure, the inequalities $n_A,n_B\geq1$, and the common Jacobian factorization \eqref{eq:simultaneous-jacobian}. It therefore applies to all three types of crossings.

\begin{enumerate}[label=\textup{(\roman*)}]
\item At an exceptional--exceptional crossing, the exponents $\nu_A-1$ and $\nu_B-1$ occur simultaneously in the local Jacobian through the relative canonical divisor.
\item At an exceptional--strict crossing, the strict transform has $\nu=1$, so its Jacobian exponent is zero, while its $h$-multiplicity remains positive. If $B$ is the strict transform of a branch $C$, then $n_B=\sum_i\ord_C(f_i)\geq1$.
\item At a strict--strict crossing, the resolution may be locally the identity. Both values of $\nu$ are one, and the pullback of a nowhere vanishing holomorphic two-form has a unit coefficient, so the same argument applies.
\end{enumerate}

The three cases are illustrated by the following examples.

\begin{example}[Exceptional--exceptional]\label{ex:double-exceptional-exceptional}
In a chart obtained after two point blowups, take
\begin{equation*}
X=uv,\quad Y=uv^2,\quad dX\wedge dY=uv^2du\wedge dv.
\end{equation*}
The components $A=(u=0)$ and $B=(v=0)$ have
\begin{equation*}
(N_A,\nu_A)=((1,1),2),\quad (N_B,\nu_B)=((1,2),3).
\end{equation*}
For $F=(X,Y)$ and $\alpha=(1,1)$, both equalities in \eqref{eq:two-wall-equalities} hold, and
\begin{equation*}
\pi^*(u_\alpha dX\wedge dY)=\frac{du}{u}\wedge\frac{dv}{v}.
\end{equation*}
Thus the double residue is one and $hu_\alpha=1$.
\end{example}

\begin{example}[Exceptional--strict]\label{ex:double-exceptional-strict}
Blow up the origin and use the chart
\begin{equation*}
X=x,\quad Y=xy,\quad dX\wedge dY=xdx\wedge dy.
\end{equation*}
For $F=(X,Y)$ and $\alpha=(1,1)$, the exceptional component $A=(x=0)$ has $(N_A,\nu_A)=((1,1),2)$ and the strict transform $B=(y=0)$ has $(N_B,\nu_B)=((0,1),1)$. Both equalities in \eqref{eq:two-wall-equalities} hold, and
\begin{equation*}
\pi^*(u_\alpha dX\wedge dY)=\frac{dx}{x}\wedge\frac{dy}{y}.
\end{equation*}
The double residue is again one and $hu_\alpha=1$.
\end{example}

\begin{example}[Strict--strict]\label{ex:double-strict-strict}
On $X=\C^2$, take $F=(x,y)$ and $\alpha=(1,1)$. Then
\begin{equation*}
u_\alpha=x^{-1}y^{-1},\quad hu_\alpha=1,\quad u_\alpha dx\wedge dy=\frac{dx}{x}\wedge\frac{dy}{y}.
\end{equation*}
Both strict components have $\nu=1$, and the double residue is one. In this case the noncontainment is also immediate, since a holomorphic differential operator applied to $1$ cannot produce $x^{-1}y^{-1}$.
\end{example}

The cases $\epsilon_A(\alpha)=0$ and $\epsilon_A(\alpha)=1$ play different roles. Suppose that $A$ meets $E$ and that $N_E\cdot\alpha=\nu_E$. If $\epsilon_A(\alpha)=0$, then
\begin{equation*}
R_{E,\alpha}=U(0,y)\frac{dy}{y},
\end{equation*}
and $N_A\cdot\alpha=\nu_A$, so \Cref{thm:direct-double-residue-obstruction} applies at the crossing $E\cap A$. If $\epsilon_A(\alpha)=1$, the residue extends regularly across the intersection, and \Cref{prop:residue-one-injection} applies instead. In both cases the local monodromy of the residue system around the intersection is trivial, but the residue form has different local behaviour.


\section{Topological strong monodromy for plane curves} \label{sec:main-theorem}

In this section we prove the topological multivariable Strong Monodromy Conjecture for plane curves. For a rational point $\alpha$ on a fixed affine hyperplane, we use Blanco's results to obtain a nonzero twisted residue class at a rupture component. Intersections at which $\epsilon_A(\alpha)=1$ are handled by the extension result of \Cref{sec:residue-obstructions}, while the remaining possibilities are excluded by a numerical argument. Crossings are treated by the double-residue obstruction and strict transforms by \Cref{thm:strict-transform-wall}; Zariski density then gives the whole affine hyperplane.

\subsection{Arbitrary branch multiplicities}

Let $g_1,\ldots,g_m$ be the distinct irreducible branches of the reduced divisor $(h=0)$ and write
\begin{equation*}
f_i=u_i\prod_{j=1}^m g_j^{a_{ij}},
 \quad a_{ij}\in\Z_{\geq0},
\end{equation*}
with $u_i$ a unit. For a total-transform component $A$, put $n_{A,j}=\ord_A(g_j\circ\pi)$. Then
\begin{equation}\label{eq:multiplicity-transport}
 N_{A,i}=\sum_{j=1}^m a_{ij}n_{A,j}.
\end{equation}

\begin{proposition}\label{prop:multiplicity-transport}
The resolution formula, the grouped Laurent dichotomy, and the residue obstructions used in the proof of \Cref{thm:main-intro} (=\Cref{thm:main-thm-top}) remain valid for arbitrary nonnegative exponents $a_{ij}$. For the strict transform $C_j$ of $g_j$ one has
\begin{equation*}
N_{C_j}=(a_{1j},\ldots,a_{rj}),
\end{equation*}
and \Cref{thm:strict-transform-wall} applies with this full vector. For every positive integral scalarization $g_b=\prod_i f_i^{b_i}$, the reduced zero divisor and minimal embedded resolution are unchanged and
\begin{equation*}
\ord_A(g_b)=b\cdot N_A.
\end{equation*}
Consequently no entrywise reducedness hypothesis is required in \Cref{thm:main-intro} (=\Cref{thm:main-thm-top}).
\end{proposition}

\begin{proof}
The resolution and its strata depend only on the reduced support. Equation \eqref{eq:multiplicity-transport} gives all resolution vectors, while principal-divisor intersection remains valid with arbitrary coefficients and adjunction is independent of the multiplicities. Hence the resolution identities and grouped Laurent calculation are unchanged. The Green estimates use only $n_A=N_A\cdot\one\geq1$, so powers can only increase the vanishing of the shifted term. The strict-transform assertion is the arbitrary-order-vector case already proved in \Cref{thm:strict-transform-wall}.

Finally,
\begin{equation*}
g_b=\left(\prod_i u_i^{b_i}\right)\prod_jg_j^{\sum_i a_{ij}b_i}.
\end{equation*}
For each branch $g_j$ of the reduced support, some $a_{ij}$ is positive, and every $b_i$ is positive. Hence $\sum_i a_{ij}b_i>0$ for every $j$, so $g_b$ has the same reduced support as $h$. Blanco's results apply to reducible and nonreduced scalar germs, and the residue numbers of the scalarization are still $\nu_D-N_D\cdot\alpha$. The final Bernstein--Sato step evaluates functional equations for the original tuple, so it requires neither squarefreeness nor independence of the entries. 
\end{proof}

\subsection{Scalarization}

We first translate the one-variable resolution results of Blanco \cite{Blanco2026} to the coefficient powers used here. The scalar rupture-divisor mechanism goes back to Loeser \cite[Th\'eor\`eme~III.3.1 and Remarque~III.3.3]{Loeser1988}; the sharper numerical bounds and residue nonvanishing used below come from Blanco's work. We keep the minimal embedded resolution of the reduced support of $h$ fixed throughout, since Blanco's graph-theoretic bounds are stated on that resolution. By \Cref{prop:multiplicity-transport}, every scalarization has the same reduced zero divisor and hence the same minimal embedded resolution.

Fix a rupture component $E$ and a rational point $\alpha\in\Q_{>0}^r$ on the positive part of the dual hyperplane
\begin{equation*}
N_E\cdot\alpha=\nu_E.
\end{equation*}
Choose $q>0$ such that
\begin{equation*}
b=q\alpha\in\Z_{>0}^r
\end{equation*}
and put
\begin{equation*}
g_b=\prod_{i=1}^r f_i^{b_i}.
\end{equation*}
Then $F^{-\alpha}=g_b^{-1/q}$. If $A$ is a component of the total transform, write
\begin{equation*}
M_A=\ord_A(g_b)=qN_A\cdot\alpha.
\end{equation*}
At $E$ one has $M_E=q\nu_E$, and the candidate of a nowhere vanishing top form $\omega_0$ is therefore
\begin{equation*}
\sigma_E(\omega_0)
 =
 -\frac{\nu_E}{M_E}
 =
 -\frac1q.
\end{equation*}
In Blanco's notation, the relative canonical divisor is
\begin{equation*}
K_{Y/X}=\sum_A(k_A-1)A;
\end{equation*}
thus his $k_A$ is our $\nu_A$. His residue number at a component $D$ adjacent to $E$ becomes
\begin{equation}\label{eq:blanco-dictionary}
 M_D\sigma_E(\omega_0)+\nu_D
 =
 \nu_D-N_D\cdot\alpha
 =
 \epsilon_D(\alpha).
\end{equation}
Finally, Blanco's exponent $\mu_D$ is
\begin{equation*}
\mu_D=1-\epsilon_D.
\end{equation*}
Equation~\eqref{eq:blanco-dictionary} identifies the residue numbers themselves, including their affine values.

\begin{lemma}\label{lem:blanco-balance}
For the scalarization $g_b$ and the chosen nowhere vanishing holomorphic two-form $\omega_0$,
\begin{equation*}
\sum_{D\neq E,\, D\cap E\neq \emptyset}\epsilon_D=d_E-2,
\end{equation*}
where the sum runs over the actual total transform components meeting $E$ and $d_E$ is their number. Moreover, $\operatorname{Div}(\pi^*\omega_0)$ has no irreducible component outside the scalar total transform support that meets $E$. Thus Blanco's additional numerator-zero terms are absent, and the leading residue has no extra zero on $E^\circ$.
\end{lemma}

\begin{proof}
Blanco's balance formula~\cite[Lemma~4.1]{Blanco2026} gives
\begin{equation*}
\sum_{j=1}^{m_E}\epsilon_{E,j}(\omega)
 +
 \sum_{\ell}\delta_{E,\ell}(\omega)
 =
 m_E-2.
\end{equation*}
The terms $\delta_{E,\ell}$ record components of $\operatorname{Div}(\pi^*\omega)$ outside the scalar total transform support which meet $E$. Since $\omega_0$ is nowhere vanishing,
\begin{equation*}
\operatorname{Div}(\pi^*\omega_0)=K_{Y/X}.
\end{equation*}
This divisor is exceptional. Every exceptional component maps to $0$, and each $f_i$ vanishes at $0$; hence every exceptional component already lies in the support of $\operatorname{Div}(g_b\circ\pi)$. There are consequently no $\delta$-terms. Blanco's convention also includes auxiliary zero divisors in certain predecessor directions, each with residue $1$. Removing these $m_E-d_E$ auxiliary terms gives the stated identity for the components of the total transform that actually meet $E$.
\end{proof}

\begin{proposition}\label{prop:standard-residue-bounds}
For every $\alpha\in\Q_{>0}^r$ satisfying $N_E\cdot\alpha=\nu_E$, the coefficient residues of $\omega_0$ satisfy
\begin{equation*}
-1\leq\epsilon_D(\alpha)\leq1
 \quad(D\neq E,\ D\cap E\neq\varnothing).
\end{equation*}
\end{proposition}

\begin{proof}
Blanco's form-modification theorem~\cite[Theorem~A]{Blanco2026} produces a holomorphic top form
\[
\omega'=a\omega_0,\quad a\in\mathcal O_{X,0},
\]
such that the candidate pole associated with $E$ is unchanged and all adjacent residue numbers lie in $[-1,1]$. Hence
\[
    -\frac{\nu_E+v_E(a)}{M_E} = -\frac{\nu_E}{M_E} = -\frac1q,
\]
so $v_E(a)=0$. The exceptional valuation $v_E$ is centered at the maximal ideal of $\mathcal O_{X,0}$, and every nonunit has positive $v_E$-value. Thus $a$ is a unit. Its pullback has order zero along every component, so the residue numbers of $\omega'$ are exactly those of $\omega_0$. This proves the bounds for $\omega_0$.
\end{proof}

\subsection{The twisted residue class}

We use the coefficient local system $\Lcal_{E,\alpha}$ introduced in \Cref{sec:residue-obstructions}. A local scalar representative of the residue has the form
\begin{equation*}
y^{\epsilon_D(\alpha)-1}\,dy,
\end{equation*}
and has analytic-continuation factor
\begin{equation}\label{eq:blanco-actual-character}
\chi_{E,\alpha}^{\mathrm{act}}(\gamma_D)=e^{2\pi i\epsilon_D(\alpha)}.
\end{equation}
Accordingly, the parallel-transport monodromy of the coefficient local system $\Lcal_{E,\alpha}$ is $e^{-2\pi i\epsilon_D(\alpha)}$. This is exactly Blanco's convention: in \cite[Section~7.3]{Blanco2026}, his local system $L$ has monodromy $e^{-2\pi i\epsilon_D(\alpha)}$, the multivalued residue form is $L$-valued, and the integration cycles have coefficients in $L^\vee$. Thus Blanco's $L$ identifies directly with $\Lcal_{E,\alpha}$.

Let $E^\circ$ be the complement in $E$ of its intersections with all other irreducible components of the reduced total transform. Recall that the Poincar\'e residue is defined by
\begin{equation*}
R_{E,\alpha}:=\Res_E\pi^*(u_\alpha\omega)
\in\Omega^1(E^\circ;\Lcal_{E,\alpha}).
\end{equation*}
Let $S_1$ be the set of intersection points $p_D\in E\cap D$ for which $\epsilon_D(\alpha)=1$, and put
\begin{equation*}
E^{\mathrm{eff}}=E^\circ\cup S_1=E\setminus\{p_D:\epsilon_D(\alpha)\ne1\}.
\end{equation*}
At the points of $S_1$, both $\Lcal_{E,\alpha}$ and the residue form extend regularly by \Cref{subsec:residue-one}. In the application below every integral adjacent residue is $1$. The remaining punctures are therefore exactly
\begin{equation}\label{eq:blanco-effective-punctures}
S_{\mathrm{eff}}(\alpha):=\{p_D:\epsilon_D(\alpha)\notin\Z\}=\{p_1,\ldots,p_k\},\quad E^{\mathrm{eff}}=E\setminus S_{\mathrm{eff}}(\alpha).
\end{equation}
At each $p_j$, the scalar branch factor is $e^{2\pi i\epsilon_j}\ne1$, equivalently the monodromy of the coefficient local system is $e^{-2\pi i\epsilon_j}\ne1$.

\begin{theorem}\label{thm:blanco-residue-class}
Suppose that every integral adjacent residue equals $1$. Let $S_{\mathrm{eff}}(\alpha)=\{p_1,\ldots,p_k\}$ be the effective puncture set in \eqref{eq:blanco-effective-punctures}, and suppose that
\begin{equation*}
k\geq3,
 \quad
 \sum_{j=1}^k(1-\epsilon_j)=2.
\end{equation*}
Then the coefficient Poincar\'e residue defines a nonzero class
\begin{equation*}
[R_{E,\alpha}]
 \neq0
 \quad\text{in}\quad
 H^1_{\mathrm{dR}}
 \bigl(E^{\mathrm{eff}},\Lcal_{E,\alpha}\bigr).
\end{equation*}
Its restriction to the fully punctured curve is also nonzero:
\begin{equation*}
[R_{E,\alpha}]
 \neq0
 \quad\text{in}\quad
 H^1_{\mathrm{dR}}
 \bigl(E^\circ,\Lcal_{E,\alpha}\bigr).
\end{equation*}
\end{theorem}

\begin{proof}
By \eqref{eq:blanco-effective-punctures}, the effective curve is the punctured projective line $E^{\mathrm{eff}}=E\setminus\{p_1,\ldots,p_k\}$ with $k\geq3$. Each $\epsilon_j$ is nonintegral, so the corresponding coefficient monodromy $e^{-2\pi i\epsilon_j}$ is nontrivial. The displayed equality is precisely the required degree-two balance. Finally, the standard form has no additional zero on $E^\circ$ by \Cref{lem:blanco-balance}; in particular, the leading multivalued residue section to which Blanco applies the Deligne--Mostow criterion \cite[Proposition~2.14]{DeligneMostow1986} is nonzero. Blanco's proposition therefore gives a nonzero twisted de Rham class on $E^{\mathrm{eff}}$ \cite[Proposition~7.1 and Section~7.3]{Blanco2026}. Related residue calculations at rupture components also occur in Blanco's earlier study of complex zeta functions \cite{Blanco2019}.

It remains to identify Blanco's form with the coefficient residue used here. At a generic point of $E$, absorb the local unit in $g_b\circ\pi$ into a choice of normal coordinate $x$ and write
\begin{equation*}
g_b\circ\pi=x^{M_E},
 \quad
 \pi^*\omega_0
 =
 U(x,y)x^{\nu_E-1}\,dx\wedge dy.
\end{equation*}
Since $M_E=q\nu_E$,
\begin{equation*}
g_b^{-1/q}\pi^*\omega_0
 =
 U(x,y)\frac{dx}{x}\wedge dy,
\end{equation*}
and hence
\begin{equation*}
R_{E,\alpha}=U(0,y)\,dy.
\end{equation*}
On the other hand,
\begin{equation*}
\frac{\pi^*\omega_0}{d(g_b\circ\pi)}
 =
 \frac{U(x,y)}{M_E}x^{\nu_E-M_E}\,dy.
\end{equation*}
Consequently Blanco's leading form and the coefficient residue satisfy, on each compatible branch,
\begin{equation}\label{eq:blanco-residue-normalization}
 R_E(\omega_0)=\frac1{M_E}R_{E,\alpha}.
\end{equation}
Coordinate units and sheet choices act on both scalar expressions by the same nonzero transition factor. Since Blanco's $L$ is precisely $\Lcal_{E,\alpha}$, \eqref{eq:blanco-residue-normalization} identifies Blanco's leading form with the coefficient residue as sections of the same local system. Blanco's nonvanishing therefore gives the first assertion. The residue extends regularly across every point of $S_1$; hence \Cref{prop:residue-one-injection} applies to the restriction
\begin{equation*}
H^1_{\mathrm{dR}}
 \bigl(E^{\mathrm{eff}},\Lcal_{E,\alpha}\bigr)
 \to 
 H^1_{\mathrm{dR}}
 \bigl(E^\circ,\Lcal_{E,\alpha}\bigr)
\end{equation*}
and proves that the restricted class is still nonzero.
\end{proof}

The scalarized Blanco input is used only for rupture components. Strict-transform hyperplanes are handled independently by \Cref{thm:strict-transform-wall}, which follows from Proposition~1.3 of~\cite{BvdVVW2024} with shift $\one$ and $c=1$. That
result uses the full order vector
\begin{equation*}
N_C=(\ord_C(f_1),\ldots,\ord_C(f_r)),
\end{equation*}
so it includes arbitrary multiplicities of a branch in any number of tuple entries.

\subsection{Numerical analysis at a rupture component}

We now isolate the numerical argument used at a rupture component $E$. Assume that no component defining the same affine hyperplane $H$ as $E$ meets $E$, and that
\begin{equation*}
R_{E,H}\neq0
\end{equation*}
in $K(H)$. At a point $\alpha\in W_H^+$ for which the adjacent denominators are nonzero,
\begin{equation}\label{eq:selected-component-residue}
 R_{E,H}(-\alpha)
 =
 2-d_E+\sum_{D\neq E,\, D\cap E\neq \emptyset}\frac1{\epsilon_D(\alpha)}.
\end{equation}

Let $\mathcal U_E(\Q)$ be the dense set constructed in \Cref{lem:selected-rupture-density}. At every $\alpha\in\mathcal U_E(\Q)$:
\begin{enumerate}[label=\textup{(\roman*)}]
 \item every $\epsilon_D(\alpha)$ is nonzero;
 \item $R_{E,H}(-\alpha)\neq0$;
 \item every nonconstant residue function is nonintegral.
\end{enumerate}
By \Cref{prop:standard-residue-bounds}, a constant integral residue is one of $-1$, $0$, and $1$. The bounded polytope in \Cref{lem:selected-rupture-density} ensures that only finitely many integral level sets have to be removed.

\begin{remark}\label{rem:pointwise-zero}
For a neighbour $D$ of $E$, the function $\epsilon_D$ may vanish on a proper affine subset of $W_H$ without $D$ defining the same hyperplane as $E$. The latter occurs precisely when
\begin{equation*}
\epsilon_D|_{W_H}\equiv0.
\end{equation*}
The set $\mathcal U_E(\Q)$ avoids the proper zero sets of the nonconstant functions $\epsilon_D$. An identically vanishing restriction would instead place both $D$ and $E$ in $S_H$, contrary to the assumption that no two such components meet in the local fibre.
\end{remark}

\begin{remark}
For $r=1$, $W_H^+$ consists of the single rational point $\alpha=\nu_E/N_E$. All residue functions are constant, and the assumptions already give nonvanishing of the adjacent denominators and of $R_{E,H}(-\alpha)$. The argument below is therefore pointwise.
\end{remark}

\begin{proposition}\label{prop:selected-numerical-exhaustion}
At every point $\alpha\in\mathcal U_E(\Q)$, after filling the intersections for which $\epsilon_D(\alpha)=1$, at least three nonintegral residues remain. They satisfy
\begin{equation*}
\sum_j(1-\epsilon_j)=2.
\end{equation*}
Consequently the hypotheses of \Cref{thm:blanco-residue-class} hold.
\end{proposition}

\begin{proof}
By \Cref{lem:blanco-balance},
\begin{equation}\label{eq:epsilon-balance-selected}
 \sum_{D\neq E,\, D\cap E\neq \emptyset}\epsilon_D=d_E-2.
\end{equation}
There is no zero residue on $\mathcal U_E(\Q)$. Suppose that $\epsilon_{D_0}(\alpha)=-1$ for some component $D_0$ meeting $E$. By \Cref{prop:standard-residue-bounds}, every other adjacent residue satisfies $\epsilon_D(\alpha)\leq1$. Hence \Cref{lem:blanco-balance} gives
\begin{equation*}
d_E-2
   =\sum_{D\neq E,\, D\cap E\neq \emptyset}\epsilon_D(\alpha)
   \leq -1+(d_E-1)
   =d_E-2.
\end{equation*}
Equality therefore holds termwise, so
\begin{equation*}
\epsilon_D(\alpha)=1
  \quad\text{for every }D\neq D_0.
\end{equation*}
Consequently,
\begin{equation*}
R_{E,H}(-\alpha)
   =2-d_E+\sum_{D\neq E,\, D\cap E\neq \emptyset}\frac{1}{\epsilon_D(\alpha)}=2-d_E-1+(d_E-1)=0,
\end{equation*}
contrary to the condition $R_{E,H}(-\alpha)\neq0$. 

All remaining integral residues are therefore $1$. Fill their intersections, let $\ell$ be their number, and denote the nonintegral residues by $x_1,\ldots,x_k$. Equation~\eqref{eq:epsilon-balance-selected} gives
\begin{equation*}
\ell+\sum_{i=1}^k x_i=\ell+k-2,
\end{equation*}
and hence
\begin{equation}\label{eq:effective-balance}
 \sum_{i=1}^k x_i=k-2,
 \quad
 \sum_{i=1}^k(1-x_i)=2.
\end{equation}
The cases $k=0$ and $k=1$ are impossible: they give respectively $0=-2$ and $x_1=-1$. If $k=2$, then $x_1+x_2=0$ and, using $d_E=\ell+2$,
\begin{equation*}
R_{E,H}(-\alpha)
 =
 2-(\ell+2)+\ell+\frac1{x_1}+\frac1{x_2}=0.
\end{equation*}
This again contradicts $R_{E,H}(-\alpha)\neq0$. Therefore $k\geq3$, and Equation~\eqref{eq:effective-balance} is precisely Blanco's degree-two balance.
\end{proof}

\begin{corollary}\label{cor:selected-blanco-class}
At every point of $\mathcal U_E(\Q)$, the coefficient Poincar\'e residue is nonzero on the fully punctured curve:
\begin{equation*}
[R_{E,\alpha}]
 \neq0
 \quad\text{in}\quad
 H^1_{\mathrm{dR}}
 \bigl(E^\circ,\Lcal_{E,\alpha}\bigr).
\end{equation*}
\end{corollary}

\begin{proof}
The proposition shows that all the hypotheses of \Cref{thm:blanco-residue-class} are satisfied. Its second conclusion, together with the injectivity in \Cref{prop:residue-one-injection}, is exactly the asserted nonvanishing on $E^\circ$.
\end{proof}

\begin{remark}
The argument above concerns the value of $R_{E,H}$ at the chosen point $-\alpha$. The vanishing computations used in the excluded cases do not imply that $R_{E,H}$ vanishes identically in $K(H)$.
\end{remark}

\subsection{Completion of the proof}

We now prove the main theorem by considering separately the two cases in the grouped Laurent expansion. A crossing between two components defining the same affine hyperplane is handled by the double-residue obstruction. If no such crossing occurs, the grouped simple-pole coefficient gives either a strict transform or a rupture component, and the latter is treated by the twisted Poincar\'e residue. 

For $r=1$, the conclusion below recovers the classical plane-curve strong statement obtained by combining Loeser's Bernstein roots for rupture and strict-transform divisors with Veys's pole classification \cite{Loeser1988,Veys1995}. Weak multivariable monodromy statements for plane curves were obtained earlier by Nicaise and Budur \cite{Nicaise2004,Budur2015}; the argument here keeps the full affine hyperplane rather than only its exponentiated character.

\begin{theorem}\label{thm:main-thm-top}
For every tuple $F=(f_1,\ldots,f_r)$ of arbitrary nonzero nonunit plane curve germs on a smooth complex surface germ,
\begin{equation*}
\PL\bigl(Z^{\mathrm{top}}_{F,0}\bigr) \subseteq Z(B_{F,0}).
\end{equation*}
\end{theorem}

\begin{proof}
Let $H$ be an irreducible component of the polar locus
\begin{equation*}
\PL(Z^{\mathrm{top}}_{F,0}).
\end{equation*}
By the resolution formula \eqref{eq:surface-zeta}, $H=V(L_A)$ for at least one component $A$ of the total transform. Give it the normalization
\begin{equation*}
H=V(L_H),
 \quad
 L_H(s)=1+\lambda_H\cdot s,
\end{equation*}
of \Cref{lem:positive-wall-normalization}, and let $S_H$ be the set of all total transform components whose affine form defines $H$. We apply the two alternatives of \Cref{thm:grouped-wall-dichotomy}.

Suppose first that two components $A,B\in S_H$ meet at a point of $\pi^{-1}(0)$. At every positive rational dual point
\begin{equation*}
\alpha\in W_H^+\cap\Q^r,
 \quad
 \lambda_H\cdot\alpha=1,
\end{equation*}
\Cref{thm:same-wall-positivity} gives
\begin{equation*}
N_A\cdot\alpha=\nu_A,
 \quad
 N_B\cdot\alpha=\nu_B.
\end{equation*}
The double-residue obstruction \Cref{thm:direct-double-residue-obstruction} therefore yields
\begin{equation*}
u_\alpha\notin\D_{X,0}(hu_\alpha).
\end{equation*}
The rational points of $W_H^+$ are Zariski dense in $W_H$ when $r\geq2$ by \Cref{lem:bounded-polytope-density}; for $r=1$ the hyperplane is the corresponding single rational point. Hence \Cref{thm:dense-wall-lifting} gives
\begin{equation*}
H\subseteq Z(B_{F,0}).
\end{equation*}

Suppose next that any two components in $S_H$ are disjoint over $\pi^{-1}(0)$. By \Cref{thm:grouped-wall-dichotomy}, there is a component $A\in S_H$ with $R_{A,H}\neq0$. By \Cref{thm:selected-component-classification}, $A$ is either a strict transform or a rupture component. In the strict-transform case, $H=V(N_A\cdot s+1)$, and the desired inclusion follows directly from \Cref{thm:strict-transform-wall}.

It remains to treat the case $A=E$ with $E$ rupture. Let
\begin{equation*}
\mathcal U_E(\Q)\subset W_H^+\cap\Q^r
\end{equation*}
be the Zariski-dense set from \Cref{lem:selected-rupture-density}. For every $\alpha\in\mathcal U_E(\Q)$, \Cref{prop:selected-numerical-exhaustion}, \Cref{thm:blanco-residue-class}, and \Cref{prop:residue-one-injection} give
\begin{equation*}
[R_{E,\alpha}]\neq0
 \quad\text{in}\quad
 H^1_{\mathrm{dR}}(E^\circ,\Lcal_{E,\alpha});
\end{equation*}
equivalently, this is \Cref{cor:selected-blanco-class}. The residue obstruction \Cref{thm:direct-residue-obstruction} now gives
\begin{equation*}
u_\alpha\notin\D_{X,0}(hu_\alpha)
 \quad(\alpha\in\mathcal U_E(\Q)).
\end{equation*}
\Cref{prop:affine-point-lifting} places every $-\alpha$ in $Z(B_{F,0})$, and \Cref{thm:dense-wall-lifting} yields
\begin{equation*}
H\subseteq Z(B_{F,0}).
\end{equation*}

We have therefore shown that every irreducible component $H$ of $\PL(Z^{\mathrm{top}}_{F,0})$ is contained in $Z(B_{F,0})$. Taking the union over all irreducible components gives
\begin{equation*}
\PL\bigl(Z^{\mathrm{top}}_{F,0}\bigr)
  \subseteq Z(B_{F,0}).
\end{equation*}
\Cref{prop:multiplicity-transport} shows that the same argument covers powers, repeated branches, common factors, and dependent tuple entries. The point case $r=1$ was included in both density steps above.
\end{proof}

\begin{remark}
The proof establishes a set-theoretic inclusion for the reduced polar locus of the topological zeta function. It neither compares pole orders with scheme-theoretic multiplicities of Bernstein components nor identifies the remaining components of $Z(B_{F,0})$. The non-equivariant motivic strengthening is proved in \Cref{sec:motivic-comparison}.
\end{remark}

\subsection{An outer rupture example}

We illustrate the rupture-component case of the proof on an example for which all adjacent residues are nonintegral, so no extension across intersection points is needed.

Consider
\begin{equation*}
F=(f_1,y), \quad f_1 = y^4-2x^3y^2-4x^5y+x^6-x^7.
\end{equation*}
Five point blow-ups resolve the reduced union. With $C_1$ the strict transform of $(f_1=0)$ and $C_2$ that of $(y=0)$, the dual graph is of the form 
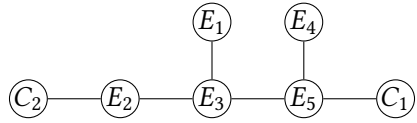
\begin{equation*}
    \begin{tikzpicture}[
            every node/.style={circle, draw, minimum size=5mm, inner sep=0pt},
            node distance=5mm and 7mm
        ]
        \node (E3) {\small $E_3$};
        \node[left=of E3] (E2) {\small $E_2$};
        \node[left=of E2] (C2) {\small $C_2$};
        \node[right=of E3] (E5) {\small $E_5$};
        \node[right=of E5] (C1) {\small $C_1$};
        \node[above=of E3] (E1) {\small $E_1$};
        \node[above=of E5] (E4) {\small $E_4$};
        \draw (C2) -- (E2) -- (E3) -- (E5) -- (C1);
        \draw (E3) -- (E1);
        \draw (E5) -- (E4);
    \end{tikzpicture}
    \end{equation*}
The exceptional numerical data are
\begin{equation*}
\begin{array}{c@{\quad}c@{\quad}c}
\toprule
 A&N_A&\nu_A\\
\midrule
 E_1&(4,1)&2\\
 E_2&(6,2)&3\\
 E_3&(12,3)&5\\
 E_4&(13,3)&6\\
 E_5&(26,6)&11\\
\bottomrule
\end{array}
\end{equation*}
Indeed, the first two blow-ups give the vectors $(4,1)$ and $(6,2)$. Blowing up their intersection, where the strict transform of $f_1$ has multiplicity two, gives $(12,3)$; resolving the ensuing smooth tangency gives $(13,3)$; and blowing up the resulting triple point gives $(26,6)$. Thus the outer rupture component $E=E_5$ has
\begin{equation*}
N_E=(26,6),
 \quad
 \nu_E=11.
\end{equation*}
The three neighbours relevant to this component have data
\begin{equation*}
\begin{array}{c@{\quad}c@{\quad}c}
\toprule
 D&N_D&\nu_D\\
\midrule
 E_3&(12,3)&5\\
 E_4&(13,3)&6\\
 C_1&(1,0)&1\\
\bottomrule
\end{array}
\end{equation*}
where $C_1$ is the indicated strict-transform neighbour. In particular, $d_E=3$. The two strict-transform vectors are $N_{C_1}=(1,0)$ and $N_{C_2}=(0,1)$. Together with the exceptional table, they show that no other resolution component defines the same normalized affine hyperplane as $E_5$. Its positive dual hyperplane is
\begin{equation}\label{eq:pc004-positive-wall}
 26\alpha_1+6\alpha_2=11.
\end{equation}
Write $a=\alpha_1$. Then
\begin{equation*}
\alpha_2=\frac{11-26a}{6},
 \quad
 0<a<\frac{11}{26}.
\end{equation*}
The three adjacent coefficient residues are
\begin{equation}\label{eq:pc004-residues}
 \epsilon_3=a-\frac12,
 \quad
 \epsilon_4=\frac12,
 \quad
 \epsilon_1=1-a.
\end{equation}
Every residue in Equation~\eqref{eq:pc004-residues} is nonintegral on the open positive segment, and
\begin{equation*}
(1-\epsilon_3)+(1-\epsilon_4)+(1-\epsilon_1)=2.
\end{equation*}
The component residue is also nonzero there:
\begin{equation*}
\begin{aligned}
 R_{E,H}(-\alpha)
 &=
 2-3+\frac1{a-\frac12}+\frac1{\frac12}+\frac1{1-a}\\
 &=
 \frac{a(\frac32-a)}
 {(a-\frac12)(1-a)}
 \neq0.
 \end{aligned}
\end{equation*}
Because $E_5$ is the only component defining this normalized affine hyperplane, the normalization $L_H=L_E/\nu_E$ gives
\begin{equation*}
C_{-1}(H)=\frac1{11}R_{E,H}\neq0.
\end{equation*}
Hence $26s_1+6s_2+11=0$ is an actual topological polar hyperplane.

The hypotheses of \Cref{thm:blanco-residue-class} therefore hold without filling any intersection point. For every rational point of the positive segment defined by Equation~\eqref{eq:pc004-positive-wall}, the twisted coefficient residue class is nonzero. \Cref{thm:direct-residue-obstruction} gives
\begin{equation*}
F^{-\alpha}\notin\D_{X,0}(hF^{-\alpha}),
\end{equation*}
and \Cref{prop:affine-point-lifting} places $-\alpha$ in $Z(B_{F,0})$. Rational points of the open segment are Zariski dense in the affine line, so \Cref{thm:dense-wall-lifting} yields
\begin{equation*}
V(26s_1+6s_2+11)
 \ \subseteq\
 Z(B_{F,0}).
\end{equation*}
Thus the one-variable scalarization supplies the nonzero twisted residue class, while the direct cyclic obstruction keeps the exact affine hyperplane.


\section{Motivic zeta functions for plane curves}\label{sec:motivic-comparison}

We now strengthen the main theorem from the topological zeta function to a non-equivariant motivic zeta function. The coefficient ring is
\begin{equation*}
\Mmot:=K_0(\operatorname{Var}_{\C})[\Lef^{-1}],
\end{equation*}
without passing to its dimensional completion. We use only the ordinary Grothendieck classes of the resolution strata, with neither a $\widehat\mu$-action nor a finite monodromic cover included in the invariant considered here. This distinction is essential for the cancellation identities below. The Grothendieck ring itself is not a domain \cite{Poonen2002}; our arguments do not require either it or its localization to be a domain. Throughout the paper, the term motivic zeta function refers to this non-equivariant version.

\subsection{The resolution expression}

With the resolution notation of \Cref{sec:setup}, put $\mathbf T^N=\prod_iT_i^{N_i}$ for $N=(N_1,\cdots,N_r)$ and define
\begin{equation}\label{eq:motivic-resolution}
 \motzeta(\mathbf T)=\Lef^{-2}\sum_{\varnothing\ne I\subseteq J}[D_I^\circ\cap\pi^{-1}(0)]\prod_{A\in I}\frac{(\Lef-1)\Lef^{-\nu_A}\mathbf T^{N_A}}{1-\Lef^{-\nu_A}\mathbf T^{N_A}}.
\end{equation}
Each inverse denotes its geometric series, so \eqref{eq:motivic-resolution} lies in $\Mmot[\![T_1,\ldots,T_r]\!]$. For algebraic germs this is the local-fibre, non-equivariant image of the usual motivic construction and resolution formula \cite[Theorem~2.2.1]{DenefLoeser1998}. We normalize motivic measure so that a stable cylinder represented at jet level $k$ on a surface has class $[C]\Lef^{-2(k+1)}$; this accounts for the global factor $\Lef^{-2}$, and the corresponding factor in dimension $n$ is $\Lef^{-n}$. Guibert's multivariable formula is stated in the relative Grothendieck ring of varieties over $\mathbb G_m^r\times X_0$. After base change to the fibre over $0\in X_0$ and then forgetting the remaining structural morphism, the torus bundle satisfies $[U_I|_0]=(\Lef-1)^{|I|}[D_I^\circ\cap\pi^{-1}(0)]$, giving the non-equivariant formula above up to this displayed global normalization \cite[D\'efinition~4.1.1 and Th\'eor\`eme~4.1.2]{Guibert2002}. For arbitrary holomorphic plane germs, \eqref{eq:motivic-resolution} is the definition; its independence of the embedded resolution is proved below. Multiplication by the unit $\Lef^{-2}$ has no effect on the order that we define.

For $a\in\Z$ and $b\in\Z_{\geq0}^r\setminus\{0\}$, set
\begin{equation*}
d_{a,b}=1-\Lef^a\mathbf T^b,
 \quad H(a,b)=\{s\in\C^r:b\cdot s-a=0\}.
\end{equation*}
Let $S_r$ be the multiplicative monoid generated by all $d_{a,b}$.

\begin{lemma}\label{lem:motivic-formal-embedding}
The canonical map
\begin{equation*}
S_r^{-1}\Mmot[\mathbf T]\to \Mmot[\![\mathbf T]\!]
\end{equation*}
is injective.
\end{lemma}

\begin{proof}
The constant coefficient of $d_{a,b}$ is one, and its inverse in the formal-series ring is $\sum_{k\geq0}\Lef^{ak}\mathbf T^{kb}$. If $d_{a,b}G=0$, comparison in increasing total $\mathbf T$-degree starts with coefficient one and recursively gives $G=0$. Every element of $S_r$ therefore acts injectively. Since the polynomial ring embeds coefficientwise, so does the localization. 
\end{proof}

Write $\mathscr R_r=S_r^{-1}\Mmot[\mathbf T]\subset\Mmot[\![\mathbf T]\!]$. Fix a rational affine hyperplane $H$ and $q\geq0$. Let $\mathcal G_{H,q}$ consist of all finite reciprocal products
\begin{equation*}
\prod_{(a,b)\in S}d_{a,b}^{-1}
\end{equation*}
for which at most $q$ members of the finite multiset $S$, counted with multiplicity, satisfy $H(a,b)=H$. Define the fixed submodule
\begin{equation}\label{eq:motivic-NX-module}
 \mathscr P_{H,\le q}=\sum_{G\in\mathcal G_{H,q}}\Mmot[\mathbf T]G\subset\Mmot[\![\mathbf T]\!].
\end{equation}

\begin{definition}[Multivariable Nicaise--Xu order]\label{def:motivic-NX-order}
For $Z\in\mathscr R_r$, set
\begin{equation*}
\ord_H^{\mathrm{mot}}(Z)=\min\{q\geq0:Z\in\mathscr P_{H,\le q}\}.
\end{equation*}
The non-equivariant motivic polar locus associated with this order is
\begin{equation*}
\PL_{\mathrm{mot}}(Z)=\bigcup_{\ord_H^{\mathrm{mot}}(Z)>0}H.
\end{equation*}
\end{definition}

The minimum exists because every rational presentation gives a finite upper bound. It is independent of the presentation because membership is tested in the fixed submodule \eqref{eq:motivic-NX-module} of the fixed formal-series ring. Equality $H(a,b)=H(a',b')$ means equality of affine zero sets, equivalently positive rational proportionality after normalization; equality of the raw pairs is not required. For $r=1$, this is the definition of Nicaise--Xu \cite[Remark~3.7]{NicaiseXu2016}. The replacement of a scalar candidate value by a rational affine hyperplane is the several-variable extension introduced here. In particular, the resolution factor $1-\Lef^{-\nu_A}\mathbf T^{N_A}$ corresponds exactly to the affine hyperplane $V(L_A)$.

\subsection{Resolution independence}

For every total-transform component put
\begin{equation*}
q_A=\Lef^{-\nu_A}\mathbf T^{N_A},
 \quad \Phi_A=\frac{(\Lef-1)q_A}{1-q_A}.
\end{equation*}

\begin{proposition}\label{prop:motivic-blowup-identities}
The local expression \eqref{eq:motivic-resolution} is unchanged by blowing up either a smooth point of one total-transform component or a crossing of two components.
\end{proposition}

\begin{proof}
If a smooth point of $A$ is blown up and $B$ is the new exceptional curve, then $N_B=N_A$, $\nu_B=\nu_A+1$, and $q_B=\Lef^{-1}q_A$. The old point contribution $\Phi_A$ is replaced by the open exceptional stratum $\mathbb C^1$ and its crossing with the strict transform of $A$. Direct substitution gives
\begin{equation}\label{eq:motivic-smooth-blowup}
 \Lef\Phi_B+\Phi_A\Phi_B=\Phi_A,
\end{equation}
because $\Lef+\Phi_A=(\Lef-q_A)/(1-q_A)$ and $\Phi_B=(\Lef-1)q_A/(\Lef-q_A)$.

If a crossing $A\cap C$ is blown up, then $N_B=N_A+N_C$, $\nu_B=\nu_A+\nu_C$, and $q_B=q_Aq_C$. The old crossing contribution is replaced by the open exceptional stratum $\mathbb G_m$ and its two crossings, and
\begin{equation}\label{eq:motivic-crossing-blowup}
 (\Lef-1)\Phi_B+\Phi_A\Phi_B+\Phi_C\Phi_B=\Phi_A\Phi_C.
\end{equation}
After multiplication by $(1-q_A)(1-q_C)(1-q_Aq_C)$, this is an identity in $\Z[\Lef,q_A,q_C]$. Both calculations are explicit polynomial identities and include crossings involving strict transforms.
\end{proof}

\begin{corollary}\label{cor:motivic-resolution-independent}
For holomorphic plane germs, \eqref{eq:motivic-resolution} is independent of the embedded resolution. Consequently its multivariable Nicaise--Xu orders are intrinsic.
\end{corollary}

\begin{proof}
Any two embedded resolutions of a plane divisor admit a common refinement by point blowups at SNC centers. Centers on the relevant local total transform are of the two types in \Cref{prop:motivic-blowup-identities}; centers away from the local fibre have no local contribution.
\end{proof}

\subsection{Cancellation for exceptional components of valency one or two}

Let $E\simeq\mathbb P^1$ be exceptional, let $d=d_E$ be its valency in the graph of the reduced total transform, including strict-transform neighbours, and put $a=-E^2>0$. Every neighbour meets $E$ at a distinct transverse point, so $[E^\circ]=\Lef+1-d$. The identities of \Cref{lem:exceptional-neighbour-identities} imply
\begin{equation}\label{eq:motivic-neighbor-product}
 \prod_{D\neq E,\, D\cap E\neq \emptyset}q_D=\Lef^{2-d}q_E^a.
\end{equation}
The contribution of $E$ together with the crossing strata adjacent to it is
\begin{equation}\label{eq:motivic-complete-star}
 Z_E^\star=\Lef^{-2}\Phi_E\left(\Lef+1-d+\sum_{D\neq E,\, D\cap E\neq \emptyset}\Phi_D\right).
\end{equation}

\begin{theorem}\label{thm:motivic-nonrupture-cancellation}
If $d=1$ or $d=2$, the factor $1-q_E$ cancels from this contribution. More precisely, if $d=1$ and $D$ is the unique neighbour, then $q_D=\Lef q_E^a$ and
\begin{equation}\label{eq:motivic-valency-one-factorization}
 Z_E^\star=\Lef^{-1}(\Lef-1)q_E\frac{1+q_E+\cdots+q_E^{a-1}}{1-q_D}.
\end{equation}
If $d=2$ and the neighbours are $D_1,D_2$, then $q_{D_1}q_{D_2}=q_E^a$ and
\begin{equation}\label{eq:motivic-valency-two-factorization}
 Z_E^\star=\Lef^{-2}(\Lef-1)^2q_E\frac{1+q_E+\cdots+q_E^{a-1}}{(1-q_{D_1})(1-q_{D_2})}.
\end{equation}
These identities include every configuration in which one or both neighbours are strict transforms and allow arbitrary branch multiplicities.
\end{theorem}

\begin{proof}
For $d=1$, substitute $[E^\circ]=\Lef$ and $q_D=\Lef q_E^a$ into \eqref{eq:motivic-complete-star}:
\begin{equation*}
\begin{aligned}
 Z_E^\star&=\Lef^{-2}\frac{(\Lef-1)q_E}{1-q_E}\left(\Lef+\frac{(\Lef-1)q_D}{1-q_D}\right)\\
 &=\Lef^{-1}(\Lef-1)q_E\frac{1-q_E^a}{(1-q_E)(1-q_D)}.
 \end{aligned}
\end{equation*}
The polynomial identity $1-q_E^a=(1-q_E)(1+q_E+\cdots+q_E^{a-1})$ gives \eqref{eq:motivic-valency-one-factorization}.

For $d=2$, one has $[E^\circ]=\Lef-1$ and $q_{D_1}q_{D_2}=q_E^a$. A common-denominator calculation gives
\begin{equation*}
\begin{aligned}
 Z_E^\star&=\Lef^{-2}\frac{(\Lef-1)q_E}{1-q_E}\left(\Lef-1+\sum_{i=1}^2\frac{(\Lef-1)q_{D_i}}{1-q_{D_i}}\right)\\
 &=\Lef^{-2}(\Lef-1)^2q_E\frac{1-q_{D_1}q_{D_2}}{(1-q_E)(1-q_{D_1})(1-q_{D_2})},
 \end{aligned}
\end{equation*}
and the same polynomial factorization proves \eqref{eq:motivic-valency-two-factorization}. Both factorizations are polynomial identities before passing to $\Mmot$, and hence remain valid in $\Mmot$.
\end{proof}

These identities are the motivic counterpart of the cancellations of valency-one and valency-two exceptional components that underlie the scalar plane-curve pole analysis; compare \cite{Veys1995,Veys1997}.

Repeated branches and common factors do not create repeated strict-transform components. In the notation of \Cref{prop:multiplicity-transport}, the strict transform of $g_j$ occurs once and carries the arbitrary vector $(a_{1j},\ldots,a_{rj})$, which is already allowed in \eqref{eq:motivic-neighbor-product} and \Cref{thm:motivic-nonrupture-cancellation}.

\subsection{Components defining the same affine hyperplane}

Because $\nu_A>0$, two components define the same affine hyperplane precisely when $N_A/\nu_A=N_B/\nu_B$, equivalently when their affine forms agree after normalization.

\begin{lemma}\label{lem:motivic-hyperplane-propagation}
Let $E$ be exceptional. If $d_E=1$, its unique neighbour does not define $V(L_E)$. If $d_E=2$ and one neighbour defines $V(L_E)$, then the other neighbour does as well. Consequently a nontrivial chain of valency-two exceptional curves defining one affine hyperplane terminates at a strict transform or an exceptional component of valency at least three, and never at a valency-one exceptional curve.
\end{lemma}

\begin{proof}
Restrict the affine identity \eqref{eq:affine-neighbour-identity} to $V(L_E)$. For $d_E=1$, the unique neighbour restricts to the constant $-1$ and therefore does not vanish identically. For $d_E=2$, the two restrictions sum to zero; if one vanishes identically, then so does the other. The total-transform graph is a finite tree: a smooth-point blowup attaches a leaf and a crossing blowup subdivides an edge. Iterating the two-sided rule reaches an endpoint, and the first assertion excludes a valency-one exceptional endpoint.
\end{proof}

\begin{proposition}\label{prop:motivic-denominator-deletion}
Fix an affine hyperplane $H$ occurring among the resolution factors which, on the minimal embedded resolution, is defined by no strict transform and no exceptional component of valency at least three. Then
\begin{equation*}
\ord_H^{\mathrm{mot}}(\motzeta)=0.
\end{equation*}
\end{proposition}

\begin{proof}
Every component defining $H$ is exceptional of valency one or two. No two such components are adjacent, since \Cref{lem:motivic-hyperplane-propagation} would continue their common hyperplane to a strict transform or an exceptional component of valency at least three. The corresponding contributions of these components and their adjacent crossing strata are therefore disjoint. Apply \Cref{thm:motivic-nonrupture-cancellation} simultaneously. Each replacement uses only neighbour denominators, and no neighbour defines $H$, again by propagation. The resulting finite expression has no $H$-denominator and belongs to $\mathscr P_{H,\le0}$.
\end{proof}

\subsection{Finite denominator presentations and positive integral rays}

The arguments in this subsection are formal and do not use the surface assumption. They will also be used in arbitrary dimension in \Cref{sec:higher-motivic-orders}.

For $m=(m_1,\ldots,m_r)\in\Z_{>0}^r$, define
\begin{equation*}
\rho_m:\Mmot[\![\mathbf T]\!]\to \Mmot[\![T]\!],
 \quad T_i\mapsto  T^{m_i}.
\end{equation*}
This is well-defined because a fixed $T$-degree has only finitely many preimages in $\Z_{\geq0}^r$ when every $m_i$ is positive. Put $g_m=\prod_i f_i^{m_i}$.

\begin{lemma}\label{lem:motivic-ray-identities}
One has
\begin{align}
 \motzeta(T^{m_1},\ldots,T^{m_r})&=Z^{\mathrm{mot}}_{g_m,0}(T),\label{eq:motivic-ray-identity}\\
 Z^{\mathrm{top}}_{F,0}(m_1t,\ldots,m_rt)&=Z^{\mathrm{top}}_{g_m,0}(t).\label{eq:topological-ray-identity}
\end{align}
\end{lemma}

\begin{proof}
Since every $m_i$ is positive, $g_m=\prod_{i=1}^r f_i^{m_i}$ and $h$ have the same reduced zero divisor. Thus the fixed resolution $\pi:Y\to X$ also resolves $g_m$. The total-transform components, the strata $D_I^\circ\cap\pi^{-1}(0)$, and the numbers $\nu_A$ are unchanged. For every $A\in J$, the multiplicity of $g_m\circ\pi$ along $A$ is $\ord_A(g_m\circ\pi)=\sum_{i=1}^r m_iN_{A,i}=N_A\cdot m$.

Under the substitution $T_i=T^{m_i}$, each factor in the motivic resolution formula becomes
\[
\rho_m\left(\frac{(\mathbb L-1)\mathbb L^{-\nu_A}\mathbf T^{N_A}}{1-\mathbb L^{-\nu_A}\mathbf T^{N_A}}\right)
=
\frac{(\mathbb L-1)\mathbb L^{-\nu_A}T^{N_A\cdot m}}{1-\mathbb L^{-\nu_A}T^{N_A\cdot m}}.
\]
This is exactly the corresponding factor for $g_m$. The stratum classes and the normalization factor $\mathbb L^{-2}$ are unchanged, so the two resolution expressions agree term by term. This proves \eqref{eq:motivic-ray-identity}.

Similarly, the substitution $s_i=m_it$ gives $L_A(m_1t,\ldots,m_rt)=(N_A\cdot m)t+\nu_A$. Since the Euler characteristics of the strata are unchanged, the topological resolution expressions also agree term by term, proving \eqref{eq:topological-ray-identity}. Neither argument requires $g_m$ to be reduced or the entries of $F$ to be pairwise coprime.
\end{proof}

\begin{lemma}\label{lem:motivic-finite-witness}
If $Z\in\mathscr P_{H,\le q}$, then there is an equality
\begin{equation} \label{eq:motivic-finite}
Z=\sum_{\ell=1}^MP_\ell(\mathbf T)\prod_{(a,b)\in S_\ell}d_{a,b}^{-1}
\end{equation}
in which every finite multiset $S_\ell$ contains at most $q$ pairs defining $H$.
\end{lemma}

\begin{proof}
The sum in \eqref{eq:motivic-NX-module} is an algebraic sum, so each of its elements belongs to the sum of finitely many of the indicated cyclic submodules. Since $Z\in\mathscr P_{H,\le q}$, we may therefore write $Z=Z_1+\cdots+Z_M$, where each $Z_\ell$ belongs to a cyclic submodule generated by a product
\[
\prod_{(a,b)\in S_\ell}d_{a,b}^{-1}.
\]
By the definition of $\mathscr P_{H,\le q}$, each $S_\ell$ is a finite multiset containing at most $q$ pairs defining $H$. For each $\ell$, membership in this cyclic submodule gives a polynomial $P_\ell(\mathbf T)$ such that $Z_\ell=P_\ell(\mathbf T)\prod_{(a,b)\in S_\ell}d_{a,b}^{-1}$. Substituting these expressions into the finite sum for $Z$ gives \eqref{eq:motivic-finite} with the required bound on every $S_\ell$.
\end{proof}

Represent $H$ as $H(a_H,b_H)$ with $a_H\ne0$. Under the ray $m$, its scalar value is $t_H(m)=a_H/(b_H\cdot m)$. A second pair $(a,b)$ has the same specialized value precisely when
\begin{equation}\label{eq:motivic-collision-equation}
 (a_Hb-ab_H)\cdot m=0.
\end{equation}

\begin{lemma}\label{lem:motivic-positive-ray}
Let $H=H(a_H,b_H)$ with $a_H\ne0$. For finitely many pairs $(a,b)$ satisfying $H(a,b)\ne H$, there exists $m\in\Z_{>0}^r$ satisfying none of the corresponding equations \eqref{eq:motivic-collision-equation}.
\end{lemma}

\begin{proof}
For each such pair $(a,b)$, the vector $a_Hb-ab_H$ is nonzero. Indeed, if $a_Hb-ab_H=0$, then the affine equations $b_H\cdot s-a_H=0$ and $b\cdot s-a=0$ are proportional and hence define the same hyperplane, contrary to $H(a,b)\ne H$. Thus each equation \eqref{eq:motivic-collision-equation} defines a proper rational hyperplane in the $m$-space.

For $r>1$, set
\[
m(t)=(1,t,t^2,\ldots,t^{r-1}).
\]
For every nonzero vector $c\in\Q^r$, the function $c\cdot m(t)$ is a nonzero polynomial in $t$, and hence vanishes for only finitely many positive integers $t$. Since there are only finitely many pairs $(a,b)$ under consideration, we may choose a positive integer $t$ outside the union of these finite root sets. Then $m(t)\in\Z_{>0}^r$ satisfies none of the equations \eqref{eq:motivic-collision-equation}.

For $r=1$, each equation is a nonzero homogeneous linear equation in $m$, so it has no positive solution.
\end{proof}

\begin{lemma}\label{lem:motivic-witness-specialization}
Let $H=H(a_H,b_H)$ with $a_H\ne0$, and let $Z\in\mathscr P_{H,\le q}$. Fix a finite presentation of $Z$ as in Lemma~\ref{lem:motivic-finite-witness}. Choose $m\in\Z_{>0}^r$ avoiding \eqref{eq:motivic-collision-equation} for every denominator in this presentation that does not define $H$. Then the scalar series $\rho_m(Z)$ has one-variable Nicaise--Xu order at most $q$ at $t_H(m)$.
\end{lemma}

\begin{proof}
Applying $\rho_m$ to the chosen finite presentation gives
\[
\rho_m(Z)=\sum_{\ell=1}^M P_\ell(T^{m_1},\ldots,T^{m_r})\prod_{(a,b)\in S_\ell}(1-\mathbb L^aT^{b\cdot m})^{-1}.
\]
Since $b\cdot m>0$, every specialized denominator is of the allowed one-variable form and has candidate value $a/(b\cdot m)$. If $H(a,b)=H$, this value equals $t_H(m)$. For every other pair in the presentation, equality would imply \eqref{eq:motivic-collision-equation}, which is excluded by the choice of $m$.

Thus a specialized denominator has candidate value $t_H(m)$ precisely when the original pair defines $H$. Each $S_\ell$ contains at most $q$ such pairs, counted with multiplicity, and every specialized numerator is a polynomial in $T$. The displayed equality therefore gives a one-variable denominator presentation with at most $q$ factors having candidate value $t_H(m)$ in each product. The conclusion follows from the definition of the one-variable Nicaise--Xu order.
\end{proof}

The order of choices matters: the finite denominator presentation is fixed before the positive integral ray is selected. Every hyperplane used below comes from a resolution factor, hence has $a_H=-\nu_A\ne0$ in the convention above. Whenever a ray is used below, we also avoid the finitely many specialization-coincidence equations coming from distinct resolution hyperplanes.

\begin{lemma}\label{lem:motivic-rational-restriction}
Suppose that the polar divisor of $R(s)\in\C(s_1,\ldots,s_r)$ is supported on finitely many affine hyperplanes and that $R$ is regular at the generic point of $H$. Choose $m\in\Z_{>0}^r$ such that $t_H(m)m$ lies on no polar hyperplane of $R$ distinct from $H$. Then $R(tm)$ is regular at $t=t_H(m)$.
\end{lemma}

\begin{proof}
Write $R=P/Q$ with coprime polynomials $P,Q\in\C[s_1,\ldots,s_r]$. By the hypothesis on the polar divisor, we may write
\[
Q(s)=c\prod_{j=1}^k\ell_j(s)^{e_j},
\]
where $c\in\C^\times$, each $e_j$ is a positive integer, and the $\ell_j$ are nonconstant affine linear polynomials defining the distinct polar hyperplanes of $R$. The case of constant $Q$ is included by taking $k=0$.

Since $R$ is regular at the generic point of $H$, none of the $\ell_j$ defines $H$. The choice of $m$ therefore gives $\ell_j(t_H(m)m)\ne0$ for every $j$. Hence $Q(t_H(m)m)\ne0$, so $R(tm)=P(tm)/Q(tm)$ is regular at $t=t_H(m)$.
\end{proof}

\subsection{Scalar and multivariable comparison on a surface}

We first establish the one-variable comparison needed for the positive-ray argument. It combines the notion of motivic pole order of Nicaise--Xu \cite[Remark~3.7]{NicaiseXu2016} with Veys's description of the poles of the topological zeta function of a plane curve \cite{Veys1995}.

\begin{proposition}\label{prop:scalar-motivic-topological-order}
Let $g$ be an arbitrary nonzero nonunit holomorphic plane germ, possibly nonreduced. If $s_0$ is a pole of $Z^{\mathrm{top}}_{g,0}$ of order $p$, then $p\in\{1,2\}$ and
\begin{equation*}
\ord_{s_0}^{\mathrm{mot}}Z^{\mathrm{mot}}_{g,0}=p.
\end{equation*}
\end{proposition}

\begin{proof}
We first prove
\begin{equation}\label{eq:scalar-topological-motivic-inequality}
 \ord_{s_0}Z^{\mathrm{top}}_{g,0}\leq\ord_{s_0}^{\mathrm{mot}}Z^{\mathrm{mot}}_{g,0}.
\end{equation}
Suppose that $Z^{\mathrm{mot}}_{g,0}$ has a finite presentation in which each reciprocal product contains at most $q$ denominators with candidate value $s_0$. By \Cref{lem:motivic-formal-embedding}, its equality with the fixed resolution expression may be read in the common rational localization generated by all denominators occurring in either finite expression. Apply the diagonal Hodge--Deligne realization \cite[\S4.2]{DenefLoeser1998}
\begin{equation*}
E_\Delta:\Mmot\to \Z[u,u^{-1}],
 \quad [V]\mapsto  E_c(V;u,u),
 \quad E_\Delta(\Lef)=u^2,
\end{equation*}
extend it to the corresponding target localization by inverting the images of these finitely many denominators, and only then put $u=1+h$, $\ell=\log(1+h)$, and $T=u^{-2s}$. The resulting equality lies in $\C(s)((h))$.

For $\delta=a-bs$, the exact factorization
\begin{equation*}
1-u^{2\delta}=-2\delta\ell\,U(\delta,\ell),
 \quad U(\delta,\ell)=\frac{e^{2\delta\ell}-1}{2\delta\ell},
 \quad U(0,0)=1,
\end{equation*}
has $U\in\C[\delta][\![\ell]\!]^\times$, and $\ell/h$ is a unit in $\C[\![h]\!]$. Therefore every coefficient of the inverse of a denominator whose candidate value is different from $s_0$ is regular at $s_0$, whereas a denominator with candidate value $s_0$ contributes at most one factor $(s-s_0)^{-1}$. The specialized polynomial numerators are also coefficientwise regular at $s_0$. Since the presentation is finite and each reciprocal product contains at most $q$ denominators with candidate value $s_0$, every fixed $h$-coefficient has pole order at most $q$. 

On the resolution side,
\begin{equation*}
\lim_{u\to1}(u^2-1)\frac{u^{-2(N_As+\nu_A)}}{1-u^{-2(N_As+\nu_A)}}=\frac1{N_As+\nu_A},
\end{equation*}
and $E_c(D_I^\circ\cap\pi^{-1}(0);1,1)=\chi_c(D_I^\circ\cap\pi^{-1}(0))$ for every stratum. The plane-resolution strata at hand are finite sets of points or punctured smooth projective curves; for these strata $\chi_c=\chi$, so this is the convention used in \eqref{eq:top-zeta}. Thus the constant $h$-coefficient is exactly $Z^{\mathrm{top}}_{g,0}(s)$; the global factor $\Lef^{-2}$ specializes to one. This proves \eqref{eq:scalar-topological-motivic-inequality}.

For the reverse inequality, use the minimal embedded resolution. Veys's plane-curve classification \cite[Theorems~4.2--4.3]{Veys1995} gives $p\in\{1,2\}$ and states that $p=2$ precisely when two total-transform components carrying the value $s_0$ meet in the local fibre. If $p=1$, no two components carrying $s_0$ meet in the local fibre, so every contributing summand of \eqref{eq:motivic-resolution} contains at most one target denominator. If $p=2$, every nonempty SNC stratum on a surface contains at most two components. In either case the displayed resolution expression gives a denominator presentation of order at most $p$. Combining this with \eqref{eq:scalar-topological-motivic-inequality} proves equality.
\end{proof}

We use the following topological consequence of the results already proved in \Cref{sec:setup}. On the minimal resolution, a rational affine hyperplane is an actual topological polar hyperplane precisely when it is defined by a strict transform or a rupture component. Indeed, one direction follows from the grouped dichotomy, \Cref{thm:selected-component-classification}, and \Cref{lem:motivic-hyperplane-propagation}. Conversely, choose a positive integral ray separating all distinct resolution hyperplanes. Veys's scalar theorem makes the corresponding strict-transform or rupture value a scalar pole, and \Cref{lem:motivic-ray-identities,lem:motivic-rational-restriction} show that the original multivariable function could not have been regular along the hyperplane. Its generic order is two exactly when two components defining it meet in the local fibre; the positive coefficient in \Cref{thm:same-wall-positivity} proves the forward assertion, and otherwise every contributing summand has at most one target denominator.

\begin{theorem}\label{thm:plane-motivic-comparison}
For every rational affine hyperplane $H\subset\C^r$,
\begin{equation}\label{eq:plane-motivic-order-equality}
 \ord_H^{\mathrm{mot}}(\motzeta)=\ord_H(Z^{\mathrm{top}}_{F,0}).
\end{equation}
Both sides are at most two. Consequently
\begin{equation*}
\PL_{\mathrm{mot}}(\motzeta)=\PL(Z^{\mathrm{top}}_{F,0}).
\end{equation*}
\end{theorem}

\begin{proof}
Let us work on the minimal embedded resolution and write $p=\ord_H(Z^{\mathrm{top}}_{F,0})$.

Suppose first that $p\in\{1,2\}$ and that $\ord_H^{\mathrm{mot}}(\motzeta)\leq p-1$. Fix a finite presentation by \Cref{lem:motivic-finite-witness}. After fixing it, choose a positive integral ray for which no denominator not defining $H$ in the presentation and no distinct resolution hyperplane specializes to the target scalar value. By \Cref{lem:motivic-ray-identities,lem:motivic-witness-specialization}, $Z^{\mathrm{mot}}_{g_m,0}$ has scalar Nicaise--Xu order at most $p-1$ at the target value. The scalarization has the same reduced support and minimal resolution, scalar multiplicities $N_A\cdot m$, and unchanged discrepancies. The incidence characterization in the preceding paragraph and the avoidance of these coincidences make its scalar topological order exactly $p$. This contradicts \Cref{prop:scalar-motivic-topological-order}. Hence the motivic order is at least $p$.

For the reverse inequality, consider three cases. Suppose first that $p=0$. If no resolution component defines $H$, the motivic resolution formula contains no denominator defining $H$. Otherwise, the topological classification shows that no strict transform or rupture component defines $H$, and \Cref{prop:motivic-denominator-deletion} gives an expression with no denominator defining $H$. Thus the motivic order is zero. If $p=1$, no two components defining $H$ meet in the local fibre; hence every contributing SNC stratum contains at most one target component, even if several pairwise disjoint components define the same affine hyperplane. Formula \eqref{eq:motivic-resolution} gives a denominator presentation of order at most one. If $p=2$, every nonempty surface SNC stratum contains at most two components, so the same expression gives a denominator presentation of order at most two. These upper bounds and the preceding lower bound prove \eqref{eq:plane-motivic-order-equality}. Both orders are at most two. Taking the union of all rational affine hyperplanes with positive order gives the equality of the polar loci.
\end{proof}

\begin{corollary}\label{cor:plane-motivic-smc}
For every tuple of arbitrary nonzero nonunit holomorphic plane germs,
\begin{equation*}
\PL_{\mathrm{mot}}(\motzeta)=\PL(Z^{\mathrm{top}}_{F,0})\subseteq Z(B_{F,0}).
\end{equation*}
\end{corollary}

\begin{proof}
Combine \Cref{thm:plane-motivic-comparison,thm:main-thm-top}.
\end{proof}

\begin{remark}
The motivic Strong Monodromy statement above concerns the series \eqref{eq:motivic-resolution}, which carries no $\widehat\mu$-action, with poles measured by the multivariable Nicaise--Xu order of \Cref{def:motivic-NX-order}. We do not treat the $\widehat\mu$-equivariant or monodromic motivic zeta function, nor do we compare motivic pole order with scheme-theoretic multiplicity in $B_{F,0}$.
\end{remark}


\section{Higher dimensional obstructions}
\label{sec:higher-obstructions}

In this section, $(X,0)$ is a smooth complex germ of arbitrary dimension $n$. The entries of $F=(f_1,\ldots,f_r)$ are nonzero nonunit holomorphic germs and need not be reduced. We continue to use the notation $h=\prod_i f_i$, $u_\alpha=F^{-\alpha}$, and $B_{F,0}$ for the multivariable Bernstein--Sato ideal defined in \Cref{sec:setup}. Let $\pi:Y\to X$ be a log resolution of the reduced support of $(h=0)$, and write the reduced total transform as $D=\bigcup_{a\in J}D_a$. Put
\begin{equation*}
N_a=(\ord_{D_a}(f_i\circ\pi))_i,
 \quad n_a=N_a\cdot\one,
 \quad \nu_a=1+\ord_{D_a}(K_{Y/X}),
\end{equation*}
with $\nu_a=1$ for a strict transform. Thus $n_a\geq1$.

\subsection{Ordered coefficient residues on SNC strata}

Let $I=(a_1,\ldots,a_k)$ be an ordered tuple of distinct elements of $J$, and define
\[
D_I^\circ=
\left(\bigcap_{j=1}^k D_{a_j}\right)
\setminus
\bigcup_{b\in J\setminus\{a_1,\ldots,a_k\}}D_b.
\]
Fix a connected component of $D_I^\circ$ that meets the local fibre $\pi^{-1}(0)$. Fix $\alpha\in\Q_{>0}^r$ and suppose 
\begin{equation}\label{eq:iterated-zero-index}
 N_{a_j}\cdot\alpha=\nu_{a_j}\quad(1\leq j\leq k).
\end{equation}
Let $\Lcal_{I,\alpha}^{\mathrm{act}}$ denote the tangential local system obtained by descending the actual branches of $u_\alpha$. Their normal monodromy around $D_{a_j}$ is $e^{-2\pi iN_{a_j}\cdot\alpha}=1$ under \eqref{eq:iterated-zero-index}, so this descent is defined. We use
\[
 \Lcal_{I,\alpha}:=\bigl(\Lcal_{I,\alpha}^{\mathrm{act}}\bigr)^\vee
\]
as the coefficient local system for ordered residues. Scalar branches of $\pi^*(u_\alpha\omega)$ are the local representatives of $\Lcal_{I,\alpha}$-valued forms, and integration cycles carry coefficients in $\Lcal_{I,\alpha}^\vee=\Lcal_{I,\alpha}^{\mathrm{act}}$.

Choose SNC coordinates $(x_1,\ldots,x_k,y_1,\ldots,y_{n-k})$ such that $D_{a_j}=(x_j=0)$, with the order of the normal coordinates compatible with the ordering of $I$. In a local flat frame of the coefficient system, a logarithmic top form has the expression
\begin{equation}\label{eq:ordered-log-expansion}
\Theta=
\frac{dx_1}{x_1}\wedge\cdots\wedge\frac{dx_k}{x_k}\wedge\beta(x,y),
\end{equation}
where $\beta(x,y)$ is a holomorphic tangential $(n-k)$-form. Define
\[
\Res_I(\Theta)=\beta(0,y).
\]
This is the ordered, coefficient-valued form of the classical logarithmic residue construction \cite{Deligne1970}.

\begin{lemma}\label{lem:ordered-residue-conventions}
The residue $\Res_I(\Theta)$ is a well-defined $\Lcal_{I,\alpha}$-valued $(n-k)$-form on $D_I^\circ$, independent of the choice of local defining equations and tangential coordinates. For $\sigma\in S_k$, let $I_\sigma=(a_{\sigma(1)},\ldots,a_{\sigma(k)})$. Then $\Res_{I_\sigma}(\Theta)=\sgn(\sigma)\Res_I(\Theta)$. The orientations of the normal torus associated with $I_\sigma$ and $I$ differ by the same sign.
\end{lemma}

\begin{proof}
Choose a local flat frame of the coefficient system. Two local defining equations for the same component have the form $x'_j=v_jx_j$, where $v_j$ is a holomorphic unit. Hence $dx'_j/x'_j=dx_j/x_j+dv_j/v_j$. The form $dv_j/v_j$ is regular. Since $dx_i=x_i\,dx_i/x_i$, any contribution of this regular term to the coefficient of the full logarithmic normal wedge is divisible by at least one normal coordinate $x_i$. Such contributions vanish on $D_I^\circ$. Thus the restriction $\beta(0,y)$ in \eqref{eq:ordered-log-expansion} is independent of the local defining equations.

A change of tangential coordinates induces the usual transformation of $\beta(0,y)$ as an $(n-k)$-form on $D_I^\circ$. Moreover, the actual branch system descends because its normal characters are trivial, and its dual coefficient system descends with it to $\Lcal_{I,\alpha}$. The local residues are compatible with its transition functions and therefore glue to an $\Lcal_{I,\alpha}$-valued form on $D_I^\circ$.

Finally, an interchange of two adjacent components changes the sign of the wedge of their logarithmic differentials. It also reverses the product orientation of the corresponding two circle factors. Since every permutation is a product of adjacent transpositions, both the residue and the torus orientation change by $\sgn(\sigma)$.
\end{proof}

Throughout this section, $H_\bullet$ denotes ordinary singular homology with local coefficients. After shrinking the local representative, $D_I^\circ$ has finite CW type: remove compatible open tubes around the boundary divisor from the compact piece over a small closed ball and retract radially onto the resulting compact manifold with corners. The twisted de Rham theorem and cellular duality over $\C$ then give a perfect pairing
\begin{equation}\label{eq:iterated-perfect-pairing}
 H^q_{\mathrm{dR}}(D_I^\circ,\Lcal_{I,\alpha})\times
 H_q(D_I^\circ,\Lcal_{I,\alpha}^\vee)\to \C.
\end{equation}
Indeed, on a finite CW model the cochain complex with coefficients in $\Lcal_{I,\alpha}$ is the vector-space dual of the chain complex with the dual system. Every singular cycle has compact image even though the open stratum need not be compact. See \cite{Deligne1970} for the de Rham description of regular flat connections and \cite[Chapter~VI]{Bredon1993} for homology with local coefficients.

\subsection{Pullbacks}

We use the Green formula \Cref{prop:green-formula}, which was proved in arbitrary dimension. We next record the local form of the relevant pullbacks along an SNC stratum satisfying \eqref{eq:iterated-zero-index}.

\begin{lemma}
\label{lem:iterated-simultaneous-orders}
Let $\omega$ be a nowhere vanishing holomorphic top form near $0$. At every point of $D_I^\circ$, there are SNC coordinates in which
\begin{align}
 \pi^*\omega&=J(x,y)\prod_{j=1}^kx_j^{\nu_{a_j}-1}
 dx_1\wedge\cdots\wedge dx_k\wedge
 dy_1\wedge\cdots\wedge dy_{n-k},\label{eq:iterated-jacobian}\\
 f_i\circ\pi&=u_i(x,y)\prod_{j=1}^kx_j^{N_{a_j,i}},
 \label{eq:iterated-functions}
\end{align}
where $J$ and the $u_i$ are holomorphic units. Under
\eqref{eq:iterated-zero-index},
\begin{align}
 \pi^*(u_\alpha\omega)&=
 U(x,y)\bigwedge_{j=1}^k\frac{dx_j}{x_j}
 \wedge dy_1\wedge\cdots\wedge dy_{n-k},
 \label{eq:iterated-leading-form}\\
 \pi^*((hu_\alpha)\theta)&=
 \left(\prod_{j=1}^kx_j^{n_{a_j}-1}\right)
 (\text{a holomorphic top form})
 \label{eq:iterated-shifted-order}
\end{align}
where $U$ is a holomorphic unit, and the second identity holds for every holomorphic top form $\theta$ near $0$.
\end{lemma}

\begin{proof}
Fix $p\in D_I^\circ$ and choose SNC coordinates $(x_1,\ldots,x_k,y_1,\ldots,y_{n-k})$ near $p$ such that $D_{a_j}=(x_j=0)$ and no other component of $D$ meets this neighbourhood. The vanishing orders of $f_i\circ\pi$ along these components give \eqref{eq:iterated-functions}, with each $u_i$ a holomorphic unit. Since $\omega$ is nowhere vanishing, the divisor of $\pi^*\omega$ is $K_{Y/X}$. Its multiplicity along $D_{a_j}$ is $\nu_{a_j}-1$, which gives \eqref{eq:iterated-jacobian} with $J$ a holomorphic unit.

After shrinking the coordinate neighbourhood, choose holomorphic branches of $u_i^{-\alpha_i}$ and set $U=J\prod_i u_i^{-\alpha_i}$. This is a holomorphic unit. Combining \eqref{eq:iterated-jacobian} and \eqref{eq:iterated-functions}, the exponent of $x_j$ in the coefficient of $\pi^*(u_\alpha\omega)$ is
\[
\nu_{a_j}-1-\sum_i\alpha_iN_{a_j,i}
=\nu_{a_j}-1-N_{a_j}\cdot\alpha
=-1,
\]
where the last equality follows from \eqref{eq:iterated-zero-index}. This proves \eqref{eq:iterated-leading-form}.

Finally, let $\theta$ be any holomorphic top form near $0$. Since $\omega$ is nowhere vanishing, $\theta=c\omega$ for a holomorphic function $c$. Moreover,
\[
h\circ\pi=\left(\prod_i u_i\right)\prod_{j=1}^k x_j^{n_{a_j}}.
\]
Thus $\pi^*((hu_\alpha)\theta)=(c\circ\pi)(h\circ\pi)\pi^*(u_\alpha\omega)$. The resulting expression contains the factor $\prod_{j=1}^k x_j^{n_{a_j}-1}$, and all remaining factors are holomorphic. So \eqref{eq:iterated-shifted-order} follows.
\end{proof}

\subsection{The iterated Gysin map and tube formula}

Fix a connected component $Z$ of $D_I^\circ$ that meets $\pi^{-1}(0)$. Put $L_j=N_{D_{a_j}/Y}|_Z$. The SNC normal bundle is the ordered sum $\bigoplus_j L_j$. Choose Hermitian metrics on the $L_j$, and let $S(L_j)$ denote the corresponding unit circle bundle. Set
\[
T_j=S(L_1)\times_Z\cdots\times_Z S(L_j),
\quad T_0=Z,
\]
and let $p_j:T_j\to T_{j-1}$ be the projection forgetting the last circle factor. The order $(a_1,\ldots,a_k)$ orients the fibers of $q_I:T_k\to Z$. For any compact singular cycle in $Z$, sufficiently small positive radius functions identify the restriction of $T_k$ over its image with an embedded normal torus in the complement of $D$. Such radii exist by compactness and may be joined through positive choices; all integrations below use this embedded representative.

\begin{proposition}
\label{prop:ordered-torus-umkehr}
For every local system $\mathcal V$ on $Z$ there is a natural map
\begin{equation*}
q_I^!:H_m(Z,\mathcal V)\to 
 H_{m+k}(T_k,q_I^*\mathcal V).
\end{equation*}
On a trivializing open set it sends a class represented by $c$ to the ordered product
\begin{equation*}
[S^1_{a_1}]\times\cdots\times[S^1_{a_k}]\times c.
\end{equation*}
It is independent of the metrics and radii. If the ordering of $I$ is changed by $\sigma\in S_k$, then under the induced identification of the torus bundles, the map $q_I^!$ changes by the factor $\sgn(\sigma)$.
\end{proposition}

\begin{proof}
Let $r_j:T_j\to Z$ be the natural projection, with $r_0=\mathrm{id}_Z$, and put $\mathcal V_j=r_j^*\mathcal V$. The pullback $r_{j-1}^*L_j$ is a complex line bundle over $T_{j-1}$. Let $B_j$ be its closed unit disk bundle, with projection $b_j:B_j\to T_{j-1}$. Its unit circle bundle is naturally identified with $T_j$, and the restriction of $b_j$ to $T_j$ is $p_j$.

The complex orientation gives the homology Thom isomorphism
\[
\operatorname{Th}_j:
H_d(T_{j-1},\mathcal V_{j-1})
\xrightarrow{\sim}
H_{d+2}(B_j,T_j;b_j^*\mathcal V_{j-1});
\]
see \cite[Chapter~VI, Section~11]{Bredon1993} for the Thom construction. On a trivializing open set, we use the convention $\operatorname{Th}_j(\xi)=[D^2,S^1]\times\xi$, where $[D^2,S^1]$ is the relative fundamental class of the positively oriented disk. Let $\partial_j$ be the boundary map of the pair $(B_j,T_j)$. Since $b_j^*\mathcal V_{j-1}|_{T_j}=\mathcal V_j$, its target is $H_{d+1}(T_j,\mathcal V_j)$. Define
\[
p_j^!=(-1)^{j-1}\partial_j\circ\operatorname{Th}_j,
\quad
q_I^!=p_k^!\circ\cdots\circ p_1^!.
\]
Each $p_j^!$ raises the homological degree by one, so $q_I^!$ has the required source and target.

We next check the local formula. With the chosen Thom convention, $\partial_j\circ\operatorname{Th}_j$ places the new positive circle before all existing factors. After the first $j-1$ steps, moving this circle past the preceding $j-1$ circle factors contributes the sign $(-1)^{j-1}$. The sign in the definition of $p_j^!$ therefore puts the new circle after these factors and before the original cycle. Induction gives the ordered product in the statement.

Both the Thom isomorphism and the boundary map are natural in the coefficient system, so $q_I^!$ is natural in $\mathcal V$. Different metrics and positive radii give disk and circle bundles identified by fibrewise radial rescaling. These identifications preserve the orientations and commute with the Thom and boundary maps. Hence the resulting map is independent of these choices.

Finally, it is enough to consider an interchange of two adjacent components. In the fibrewise product of their disk bundles, the two Thom maps commute because the disk fibres have even dimension, whereas the two boundary maps anticommute. Thus an adjacent interchange changes $q_I^!$ by $-1$. The formula for an arbitrary permutation follows.
\end{proof}

The following is called the iterated tube formula. 

\begin{theorem}
\label{thm:iterated-tube-formula}
Let $\Theta$ be a closed coefficient-valued top form on a punctured tubular neighbourhood, logarithmic in the $I$-normal directions. For $\gamma\in H_{n-k}(Z,\Lcal_{I,\alpha}^\vee)$,
\begin{equation}\label{eq:iterated-tube-formula}
 \int_{q_I^!\gamma}\Theta=(2\pi i)^k
 \int_\gamma\Res_I(\Theta).
\end{equation}
\end{theorem}

\begin{proof}
Let $c$ be a smooth singular cycle representing $\gamma$, and let $K\subset Z$ be its compact support. Choose a smooth tubular map over a neighbourhood $V$ of $K$ that induces the identity on the normal bundle. For sufficiently small $\varepsilon>0$, let $\iota_\varepsilon:T_k|_V\to Y\setminus D$ be the embedding obtained by multiplying all chosen normal radii by $\varepsilon$. The coefficient systems are identified along this radial homotopy. The resulting tube cycles are homologous in the punctured neighbourhood, so closedness of $\Theta$ implies that $\int_{q_I^!\gamma}\iota_\varepsilon^*\Theta$ is independent of $\varepsilon$.

The normal monodromies are trivial, so $\iota_\varepsilon^*\Theta$ has coefficients in $q_I^*\Lcal_{I,\alpha}$. Let $(q_I)_*$ denote integration along the torus fibres, oriented in the order $(a_1,\ldots,a_k)$, and put
\[
\eta_\varepsilon=(q_I)_*(\iota_\varepsilon^*\Theta).
\]
This is an $\Lcal_{I,\alpha}$-valued $(n-k)$-form on $V$. Integration along the fibres is adjoint to the Thom and boundary construction in \Cref{prop:ordered-torus-umkehr}. With the normal circle factors before the base, this gives
\[
\int_{q_I^!\gamma}\iota_\varepsilon^*\Theta
=\int_c\eta_\varepsilon.
\]

Choose local SNC coordinates $(x,y)$ and a local flat frame of the coefficient system. Let $\varphi_j$ be the positive angular coordinate on the $j$th normal circle. Since the tubular map induces the identity on the normal bundle, its shrinking satisfies
\[
\iota_\varepsilon^*\left(\frac{dx_j}{x_j}\right)
=
i\,d\varphi_j+q_I^*(d\log\lambda_j)+O(\varepsilon),
\qquad
\iota_\varepsilon^*\beta
=
q_I^*\beta(0,y)+O(\varepsilon),
\]
where $\lambda_j$ is the positive smooth radius function in this trivialization. The estimates are uniform on compact subsets. By \eqref{eq:ordered-log-expansion}, integration over the $k$-torus fibres selects the term containing all $k$ angular differentials. Hence
\[
\lim_{\varepsilon\to0}\eta_\varepsilon
=
\left(
\int_{(S^1)^k}
i\,d\varphi_1\wedge\cdots\wedge i\,d\varphi_k
\right)
\beta(0,y)
=
(2\pi i)^k\Res_I(\Theta).
\]

These local computations agree on overlaps by \Cref{lem:ordered-residue-conventions}. The transition functions of the coefficient system and its dual cancel in the integration pairing. Since $K$ is compact, finitely many such charts suffice, and the uniform convergence allows the limit to pass through integration over $c$. The tube integral is independent of $\varepsilon$, so
\[
\int_{q_I^!\gamma}\Theta
=
\lim_{\varepsilon\to0}\int_c\eta_\varepsilon
=
(2\pi i)^k\int_c\Res_I(\Theta).
\]
The residue of a closed logarithmic form is closed, so the last integral depends only on $\gamma$. This proves \eqref{eq:iterated-tube-formula}.
\end{proof}

\begin{theorem}\label{thm:iterated-tube-exactness}
If $\Theta=\nabla\eta$ on the punctured tubular neighbourhood and $\Theta$ is logarithmic in the $I$-normal directions, then
\begin{equation*}
[\Res_I(\Theta)]=0
 \quad\text{in}\quad
 H^{n-k}_{\mathrm{dR}}(Z,\Lcal_{I,\alpha}).
\end{equation*}
\end{theorem}

\begin{proof}
Let $[\gamma]\in H_{n-k}(Z,\Lcal_{I,\alpha}^\vee)$. The class $q_I^![\gamma]$ is represented at positive normal radii, where $\eta$ is defined, and it is a cycle. Twisted Stokes therefore gives
\begin{equation*}
\int_{q_I^![\gamma]}\nabla\eta=0.
\end{equation*}
Applying \Cref{thm:iterated-tube-formula} with $\Theta=\nabla\eta$ yields
\begin{equation*}
0=(2\pi i)^k\left\langle[\Res_I(\Theta)],[\gamma]\right\rangle.
\end{equation*}
This holds for every $[\gamma]$, so the perfect pairing \eqref{eq:iterated-perfect-pairing} gives $[\Res_I(\Theta)]=0$.
\end{proof}

\subsection{The iterated residue obstruction}

\begin{theorem}
\label{thm:iterated-residue-obstruction}
Let $I=(a_1,\ldots,a_k)$ be a nonempty ordered SNC stratum satisfying \eqref{eq:iterated-zero-index}, and let $\omega$ be a local nowhere vanishing holomorphic top form. If
\begin{equation*}
[R_{I,\alpha}]
 :=[\Res_I\pi^*(u_\alpha\omega)]\ne0
 \quad\text{in}\quad
 H^{n-k}_{\mathrm{dR}}(D_I^\circ,\Lcal_{I,\alpha}),
\end{equation*}
then
\begin{equation*}
u_\alpha\notin\D_{X,0}(hu_\alpha),
 \quad \text{and}\quad 
 -\alpha\in Z(B_{F,0}).
\end{equation*}
\end{theorem}

\begin{proof}
Suppose, to the contrary, that $u_\alpha=P(hu_\alpha)$ for some $P\in\D_{X,0}$. By \Cref{prop:green-formula},
\begin{equation*}
u_\alpha\omega=(hu_\alpha)P^\dagger\omega+\nabla\eta_P(hu_\alpha,\omega).
\end{equation*}
Pulling this identity back to $Y$ gives
\begin{equation*}
\pi^*(u_\alpha\omega)=\pi^*((hu_\alpha)P^\dagger\omega)+\nabla\pi^*\eta_P(hu_\alpha,\omega).
\end{equation*}
The form $P^\dagger\omega$ is holomorphic because $P$ has holomorphic coefficients. Hence \eqref{eq:iterated-shifted-order} shows that the first term on the right has zero ordered $k$-fold residue. By \eqref{eq:iterated-leading-form}, the left-hand side is logarithmic in all $k$ normal directions. It follows that the differential of the pulled-back concomitant is also logarithmic in those directions, even if the concomitant itself has higher finite poles. The residue class of this exact term is zero by \Cref{thm:iterated-tube-exactness}. Taking ordered residue classes in the pulled-back Green formula therefore gives $[R_{I,\alpha}]=0$, contradicting the hypothesis. The second assertion follows from \Cref{prop:affine-point-lifting}; its proof is independent of the dimension of $X$ and applies verbatim here.
\end{proof}

For $k=1$ this recovers the mechanism underlying \Cref{thm:direct-residue-obstruction}; for $n=k=2$ it recovers the local double-residue obstruction. In the coordinate normal-crossing model $F=(x_1,\ldots,x_k)$ with $\alpha=(1,\ldots,1)$,
\[
\Res_I\left(u_\alpha\,dx_1\wedge\cdots\wedge dx_n\right)
=
dx_{k+1}\wedge\cdots\wedge dx_n.
\]
In particular, when $k=n$, the residue is the nonzero scalar $1$. Generalized nearby cycles and Bernstein supports along logarithmic strata are studied from a different $\D$-module perspective in \cite{Wu2026}; the criterion above instead detects a specified affine parameter directly from a nonzero coefficient residue. In the hypersurface case $r=1$ and $k=1$, Chen--Zuo independently obtained a cohomological criterion for Bernstein--Sato roots \cite[Proposition~3.11]{ChenZuo2026Analytic}, which is recovered from \Cref{thm:iterated-residue-obstruction}.

\subsection{Motivic and topological orders} \label{sec:higher-motivic-orders}

For this subsection, assume that $(X,0)$ is a smooth complex algebraic germ of dimension $n$ and that the $f_i$ are regular germs. We use the notation
\[
\Mmot=K_0(\operatorname{Var}_{\C})[\Lef^{-1}]
\]
from \Cref{sec:motivic-comparison}, and write $\mathbf T=(T_1,\ldots,T_r)$ and $\mathbf T^N=\prod_iT_i^{N_i}$. We also use the localization
\[
\mathscr R_r=S_r^{-1}\Mmot[\mathbf T]\subset\Mmot[\![\mathbf T]\!]
\]
introduced there.

The non-equivariant multivariable motivic zeta function has the resolution formula
\begin{equation}\label{eq:higher-motivic-resolution}
Z^{\mathrm{mot}}_{F,0}(\mathbf T)
=
\Lef^{-n}
\sum_{\varnothing\ne I\subseteq J}
[D_I^\circ\cap\pi^{-1}(0)]
\prod_{a\in I}
\frac{(\Lef-1)\Lef^{-\nu_a}\mathbf T^{N_a}}
{1-\Lef^{-\nu_a}\mathbf T^{N_a}}.
\end{equation}
In particular,
\[
Z^{\mathrm{mot}}_{F,0}(\mathbf T)\in\mathscr R_r\subset\Mmot[\![\mathbf T]\!].
\]
Its independence of the resolution follows from the motivic construction and the change-of-variables formula \cite[Theorem~2.2.1]{DenefLoeser1998}, together with the multivariable version \cite[D\'efinition~4.1.1 and Th\'eor\`eme~4.1.2]{Guibert2002}. The factor $\Lef^{-n}$ is an invertible normalization and does not affect the motivic order.

We use the multivariable Nicaise--Xu order of \Cref{def:motivic-NX-order}. The formal results \Cref{lem:motivic-formal-embedding,lem:motivic-finite-witness} are independent of the dimension and apply here without change. We restrict the intrinsic higher-dimensional motivic statements to algebraic germs; the residue results above remain valid in the analytic setting.

For a rational affine hyperplane $H$, let
\begin{equation*}
S_H=\{a\in J:V(N_a\cdot s+\nu_a)=H\}
\end{equation*}
and define the incidence depth
\begin{equation}\label{eq:motivic-incidence-depth}
 c_H(\pi)=\max_{\substack{\varnothing\ne I\subseteq J\\D_I^\circ\cap\pi^{-1}(0)\ne\varnothing}}|I\cap S_H|,
\end{equation}
with value zero when no component defines $H$. Since $D$ is SNC, $c_H(\pi)\leq n$.

\begin{theorem}\label{thm:higher-motivic-order-bounds}
For every rational affine hyperplane $H\subset\C^r$,
\begin{equation}\label{eq:higher-motivic-order-bounds}
 0\leq\ord_HZ^{\mathrm{top}}_{F,0}\leq\ord_H^{\mathrm{mot}}Z^{\mathrm{mot}}_{F,0}\leq c_H(\pi)\leq n.
\end{equation}
The two orders in the middle are intrinsic. The quantity $c_H(\pi)$ may depend on the chosen resolution and gives an upper bound.
\end{theorem}

\begin{proof}
For every nonempty stratum $D_I^\circ\cap\pi^{-1}(0)$, the corresponding term in \eqref{eq:higher-motivic-resolution} contains exactly $|I\cap S_H|$ denominator factors defining $H$. The definition of the motivic order therefore gives $\ord_H^{\mathrm{mot}}Z^{\mathrm{mot}}_{F,0}\leq c_H(\pi)$. Since $D$ has simple normal crossings, at most $n$ components meet at any point, so $c_H(\pi)\leq n$. The nonnegativity of $\ord_HZ^{\mathrm{top}}_{F,0}$ follows from its definition.

It remains to prove
\[
\ord_HZ^{\mathrm{top}}_{F,0}\leq\ord_H^{\mathrm{mot}}Z^{\mathrm{mot}}_{F,0}.
\]
Put $q=\ord_H^{\mathrm{mot}}Z^{\mathrm{mot}}_{F,0}$. By \Cref{lem:motivic-finite-witness}, there is an equality
\[
Z^{\mathrm{mot}}_{F,0}
=
\sum_{\ell=1}^M P_\ell(\mathbf T)
\prod_{(a,b)\in S_\ell}(1-\Lef^a\mathbf T^b)^{-1},
\]
where $P_\ell(\mathbf T)\in\Mmot[\mathbf T]$ and each $S_\ell$ is a finite multiset containing at most $q$ pairs with $H(a,b)=H$, counted with multiplicity. By \Cref{lem:motivic-formal-embedding}, this sum and \eqref{eq:higher-motivic-resolution} are equal in the localization of $\Mmot[\mathbf T]$ obtained by inverting the denominator factors in the two formulas.

Choose an affine linear polynomial $\ell_H$ defining $H$, and put $R_H=\C[s_1,\ldots,s_r]_{(\ell_H)}$. This is a discrete valuation ring with uniformizer $\ell_H$. Apply the diagonal Hodge--Deligne realization $E_\Delta([V])=E_c(V;u,u)$ to the coefficients, followed by the substitutions $u=e^t$ and $T_i=e^{-2ts_i}$. A denominator $1-\Lef^a\mathbf T^b$ then becomes $1-e^{2t\delta}$, where $\delta=a-b\cdot s$. We have
\begin{equation}\label{eq:higher-motivic-exponential-factor}
1-e^{2t\delta}
=
-2t\delta\,U_\delta(t),
\quad
U_\delta(t)=\frac{e^{2t\delta}-1}{2t\delta}
\in R_H[\![t]\!]^\times.
\end{equation}
Indeed, $U_\delta(t)$ has constant term one. Since $\delta$ is a nonzero polynomial, every denominator image is invertible in $\C(s_1,\ldots,s_r)((t))$. Thus the realization and substitutions define a homomorphism $\Phi$ from the localization above to this field.

If $H(a,b)=H$, then $\delta$ is a nonzero constant multiple of $\ell_H$, and \eqref{eq:higher-motivic-exponential-factor} gives
\[
\Phi\bigl((1-\Lef^a\mathbf T^b)^{-1}\bigr)
\in \ell_H^{-1}R_H((t)).
\]
If $H(a,b)\ne H$, then $\delta$ is a unit in $R_H$, so the same inverse belongs to $R_H((t))$. The series $\Phi(P_\ell)$ has coefficients in $\C[s_1,\ldots,s_r]$: each monomial $u^d\mathbf T^\beta$ in the realized numerator becomes $e^{t(d-2\beta\cdot s)}$. Multiplication by the numerator therefore introduces no additional pole along $H$.

Each product contains at most $q$ factors whose pairs define $H$. Since all the products and the sum are finite, we obtain
\[
\Phi\bigl(Z^{\mathrm{mot}}_{F,0}\bigr)
\in \ell_H^{-q}R_H((t)).
\]

We now compute the constant term using \eqref{eq:higher-motivic-resolution}. For each resolution component,
\[
(e^{2t}-1)
\frac{e^{-2t(N_a\cdot s+\nu_a)}}{1-e^{-2t(N_a\cdot s+\nu_a)}}
=
\frac{1}{N_a\cdot s+\nu_a}+O(t).
\]
Every factor in the realized resolution formula has no negative powers of $t$. The constant term of $E_c(D_I^\circ\cap\pi^{-1}(0);e^t,e^t)$ is $\chi(D_I^\circ\cap\pi^{-1}(0))$, and $\Phi(\Lef^{-n})=e^{-2nt}$ has constant term one. Hence the constant term of $\Phi(Z^{\mathrm{mot}}_{F,0})$ is $Z^{\mathrm{top}}_{F,0}$.

It follows that $Z^{\mathrm{top}}_{F,0}\in\ell_H^{-q}R_H$, and therefore
\[
\ord_HZ^{\mathrm{top}}_{F,0}
\leq q
=
\ord_H^{\mathrm{mot}}Z^{\mathrm{mot}}_{F,0}.
\]
This completes the proof.
\end{proof}

\subsection{Maximal-order polar hyperplanes}

Poles of maximal order in the scalar setting are related to multiplicities of roots of the Bernstein--Sato polynomial and maximal monodromy Jordan blocks in \cite{MelleTorrelliVeys2009}. For motivic zeta functions, Nicaise--Xu prove that the negative log canonical threshold is the only possible pole of maximal order \cite{NicaiseXu2016}. The argument below uses only the positivity of the deepest coefficient to show that the corresponding affine hyperplane is contained in the Bernstein--Sato zero locus.

For an affine hyperplane $H$ defined by resolution components, keep $S_H$ from above and choose $L_H$ with constant term one. Then $N_a\cdot s+\nu_a=\nu_aL_H$ for $a\in S_H$. Pole order below means generic order along $H$.

\begin{lemma}
\label{lem:deepest-coefficient}
The coefficient of $L_H^{-n}$ in the resolution expression for $Z^{\mathrm{top}}_{F,0}$ is
\begin{equation}\label{eq:deepest-coefficient}
 C_{-n}(H) = \sum_{I\subseteq S_H,|I|=n}\frac{\#(D_I^\circ\cap\pi^{-1}(0))}{\prod_{a\in I}\nu_a}.
\end{equation}
Every summand is nonnegative.
\end{lemma}

\begin{proof}
Since $D$ has simple normal crossings, every nonempty stratum $D_I^\circ$ satisfies $|I|\leq n$. At the generic point of $H$, the factor $N_a\cdot s+\nu_a$ vanishes precisely when $a\in S_H$. Thus the summand indexed by $I$ has pole order at most $|I\cap S_H|$ along $H$. Only the terms with $I\subseteq S_H$ and $|I|=n$ can therefore contribute to the coefficient of $L_H^{-n}$.

For such an index set $I$, the intersection $\bigcap_{a\in I}D_a$ is zero-dimensional or empty. No other component of $D$ meets this intersection, since otherwise more than $n$ components would pass through the same point. Hence $D_I^\circ=\bigcap_{a\in I}D_a$, and $D_I^\circ\cap\pi^{-1}(0)$ is a finite set of points. Its Euler characteristic is therefore its cardinality.

For every $a\in I$, the identity $N_a\cdot s+\nu_a=\nu_aL_H$ gives
\[
\chi(D_I^\circ\cap\pi^{-1}(0))
\prod_{a\in I}\frac{1}{N_a\cdot s+\nu_a}
=
\frac{\#(D_I^\circ\cap\pi^{-1}(0))}{\prod_{a\in I}\nu_a}L_H^{-n}.
\]
Summing over all $I\subseteq S_H$ with $|I|=n$ proves \eqref{eq:deepest-coefficient}. Every numerator is a nonnegative integer and every $\nu_a$ is positive, so all summands are nonnegative.
\end{proof}

\begin{theorem}
\label{thm:maximal-order-hyperplanes}
Let $\dim X=n$, and let $H$ be an irreducible component of the polar locus of $Z^{\mathrm{top}}_{F,0}$. If the generic pole order along $H$ is $n$, then
\begin{equation*}
H\subseteq Z(B_{F,0}).
\end{equation*}
\end{theorem}

\begin{proof}
Since the pole order along $H$ is $n$, the coefficient $C_{-n}(H)$ is nonzero. Every summand in \eqref{eq:deepest-coefficient} is nonnegative. Hence there is a point $p\in\pi^{-1}(0)$ at which $n$ components $D_{a_1},\ldots,D_{a_n}$ defining $H$ meet.

Let $\alpha\in\Q_{>0}^r$ satisfy $-\alpha\in H$. Then $N_{a_j}\cdot\alpha=\nu_{a_j}$ for every $j$. Choose a nowhere vanishing holomorphic top form $\omega$ near $0\in X$ and SNC coordinates $(x_1,\ldots,x_n)$ centred at $p$, with $D_{a_j}=(x_j=0)$. The actual-branch monodromy around each component is $\exp(-2\pi iN_{a_j}\cdot\alpha)=1$, and hence the dual coefficient system is also trivial on a sufficiently small punctured polydisk around $p$. After choosing a flat frame, \eqref{eq:iterated-leading-form} gives
\[
\pi^*(u_\alpha\omega) = U(x)\frac{dx_1}{x_1}\wedge\cdots\wedge\frac{dx_n}{x_n},
\]
where $U$ is a holomorphic unit.

For sufficiently small positive radii $\varepsilon_j$, let $T_\varepsilon=\{|x_j|=\varepsilon_j,\ 1\leq j\leq n\}$, oriented by the positive circles in the order $(x_1,\ldots,x_n)$. Repeated application of Cauchy's integral formula gives
\[
\frac{1}{(2\pi i)^n}\int_{T_\varepsilon}\pi^*(u_\alpha\omega) = U(0)\ne0.
\]

Suppose that $u_\alpha=P(hu_\alpha)$ for some $P\in\D_{X,0}$. By \Cref{prop:green-formula}, pulling back the Green formula gives
\[
\pi^*(u_\alpha\omega) = \pi^*((hu_\alpha)P^\dagger\omega) + \nabla\pi^*\eta_P(hu_\alpha,\omega).
\]
The form $P^\dagger\omega$ is holomorphic. Thus \eqref{eq:iterated-shifted-order} shows that the first term on the right is holomorphic near $p$, and its integral over $T_\varepsilon$ is zero. The form $\pi^*\eta_P(hu_\alpha,\omega)$ is defined near $T_\varepsilon$. Since the torus has no boundary, the integral of the second term is also zero by Stokes' theorem. This contradicts the nonzero integral above. Therefore $u_\alpha\notin\D_{X,0}(hu_\alpha)$. \Cref{prop:affine-point-lifting} then gives $-\alpha\in Z(B_{F,0})$.

Finally, $H$ is defined by $N_{a_1}\cdot s+\nu_{a_1}=0$, where $N_{a_1}\in\Z_{\geq0}^r\setminus\{0\}$ and $\nu_{a_1}>0$. Hence the set of positive real solutions of $N_{a_1}\cdot\alpha=\nu_{a_1}$ is nonempty and relatively open in the corresponding real affine hyperplane. Its rational points are dense in this open set and therefore Zariski dense in the complex affine hyperplane. Thus the points $-\alpha$ considered above are Zariski dense in $H$. For $r=1$, $H$ consists of the single rational point already obtained. Since $Z(B_{F,0})$ is Zariski closed, we conclude that $H\subseteq Z(B_{F,0})$.
\end{proof}

\begin{remark}
In dimension three, a codimension-two SNC stratum satisfying the corresponding equalities is a curve and its iterated residue is a twisted one-form. Any nonzero period on that curve gives a point of $Z(B_{F,0})$ by \Cref{thm:iterated-residue-obstruction}. Establishing such nonvanishing uniformly is a separate geometric problem.
\end{remark}

\begin{theorem}\label{thm:maximal-motivic-equivalence}
In the algebraic setting of \Cref{thm:higher-motivic-order-bounds}, the following are equivalent for a rational affine hyperplane $H$:
\begin{enumerate}[label=\textup{(\roman*)}]
\item $\ord_H^{\mathrm{mot}}Z^{\mathrm{mot}}_{F,0}=n$;
\item $\ord_HZ^{\mathrm{top}}_{F,0}=n$;
\item some $n$ components defining $H$ have an SNC stratum meeting the local fibre.
\end{enumerate}
\end{theorem}

\begin{proof}
By \Cref{thm:higher-motivic-order-bounds}, (i) implies $c_H(\pi)=n$, which is equivalent to (iii). If (iii) holds, at least one summand in \eqref{eq:deepest-coefficient} is positive, so $C_{-n}(H)>0$ and (ii) follows. Finally, (ii) implies (i) by \Cref{thm:higher-motivic-order-bounds}.
\end{proof}

\begin{corollary}\label{cor:maximal-motivic-smc}
In the same algebraic setting,
\begin{equation*}
\ord_H^{\mathrm{mot}}Z^{\mathrm{mot}}_{F,0}=n
 \quad\Longrightarrow \quad
 H\subseteq Z(B_{F,0}).
\end{equation*}
\end{corollary}

\begin{proof}
By \Cref{thm:maximal-motivic-equivalence}, $H$ is a topological polar hyperplane of generic order $n$. Apply \Cref{thm:maximal-order-hyperplanes}.
\end{proof}

\subsection{Comparison below maximal order}

There is a useful sufficient condition below maximal order. Put $c=c_H(\pi)>0$ and let $L_H$ be the normalized affine equation of $H$ with constant term one. Then $N_a\cdot s+\nu_a=\nu_aL_H$ for every $a\in S_H$. Define
\begin{equation}\label{eq:motivic-leading-incidence-coefficient}
 R^{\mathrm{top}}_{H,c}=\sum_{\substack{D_I^\circ\cap\pi^{-1}(0)\ne\varnothing\\|I\cap S_H|=c}}\frac{\chi(D_I^\circ\cap\pi^{-1}(0))}{\prod_{a\in I\cap S_H}\nu_a}\prod_{b\in I\setminus S_H}\frac1{(N_b\cdot s+\nu_b)|_H}\in K(H).
\end{equation}
Here $K(H)$ is the function field of $H$, and every denominator in the presentation that does not define $H$ is restricted at the generic point of $H$; it may vanish on a proper subvariety without affecting this generic coefficient.

\begin{proposition}\label{prop:motivic-leading-incidence}
If $R^{\mathrm{top}}_{H,c}\ne0$, then
\begin{equation*}
\ord_HZ^{\mathrm{top}}_{F,0}=\ord_H^{\mathrm{mot}}Z^{\mathrm{mot}}_{F,0}=c_H(\pi).
\end{equation*}
\end{proposition}

\begin{proof}
By \eqref{eq:motivic-leading-incidence-coefficient}, the coefficient of $L_H^{-c}$ in the topological resolution formula is $R^{\mathrm{top}}_{H,c}$. Since this coefficient is nonzero, $\ord_HZ^{\mathrm{top}}_{F,0}\geq c$. Together with \Cref{thm:higher-motivic-order-bounds}, this gives
\[
c\leq\ord_HZ^{\mathrm{top}}_{F,0}
\leq\ord_H^{\mathrm{mot}}Z^{\mathrm{mot}}_{F,0}
\leq c.
\]
Thus both orders equal $c=c_H(\pi)$.
\end{proof}

The equality of motivic and topological orders can fail in dimension three.

\begin{example}\label{ex:motivic-cubic-cone-gap}
Let
\[
f=x_1^3+x_2^3+x_3^3
\quad\text{on }(\mathbb C^3,0).
\]
Blowing up the origin gives an exceptional divisor $E\simeq\mathbb P^2$ and a strict transform $S$ meeting transversely along a smooth plane cubic $C$. Their numerical data are
\[
(N_E,\nu_E)=(3,3),
\quad
(N_S,\nu_S)=(1,1),
\]
so both components determine the candidate value $s=-1$. Set
\[
Q_E=\frac{(\Lef-1)\Lef^{-3}T^3}{1-\Lef^{-3}T^3},
\quad
Q_S=\frac{(\Lef-1)\Lef^{-1}T}{1-\Lef^{-1}T}.
\]
The motivic resolution formula is
\begin{equation}\label{eq:motivic-cubic-expression}
Z^{\mathrm{mot}}_{f,0}
=
\Lef^{-3}\bigl([\mathbb P^2\setminus C]Q_E+[C]Q_EQ_S\bigr).
\end{equation}
In particular, the motivic order at $s=-1$ is at most two.

Apply the Hodge--Deligne realization and set $q=uv$ and $T=qw$. Denote the resulting rational function by $Z(w)\in\Q(u,v)(w)$. A denominator $1-\Lef^aT^b$ becomes $1-q^{a+b}w^b$. It has a simple zero at $w=1$ when $a/b=-1$ and is nonzero there otherwise. Thus the pole order of $Z$ at $w=1$ is a lower bound for the motivic order at $s=-1$. Since $E_c(C;u,v)=1-u-v+uv$, formula \eqref{eq:motivic-cubic-expression} gives
\[
\lim_{w\to1}(1-w)^2 Z(w) = \frac{q^{-3}(1-u-v+uv)(q-1)^2}{3}\ne0.
\]
Hence $\ord_{-1}^{\mathrm{mot}}Z^{\mathrm{mot}}_{f,0}=2$.

On the other hand, $\chi(C)=0$ and $\chi(\mathbb P^2\setminus C)=3$, so
\[
Z^{\mathrm{top}}_{f,0}(s)=\frac{3}{3s+3}=\frac{1}{s+1}.
\]
Therefore
\[
    \ord_{-1}Z^{\mathrm{top}}_{f,0} = 1 < 2 = \ord_{-1}^{\mathrm{mot}}Z^{\mathrm{mot}}_{f,0}.
\]
\end{example}

The polar loci can also differ.

\begin{example}\label{ex:motivic-fivefold-support-gap}
Let
\[
g=x_1^3+x_2^3+x_3^3+x_4^3+x_5^6
\quad\text{on }(\mathbb C^5,0).
\]
Blow up the origin and then the unique singular point of the first strict transform. Denote the final exceptional divisors by $A,B$ in this order, and the strict transform of $(g=0)$ by $C$. Their numerical data are
\[
(N_A,\nu_A)=(3,5),
\quad
(N_B,\nu_B)=(6,9),
\quad
(N_C,\nu_C)=(1,1).
\]
Inside $B\simeq\mathbb P^4$, the intersections $P=A\cap B$ and $V=B\cap C$ are the hyperplane $T=0$ and the smooth cubic threefold
\[
Y_1^3+Y_2^3+Y_3^3+Y_4^3+T^3=0.
\]
Their intersection $S=P\cap V$ is a smooth cubic surface. These intersections are transverse, and the reduced total transform has simple normal crossings.

Only $B$ has candidate value $s=-3/2$. The Euler characteristics of $B^\circ$, $P\setminus S$, $V\setminus S$, and $S$ are $16,-5,-15,9$, respectively. Thus the sum of the topological terms containing $B$ is
\[
\frac{1}{6s+9}
\left(
16-\frac{5}{3s+5}-\frac{15}{s+1}
+\frac{9}{(3s+5)(s+1)}
\right).
\]
The expression in parentheses vanishes at $s=-3/2$. All other terms are regular there, so $-3/2$ is not a topological pole.

Apply the diagonal Hodge--Deligne realization $E_\Delta([W])=E_c(W;z,z)$ and set $T=z^3w$. Let $Z(w)\in\Q(z)(w)$ be the resulting rational function. A denominator $1-\Lef^aT^b$ becomes $1-z^{2a+3b}w^b$, which has a simple zero at $w=1$ exactly when $a/b=-3/2$. As in the preceding example, the pole order at $w=1$ is a lower bound for the motivic order at $s=-3/2$.

The realized factors for $A$ and $C$ have values $z+1$ and $-z(z+1)$ at $w=1$, respectively. Using
\[
\begin{aligned}
E_\Delta(\mathbb P^4)&=1+z^2+z^4+z^6+z^8,\\
E_\Delta(P)&=1+z^2+z^4+z^6,\\
E_\Delta(V)&=1+z^2-10z^3+z^4+z^6,\\
E_\Delta(S)&=1+7z^2+z^4,
\end{aligned}
\]
the four terms containing $B$ give
\[
\lim_{w\to1}(1-w)Z(w)
=
-\frac{z^{-9}(z^2-1)(z-1)^2(z^2+z+1)(z^2+3z+1)}{6}
\ne0.
\]
Thus the motivic order at $s=-3/2$ is at least one. Since only $B$ has this candidate value, every term in the resolution formula contains at most one corresponding denominator. Consequently,
\[
\ord_{-3/2}^{\mathrm{mot}}Z^{\mathrm{mot}}_{g,0}=1,
\quad
\ord_{-3/2}Z^{\mathrm{top}}_{g,0}=0.
\]
\end{example}

\begin{remark}
Viewing the first example as a germ on $\mathbb C^{3+d}$ by adjoining $d\geq0$ variables on which $f$ does not depend gives the same order gap in every dimension at least three. Similarly, the second example gives a difference of polar supports in every dimension at least five. In both cases, the product resolution has the same local-fibre strata and numerical data, while the motivic zeta function changes only by the invertible factor $\Lef^{-d}$.
\end{remark}

\section{Summary and further questions} \label{sec:further-questions}

For tuples of arbitrary nonzero nonunit plane curve germs, we have proved the Topological Multivariable Strong Monodromy Conjecture and the equality of the non-equivariant motivic and topological generic pole orders. Their common polar locus is therefore contained in $Z(B_{F,0})$. For algebraic germs of dimension $n$, we have
\[
\ord_H Z^{\mathrm{top}}_{F,0}\leq\ord_H^{\mathrm{mot}}Z^{\mathrm{mot}}_{F,0}\leq n.
\]
One order equals $n$ if and only if the other does, and in this case $H\subseteq Z(B_{F,0})$. However, the two orders and even the two polar loci can differ in higher dimension, as shown by \Cref{ex:motivic-cubic-cone-gap,ex:motivic-fivefold-support-gap}.

For a fixed finite expression, one can still choose a positive integral specialization avoiding the finitely many denominator coincidences. This does not extend the plane comparison, since the required equality already fails for one-variable functions in higher dimension. For plane curves, an exceptional component $E$ satisfies $E\simeq\mathbb P^1$ and $[E^\circ]=\Lef+1-d_E$. Together with the principal-divisor relation and adjunction, these identities give the factorizations in \Cref{thm:motivic-nonrupture-cancellation} when $d_E=1$ or $2$. The propagation along chains of valency-two components also uses the tree structure of the resolution graph. In higher dimension, the stratum classes are not determined by a single valency, and the dual complex does not give the same propagation argument.

The iterated residue criterion uses a specific cohomology class rather than only the polar locus. Under the hypotheses of \Cref{thm:iterated-residue-obstruction}, in arbitrary dimension,
\[
    [\Res_I\pi^*(F^{-\alpha}\omega)] \ne0 \quad \Longrightarrow \quad F^{-\alpha}\notin\D_{X,0}(hF^{-\alpha}) \quad \Longrightarrow \quad -\alpha\in Z(B_{F,0}).
\]
When the pole order along $H$ is $n$, the positivity in \eqref{eq:deepest-coefficient} gives a zero-dimensional stratum where the torus residue is nonzero for every $\alpha\in\Q_{>0}^r$ with $-\alpha\in H$.

Below maximal order, positive-dimensional strata may contribute, and their contributions may cancel after applying Euler characteristic even when the corresponding Grothendieck or Hodge classes are nonzero. A nonzero iterated residue form may also represent the zero cohomology class, even when the corresponding hyperplane belongs to one or both polar loci. One would therefore like geometric conditions ensuring that such a residue class is nonzero, or that a lower-order motivic pole remains a pole after Euler-characteristic specialization. Uniform criteria for one-dimensional SNC strata on threefolds, toric resolutions, and Newton-nondegenerate tuples would require further geometric analysis.

Further questions concern the extension to the $\widehat\mu$-equivariant motivic zeta function, the scheme-theoretic multiplicities of the Bernstein--Sato locus, and a resolution formula for the entire locus $Z(B_{F,0})$. The full Strong Monodromy Conjecture in higher dimension is a separate problem. The failure of the motivic--topological comparison does not exclude strong monodromy results, which are known for arbitrary factorizations of hyperplane arrangements \cite{DavisYang2026} and for a class of homogeneous polynomials in three variables \cite{BathVeys2026}.

In one variable, the Bernstein--Sato polynomial can vary in equisingular families \cite[Example~1]{Oaku2022}. Thus a resolution graph, even with its numerical data, does not determine the entire Bernstein--Sato ideal. The relation between topological polar hyperplanes and components of the Bernstein--Sato zero locus that vary with the analytic structure remains to be understood.


\bibliographystyle{amsplain}
\bibliography{references}

\end{document}